\documentclass[english,notitlepage]{article}
\usepackage[latin9]{inputenc}
\usepackage[a4paper]{geometry}
\usepackage{hyperref}
\usepackage{mathtools}
\usepackage{amssymb,amsmath,amsthm,amscd}
\usepackage{mathrsfs}
\usepackage{stackrel}
\usepackage{tikz}
\usetikzlibrary{positioning}
\usepackage{graphics}
\usepackage{amssymb} 
\usepackage{babel}
\usepackage{amstext}
\usepackage{amscd}   
\usepackage{verbatim}
\usepackage{enumitem}
\usepackage{xcolor}

\makeatletter
\numberwithin{equation}{section}
\numberwithin{figure}{section}

 \theoremstyle{plain}
 \newtheorem{thm}{Theorem}
 \theoremstyle{plain}
 \newtheorem{thm2}{Theorem}
 \theoremstyle{plain}
 \newtheorem*{thm*}{Theorem}
  \theoremstyle{plain}
  \numberwithin{thm}{section}
  \newtheorem{cor}[thm]{Corollary}
  \theoremstyle{plain}
  \newtheorem{lem}[thm]{Lemma}
  \theoremstyle{remark}
  \newtheorem{rem}[thm]{Remark}
    \theoremstyle{remark}
    \newtheorem{rem2}{Remark}
    \theoremstyle{remark}
  \newtheorem{example}[thm]{Example}
    
   \theoremstyle{plain}
  \newtheorem{prop}[thm]{Proposition}
     \theoremstyle{plain}
  \newtheorem{definition}[thm]{Definition}
   \theoremstyle{plain}
  \newtheorem{assumption}{Assumption}
   \theoremstyle{plain}
     \newtheorem{assumption2}{Assumption}
   \theoremstyle{plain}

  \def\Ddots{\mathinner{\mkern1mu\raise\p@
\vbox{\kern7\p@\hbox{.}}\mkern2mu
\raise4\p@\hbox{.}\mkern2mu\raise7\p@\hbox{.}\mkern1mu}}
\makeatother

\newcommand{\norm}[1]{\left\| #1 \right\|}

\newcommand{\eklm}[1]{\left\langle #1 \right\rangle}

\renewcommand{\d}{\,d}

\newcommand{\N}{{\mathbb N}}
\newcommand{\Z}{{\mathbb Z}}
\newcommand{\C}{{\mathbb C}}

\newcommand{\R}{{\mathbb R}}

\newcommand{\B}{{\mathcal B}}
\newcommand{\D}{{\mathcal D}}
\newcommand{\E}{{\mathcal E}}

\newcommand{\M}{{\mathcal M}}

\newcommand{\X}{{\mathfrak X}}
\newcommand{\Xbf}{{\mathbf X}}

\newcommand{\V}{{\mathcal V}}
\newcommand{\W}{{\mathcal W}}

\newcommand{\0}{{\rm 0}}

\renewcommand{\epsilon}{\varepsilon}

\renewcommand{\rho}{\varrho}

\newcommand{\bdm}{\begin{displaymath}}
\newcommand{\edm}{\end{displaymath}}
\newcommand{\bq}{\begin{equation}}
\newcommand{\eq}{\end{equation}}
\newcommand{\bqn}{\begin{equation*}}
\newcommand{\eqn}{\end{equation*}}

\newcommand{\Cinft}{{\rm C^{\infty}}}
\newcommand{\CT}{{\rm C^{\infty}_c}}

\renewcommand{\L}{{\rm L}}

\newcommand{\GL}{\mathrm{GL}}
\newcommand{\SL}{\mathrm{SL}}
\newcommand{\SO}{\mathrm{SO}}
\newcommand{\SU}{\mathrm{SU}}

\newcommand{\g}{{\bf \mathfrak g}}

\newcommand{\aL}{{\bf \mathfrak a}}
\newcommand{\nL}{{\bf \mathfrak n}}
\renewcommand{\k}{{\bf \mathfrak k}}
\newcommand{\m}{{\bf \mathfrak m}}
\renewcommand{\t}{{\bf \mathfrak t}}
\newcommand{\p}{{\bf \mathfrak p}}

\newcommand{\Ad}{\mathrm{Ad}}

\renewcommand{\Im}{\mathrm{Im}\,}
\renewcommand{\Re}{\mathrm{Re}\,}
\newcommand{\vol}{\mathrm{vol}\,}

\newcommand{\Gam}{{_\Gamma}}
\newcommand{\gam}{\Gamma\backslash}

\DeclareMathOperator{\supp}{supp\,}
\DeclareMathOperator{\Res}{Res}

\newcommand{\longrightleftarrows}{\mathrel{\substack{\longrightarrow \\[-.6ex] \longleftarrow}}}

\newcommand\longmapsfrom{\mathrel{\reflectbox{\ensuremath{\longmapsto}}}}
\newcommand{\hooklongrightarrow}{\lhook\joinrel\longrightarrow}
\definecolor{darkblue}{rgb}{0,0,0.4}

\newcommand{\tu}[1]{\textup{#1}}
\newcommand{\Abb}[4]{\left\{ \begin{array}{ccc}
                               #1 & \rightarrow &#2\\
			       #3 &\mapsto &#4
                               \end{array}\right.}

\begin{document}
\title{Quantum-classical correspondence on associated\\ vector bundles over locally symmetric spaces}
\author{Benjamin K\"uster\hspace{5em}Tobias Weich\\\vspace*{-0.75em}
\hspace{-2em}\texttt{\scriptsize bkuester@mathematik.uni-marburg.de}\hspace{3.5em}\texttt{\scriptsize weich@math.upb.de}\vspace*{1.5em}}
\date{\today}
\maketitle
{\bf Abstract.}\quad For a compact Riemannian locally symmetric space $\mathcal M$ of rank one and an associated vector bundle $\mathbf V_\tau$ over the unit cosphere bundle $S^\ast\mathcal M$, we give a precise description of those classical (Pollicott-Ruelle) resonant states on $\mathbf V_\tau$ that vanish under covariant derivatives in the Anosov-unstable directions of the chaotic geodesic flow on $S^\ast\mathcal M$. In particular, we show that they are isomorphically mapped by natural pushforwards into generalized common eigenspaces of the algebra of invariant differential operators $D(G,\sigma)$ on compatible associated vector bundles $\mathbf W_\sigma$ over $\mathcal M$. As a consequence of this description, we obtain an 
exact band structure of the Pollicott-Ruelle spectrum. Further, under some mild assumptions on the representations $\tau$ and $\sigma$ defining the bundles $\mathbf V_\tau$ and $\mathbf W_\sigma$, we obtain a very explicit description of the generalized common eigenspaces. This allows us to relate classical Pollicott-Ruelle resonances to quantum  eigenvalues of a Laplacian in a suitable Hilbert space of sections of $\mathbf W_\sigma$. Our methods of proof are based on representation theory and Lie theory.
\tableofcontents
\newpage
\section*{Introduction}
\addcontentsline{toc}{section}{Introduction}
Let $\mathcal{M}$ be a compact Riemannian manifold without boundary. The Riemannian metric allows to define two natural differential operators of fundamental importance: The Laplace-Beltrami operator $\Delta:\Cinft(\mathcal{M})\to \Cinft(\mathcal{M})$ as well as the generator $X:\Cinft(S^\ast\M)\to \Cinft(S^\ast\M)$ of the geodesic flow $\varphi_t$ on the unit cosphere bundle $S^\ast\M$. As $\mathcal M$ is compact, 
$\Delta$ extends to a self-adjoint unbounded operator in $\L^2(\mathcal M)$ with purely discrete spectrum of finite multiplicities. It is a broad concept of spectral geometry to relate the spectrum of $\Delta$ to properties of the geodesic flow. As in physics, the presence of such relations is often called \emph{quantum-classical correspondence}.\footnote{In the language of physics, the geodesic flow $\varphi_t:S^\ast\M\to S^\ast\M$ generated by $X$ describes the classical mechanical movement of a single free point particle on $\M$ with constant speed, while  $\Delta$ is, up to rescaling, the Schr\"odinger operator associated to the corresponding quantum mechanical system.} Seminal results like Selberg's trace formula \cite{Sel56}, or more generally Gutzwiller's trace formula \cite{Gut71,CdV73, DG75}, relate the spectrum of the Laplacian to periodic geodesics. However, these results do not link the spectrum of $\Delta$ to any spectrum of $X$, nor do they give a relation between the corresponding eigenstates. It has only been recently  that Dyatlov, Faure, and Guillarmou \cite{dfg} have shown that in the case of strictly negative constant curvature an explicit relation between the eigenvalues of $\Delta$ and the so-called spectrum of Pollicott-Ruelle resonances of the geodesic flow holds and that there is even an explicit relation between the associated eigenstates. In order to state their result, let us recall the definition of Pollicott-Ruelle resonances. 

The Pollicott-Ruelle resonances form a discrete set of numbers in the complex plane that can be associated to a hyperbolic flow. Historically, Pollicott-Ruelle resonances are defined as poles of the meromorphically continued Laplace-transformed correlation function \cite{Pol85, Rue86,Rue87}. From a modern perspective, these resonances can be regarded as the discrete spectrum of $-X$ on certain anisotropic distribution spaces \cite{Liv04, BL07, FSj11, DZ16, DG16}. From this point of view, a Pollicott-Ruelle resonance $\lambda\in \C$ comes with a finite-dimensional eigenspace of \emph{Pollicott-Ruelle resonant states} $\Res(\lambda)\subset \mathcal D'(S^*\mathcal M)$ (see Section \ref{sec:defruelle} for a precise definition).

Note that a crucial assumption for the existence of a discrete spectrum of the non-elliptic operator $-X$ is that the corresponding flow $\varphi_t$ is Anosov and, in particular, there is a flow-invariant continuous splitting of the tangent bundle $T(S^\ast\mathcal M)=E_0\oplus E_s\oplus E_u$ into neutral, stable, and unstable subbundles. In general, the subbundles $E_s$ and $E_u$ might be only H\"older-continuous but for $\mathcal M$ of strictly negative constant  curvature (or, more generally, for $\mathcal M$ a Riemannian locally symmetric space of rank one) they are in fact smooth subbundles. Let us, for the moment, assume that $\mathcal M$ has strictly negative constant curvature. A central role in the result of Dyatlov, Faure, and Guillarmou is played by those resonant states that are invariant along the unstable directions.
We write 	
\begin{equation}
 \label{eq:first_band_res_state}
 \textup{Res}^0(\lambda):=\{u\in \textup{Res}(\lambda):\mathfrak Xu=0 ~\forall~ \mathfrak X\in\Gamma^ \infty(E_u)\}
\end{equation}
and call the elements of $\textup{Res}^0(\lambda)$ \emph{first band resonant states}. We call $\lambda$ a \emph{first band resonance}\footnote{The name ``first band'' comes from the observation due to Dyatlov, Faure, and Guillarmou that the Pollicott-Ruelle spectrum on surfaces of strictly negative constant curvature has an exact structure of bands parallel to the imaginary axis. The first band resonances turn out to be those lying closest to the imaginary axis.} if $\textup{Res}^0(\lambda)\neq \{0\}$.

Now, let $\pi:S^\ast\mathcal M\to\mathcal M$ be the bundle projection and recall that the pushforward
of distributions along this submersion is a well-defined continuous operator $\pi_\ast:\mathcal D'(S^\ast\mathcal M)\to \mathcal D'(\mathcal M)$. Dyatlov,  Faure, and Guillarmou then prove the following remarkable result \cite{dfg}: If $\lambda$ is a first band Pollicott-Ruelle resonance, $\pi_\ast$ restricts to
\[
 \pi_\ast: \textup{Res}^0(\lambda) \to \ker_{\L^2(\mathcal M)}(\Delta-\mu(\lambda)),
\]
where $\mu(\lambda) = -\left(\lambda+\rho\right)^2 + \rho^2$ and $\rho=(\dim(\mathcal M)-1)/2$. Furthermore, provided one has $\lambda \notin -\rho-\frac{1}{2}\N_0$, the above map is an isomorphism of complex vector spaces. 

The aim of the present article is to extend these results to the setting of associated vector bundles over arbitrary compact Riemannian locally symmetric spaces of rank one. From now on, suppose that $\mathcal M$ is such a space. Let us briefly introduce the setting: We can write our manifold as a double quotient $\mathcal M=\Gamma \backslash G/K$, where $G$ is a connected real semisimple Lie group of real rank one with finite center, $K\subset G$ a maximally compact subgroup, and 
$\Gamma\subset G$ a discrete torsion-free cocompact subgroup. The Riemannian metric on $\mathcal M$ is defined in a canonical way from the non-degenerate Killing form on the Lie algebra $\mathfrak g$ of $G$. This metric has strictly negative sectional curvature\footnote{The case of manifolds with strictly negative constant curvature, studied in \cite{dfg}, is the special case of $G=\SO(n+1,1)_0$, $n\in \N$, where the subscript $0$ denotes the identity component.}, so the geodesic flow on $S^\ast \mathcal M$ is Anosov. We can also write the cosphere bundle as a double quotient in the following way: Let $G=KAN$ be an Iwasawa decomposition and define $M\subset K$ to be the centralizer of $A$ in $K$, then $S^\ast \mathcal M = \Gamma\backslash G/M$. Given the structure as double quotients, there is the following canonical way to construct vector bundles over $\mathcal M$ and $S^\ast\mathcal M$: Let $\sigma:K\to \mathrm{End}(W)$ and $\tau:M\to \mathrm{End}(V)$ be finite-dimensional unitary representations. Then we define the associated vector bundles
\[
 \mathbf V_\tau:= \gam G\times_\tau V,\qquad\quad\qquad\mathbf W_\sigma:= \gam G\times_\sigma W
\]
over $S^\ast\mathcal M$ and $\mathcal M$, respectively. These bundles come with canonical connections. Regarding $\mathbf W_\sigma$, this allows to define the Bochner Laplacian $\Delta$ on sections of $\mathbf W_\sigma$. Due to its ellipticity, it has a purely discrete eigenvalue spectrum with finite multiplicities in $\L^2(\mathcal M, \mathbf W_\sigma)$. Regarding $\mathbf V_\tau$, the connection allows to lift the generator $X$ of the geodesic flow to sections of $\mathbf V_\tau$ as a covariant derivative (see Definition \ref{def:liftedgeodvf})\footnote{In the entire introduction we use a simplified notation that is not necessarily consistent with the notation in the following sections.}, and according to \cite{DG16} there is a notion of discrete Pollicott-Ruelle spectrum for this lifted operator. We denote the space of Pollicott-Ruelle resonant states of a resonance $\lambda\in \C$ by $\textup{Res}_{\mathbf V_\tau}(\lambda)\subset\mathcal D'(S^\ast\mathcal M, \mathbf V_\tau)$ (see Section \ref{sec:defruelle} for a precise definition). Using the covariant derivative into the unstable directions we can define analogously to (\ref{eq:first_band_res_state}) the space of first band resonant states on $\mathbf V_\tau$
\[
 \textup{Res}_{\mathbf V_\tau}^0(\lambda):=\{u\in \textup{Res}_{\mathbf V_\tau}(\lambda): \nabla_{\mathfrak X}u=0 ~\forall~ \mathfrak X\in\Gamma^ \infty(E_u)\}.
\]
Note that, contrary to the scalar case, the projection $\pi:S^\ast\mathcal M\to\mathcal M$ is in general not sufficient to define a pushforward of distributions $\mathcal D'(S^\ast\mathcal M, \mathbf V_\tau)\to \mathcal D'(\mathcal M, \mathbf W_\sigma)$ as such a 
construction requires passing from $V$ to $W$ via an $M$-equivariant homomorphism $h:V\to W$.
We call $\tau$ and $\sigma$ \emph{compatible} if there is at least one such $h$ that is non-zero. In this case one can define for each non-trivial $h\in \textup{Hom}_M(V,W)$ a non-trivial natural pushforward
\begin{equation}
 \label{eq:push_forward}
 h_{\pi*}:\mathcal D'(S^\ast\mathcal M, \mathbf V_\tau)\to  \mathcal D'(\mathcal M, \mathbf W_\sigma),
\end{equation}
see Definition \ref{def:pushforwards} for details. Note that if the representation $\tau$ or $\sigma$ is reducible, the bundle $\mathbf V_\tau$ or $\mathbf W_\sigma$ splits into a direct sum of associated subbundles, and we can treat pairs of subbundles of $\mathbf V_\tau$ and $\mathbf W_\sigma$ arising this way separately. Therefore, it is reasonable to assume for the rest of this introduction that $\tau$ and $\sigma$ are irreducible. 
By basic representation theory, $\tau$ and $\sigma$ are then compatible in the above sense iff the restriction of $\sigma$ from $K$ to $M$ contains a subrepresentation equivalent to $\tau$. In order to state the first result, let us introduce for $\lambda \in \C$
\[
\mu(\lambda) := -(\lambda+\norm{\rho})^2 + \norm{\rho}^2 +\norm{i\omega_{[\sigma]}+i\delta_\k}^2-\big\Vert i\omega_{[\tau_1^0]}+i\delta_\m\big\Vert^2+\norm{i\delta_\m}^2-\norm{i\delta_\k}^2.
\]
Here $\rho\in \g^\ast$ and $\delta_\k,\delta_\m\in i\g^\ast\subset \g^\ast\otimes_\R\C$ are sums of particular root systems defined in  \eqref{eq:rho} and \eqref{eq:rhom}, respectively, and $\omega_{[\sigma]},\omega_{[\tau_1^0]}\in i\g^\ast$ are the highest weights associated to the equivalence classes of irreducible unitary representations $[\sigma]\in \hat K$ and $[\tau_1^0]\in \hat M_0$ as introduced in Section \ref{sec:invariantsweights}. Finally, $\norm{\cdot}$ denotes the norm on $\mathfrak g^\ast$ that is obtained by the Killing form and Cartan involution (see Section~\ref{sec:setupnot}).

As a first result we obtain
\begin{thm2}\label{thm:QC_int}
 For every $\lambda \in \C$, each non-trivial pushforward $h_{\pi^*}$ as in \eqref{eq:push_forward} restricts to a map
 \bq
  h_{\pi*}: \Res^0_{\mathbf V_\tau}(\lambda) \longrightarrow \ker_{\L^2(\mathbf W_\sigma)}(\Delta - \mu(\lambda))\label{eq:restrictedhpi}
 \eq
and the map is injective unless $\lambda$ lies in a certain discrete set of points on the real axis.
Moreover, for $\lambda$ outside this discrete set, the image of the map \eqref{eq:restrictedhpi} is the joint eigenspace of the algebra of invariant differential operators with respect to a 
finite-dimensional representation $\chi_{\tau,\sigma}$ (see Definition \ref{def:eigencomnot} and Lemma \ref{lem:smbaction} for a precise definition
of these joint eigenspaces and Theorem \ref{thm:main} for a slightly more precise bijectivity statement).
\end{thm2}
As a consequence of this correspondence between first band Pollicott-Ruelle
resonances and Laplace eigenvalues, we obtain the following global band structure of the whole Pollicott-Ruelle resonance spectrum.
\begin{thm2}\label{thm:band_int}
If $\lambda\in \C$  is a Pollicott-Ruelle resonance on  $\mathbf V_\tau$, then either $\Im(\lambda) = 0$ or 
$\Re(\lambda) \in - \|\rho\| - \N_0\|\alpha_0\|$. Furthermore, if 
$s\in \tu{Res}_{\mathbf V_\tau}(-\|\rho\| + ir)$, $r\in \R\setminus\{0\}$, then 
for all $\X \in \Gamma^\infty(\gam E_u)$ we have $\X s=0$, i.e.,
\bqn
\tu{Res}_{\mathbf V_\tau}(-\|\rho\| + ir) = \tu{Res}^0_{\mathbf V_\tau}(-\|\rho\| + ir).
\eqn
\end{thm2}
\begin{rem2} An exact band structure has before been obtained in \cite{dfg} in the scalar case for manifolds of constant negative curvature. We would like to emphasize that
even the step to generalize the band structure for scalar resonances from constant curvature manifolds to general rank one manifolds heavily relies on the quantum-classical correspondence on vector bundles given in Theorem \ref{thm:QC_int}.
\end{rem2}

While Theorem~\ref{thm:QC_int} is sufficient to prove the global band structure,  
a drawback of this general result is that the image of $h_{\pi*}$ is given in a 
rather abstract sense as a generalized joint eigenspace of the algebra of invariant differential
operators on $\mathbf W_\sigma$. Contrary to the scalar case, this algebra can be quite complicated (in particular non-commutative) in the vector-valued case. However, under two mild assumptions, we are able to derive a rather simple form of this generalized joint eigenspace.
\begin{assumption}
 $\tau$ occurs in the restriction of $\sigma$ from $K$ to $M$ with multiplicity one.
\end{assumption}
This assumption is equivalent to $\dim_\C\textup{Hom}_M(V,W) = 1$, so the linear pushforward map (\ref{eq:push_forward}) is unique up to scalar multiplication. To state the second assumption, let $[\tau]$ denote the equivalence class of $\tau$ with respect to unitary equivalence of irreducible $M$-representations.
\begin{assumption} $[\tau]$ is invariant under the Weyl group action.\footnote{For details on the Weyl group and the considered action, see Section \ref{sec:setupnot}.}
\end{assumption}
\begin{rem2}
 Both assumptions above can be explicitly checked. Furthermore, they are rather mild assumptions. For example, they are fulfilled in the families of even-dimensional real hyperbolic symmetric spaces ($G=\SO(2n,1)_0$) or complex hyperbolic symmetric spaces ($G=\SU(n+1,1)$) for all pairs of compatible irreducible representations $\tau,\sigma$.
\end{rem2}
\begin{thm2}
Let $\mathcal M$ be a compact Riemannian locally symmetric space of rank one. Let $\tau$ and $\sigma$ be compatible irreducible unitary representations of $M$ and $K$, respectively, fulfilling Assumptions 1 and 2. Let $\mathrm{End}_M(W)$ denote the space of $M$-equivariant endomorphisms on $W$ and set $N:=\dim_\C \mathrm{End}_M(W)-1$. Then there are differential operators $D_1,\ldots, D_N$ on $\mathbf W_\sigma$, commuting with $\Delta$, which allow to define the Hilbert spaces $\mathcal H_\sigma:=\{s\in L^2(\mathcal M, \mathbf W_\sigma):D_js=0\;\forall\,j\}$, and for every first band Pollicott-Ruelle resonance $\lambda\in \C$, the unique\footnote{More precisely, the linear map $h_{\pi*}$ is unique up to non-zero scalar multiples, which implies that its image is unique.} natural pushforward map $h_{\pi*}$ restricts to a map
\[
h_{\pi*}:\mathrm{Res}^0_{\mathbf V_{\tau}}(\lambda)\longrightarrow \ker_{\mathcal H_\sigma}(\Delta-\mu(\lambda)).
\]
Moreover, $h_{\pi*}$ is an isomorphism of complex vector spaces unless $\lambda$ lies in a discrete subset of $\R$ of exceptional points (cf.\ Theorem \ref{thm:rough}).
\end{thm2}
Apart from the above mentioned results, we give a precise description of Jordan blocks (Theorem \ref{thm:jordan}). In Appendix~\ref{app:Bochner_eigenspaces} and Appendix~\ref{app:joint_eigenspace} we show for two examples how to apply the general results to concrete examples. In particular, we show in Appendix~\ref{app:joint_eigenspace} how to determine the joint eigenspaces on some particular associated vector bundles over $\mathbb H^3=\SO(3,1)_0/\SO(3)$. This allows us to recover the results obtained in \cite{dfg} and even give a slightly more precise description.  

Let us give a short overview of related work: The first explicit relation between first band Pollicott-Ruelle resonant states and Laplace eigenstates was obtained in \cite{dfg} in the setting of compact manifolds with constant negative curvature. On a purely spectral level (without considering the resonant states), it had before been announced in \cite{FT13} for the case of compact hyperbolic surfaces. In fact, the result for the first band on compact hyperbolic surfaces is already implicitly contained in a work on horocyclic invariant distributions of Flaminio and Forni \cite{FF03}.

While the result of \cite{dfg} on compact manifolds of strictly negative constant curvature in arbitrary dimension had to exclude certain points due to the non-bijectivity of the Poisson transform, the article \cite{GHW18a} treats also these exceptional points in the case of hyperbolic surfaces, relating their multiplicities to the Euler characteristic of the surface.

Another generalization of \cite{dfg} is the passage from compact to convex cocompact manifolds. In \cite{GHW18a}
the full Pollicott-Ruelle spectrum on convex cocompact hyperbolic surfaces is studied. A generalization to 
convex cocompact hyperbolic spaces of arbitrary dimension has been obtained (up to the exceptional points) in \cite{Had17a, Had18}.

Finally, a first generalization to general compact Riemannian locally symmetric spaces of rank one has been obtained by Guillarmou, Hilgert, and the second author in \cite{GHW18b}. There, the first band of scalar Pollicott-Ruelle resonances and their relation to Patterson-Sullivan distributions are studied.

A band structure of the full Pollicott-Ruelle resonance spectrum has before been obtained on constant negative curvature manifolds in the scalar case \cite{dfg}. However, \cite{dfg} gives also a more precise description of the higher order bands by inverting the horocycle operators which we do not cover in the present article.

Let us finally give an outline of this article and a rough sketch of the proof:

In Section~\ref{sec:prelim} we introduce the basic definitions regarding  Pollicott-Ruelle resonances, Riemannian symmetric spaces, and associated vector bundles. We also recall some basic facts from the structure theory of semisimple Lie groups and Riemannian symmetric spaces which we heavily use in the proofs. In Section \ref{sec:res_on_VB} we prove a first central result (Lemma \ref{lem:ruellerep}) that gives an isomorphism from first band Pollicott-Ruelle resonant states on $\mathbf V_\tau$ to $\Gamma$-invariant distributional vectors in non-spherical principle series representations. This representation theoretic characterization of first band resonant states is completely general and relies only on the structure theory of rank one Riemannian symmetric spaces. In Proposition \ref{prop:jordanchar} we additionally give a characterization of possible first band Jordan blocks in terms of distributional sections of a boundary vector bundle.

Section \ref{sec:fbpoisson} is then devoted to the identification of the non-spherical principle series representations with generalized common eigenspaces of the algebra of invariant differential operators on the associated vector bundle $G\times_\sigma W\to G/K$. This identification is obtained by working with a general vector-valued Poisson transform developed in the thesis of Olbrich \cite{olbrichdiss}. Some of its advantages and disadvantages in comparison to other vector-valued Poisson transforms are explained in Section \ref{sec:poissonadv}. Furthermore, we link the Poisson transform to the natural pushforwards $h_{\pi*}$ by combining the results of Section \ref{sec:setupnot} and Section \ref{sec:fbpoisson}. In Theorem \ref{thm:main} we obtain the characterization of first band resonant states in the most general setting for arbitrary compatible representations $\tau, \sigma$. 
Based on this result, we prove the statement on the band structure in Section~\ref{sec:band}. The proof is based on horocycle operators similar to those in \cite{dfg}. 

Finally, in Section \ref{sec:specialcase}, we obtain a much more explicit result by imposing the mild  Assumptions 1 and 2. This is achieved by constructing particular generators $\Delta, D_1,\ldots,D_N$ of the algebra $D(G,\sigma)$ that have a nice behavior under the considered algebra representations. 

In a subsequent article \cite{KW19} we apply large parts of the theory developed in the present article to certain bundles on real hyperbolic manifolds, providing  applications to Pollicott-Ruelle resonances and topological phenomena.

\subsection*{Acknowledgments}
We are deeply grateful to Joachim Hilgert for proposing to use 
the general results from Olbrich's thesis in the context of 
Pollicott-Ruelle resonances and for his patient explanations 
of the structure theory of symmetric spaces and different 
aspects of the occurring algebraic and representation theoretic 
techniques. We equally thank Colin Guillarmou for helpful 
discussions and clarifications of analytic and geometric aspects 
of this problem and, in particular, for his detailed explanation 
of the results in \cite{dfg}. Finally, we thank Martin Olbrich for helpful comments. 

\section{Preliminaries}\label{sec:prelim}
In this section we collect some basic notions about Pollicott-Ruelle
resonances, locally symmetric spaces and homogenous vector bundles.
\subsection{Pollicott-Ruelle resonances for geodesic flows}\label{sec:geomres}
Let $\mathcal M$ be a compact boundaryless Riemannian manifold with strictly negative sectional curvature. The geodesic flow $\varphi_t$ on the unit co-sphere bundle 
$S^*\mathcal M$ is generated by a smooth vector field $X$ on $S^*\mathcal M$
and the Liouville measure $d\mu_{\tu{Liouv}}$ is a canonical smooth measure on $S^*\mathcal M$. 
The Pollicott-Ruelle reonances form a canonical spectral invariant associated to the geodesic flow and we want to recall their definition and some
basic properties in this subsection.\footnote{Note that the notion of Pollicott-Ruelle 
resonances is not restricted to geodesic flows on negatively curved manifolds. In fact, nearly
all the statements of this subsection hold for general Anosov flows on compact manifolds
\cite{BL07, FSj11, DZ16}. Even the 
assumption of compactness is not necessary: See \cite{DG16} for flows on non-compact
manifolds with compact trapped set and \cite{BW17} for flows on surfaces with cusps,
where even the trapped set is non-compact.} Note that, unlike in the rest of the article,  
we do not assume that $\mathcal M$ is locally symmetric in this subsection. 
By the strictly negative sectional curvature, we deduce that the geodesic flow on 
$S^* \mathcal M$ is Anosov, i.e., we have the H\"older continuous splitting
\[
 T(S^*\mathcal M) = E_0\oplus E_+\oplus E_-
\]
where $E_0 = \R X$ is the neutral bundle, $E_+$ is the stable bundle and $E_-$ is the
unstable bundle of the Anosov flow. We can then define the dual splitting
$T^*(S^*\mathcal M) = E_0^*\oplus E_+^*\oplus E_-^*$ consisting of subbundles that satisfy fiber-wise
\[
 E_0^*(E_+\oplus E_-)=0,~~~E_\pm^*(E_0\oplus E_\mp) = 0.
\]
The Anosov property is crucial for the definition of Pollicott-Ruelle resonances.

As we want to introduce Pollicott-Ruelle resonances on vector bundles, let us consider
a Hermitian vector bundle $\mathcal V \to\mathcal M$ with compatible connection $\nabla$
and denote by $\Gamma^\infty(\mathcal V)$, $\L^2(S^\ast\mathcal M, \mathcal V)$, and $\mathcal D'(S^*\mathcal M, \mathcal V)$ 
the smooth sections, $\L^2$-sections and distributional sections, respectively. Given 
the connection and the geodesic vector field $X$, we can define the lifted geodesic vector field $\mathbf X:=\nabla_X$,  
which is an unbounded antisymmetric first order differential operator in 
$\L^2(S^\ast\mathcal M, \mathcal V)$. Standard spectral theory implies that the resolvent
$R(\lambda):= (-\mathbf X - \lambda)^{-1}:\L^2(S^\ast\mathcal M, \mathcal V)\to \L^2(S^\ast\mathcal M, \mathcal V)$ is holomorphic
for $\Re(\lambda)>0$. 
\begin{prop}\label{prop:merom_resolvent}
The resolvent operator has a meromorphic continuation to $\C$ as 
a family of continuous operators
\[
 R(\lambda): \Gamma^\infty(\mathcal V) \to \mathcal D'(S^*\mathcal M,\mathcal V).
\]
Given a pole $\lambda_0$ of order $J$, the resolvent takes the form
\[
 R(\lambda) = R_H(\lambda) + \sum_{j=1}^J\frac{(-\mathbf X-\lambda_0)^{j-1}\Pi_{\lambda_0}}{(\lambda-\lambda_0)^j},
\]
where $R_H(\lambda):\Gamma^\infty(\mathcal V) \to \mathcal D'(\mathcal M,\mathcal V)$ is 
a holomorphic familiy of continuous operators and 
$\Pi_{\lambda_0}:\Gamma^\infty(\mathcal V) \to \mathcal D'(\mathcal M,\mathcal V)$ is a finite 
rank operator. Furthermore, the range of the residue operator is given by 
\bq\label{eq:range_pi0}
 \textup{Ran}(\Pi_{\lambda_0}) = \{s\in\mathcal D'(S^*\mathcal M,\mathcal V): (-\mathbf X - \lambda_0)^Js=0, 
 \textup{WF}(s)\subset E_+^*\}.
\eq
Conversely, if for some $\lambda_0\in\C$ there is $s\in\mathcal D'(S^*\mathcal M,\mathcal V)\setminus\{0\}$
such that $\tu{WF(s)}\subset E_+^*$ and
$(\mathbf X+\lambda_0)^ks=0$ for some $k\in \N$, then $\lambda_0$ is a pole of $R(\lambda)$ 
and $s\in \tu{Ran}(\Pi_{\lambda_0})$.
\end{prop}
\begin{proof}
 The meromorphic continuation of the resolvent on vector bundles is a consequence of \cite[Theorem 1]{DG16}. The continuation of the resolvent in the scalar case or on particular vector bundles has been previously shown in \cite{Liv04, FSj11, DZ16}. The structure of the resolvent in a neighborhood of a pole is given in \cite[eq (3.44)(3.55)]{DG16}. The characterization of $\textup{Ran}(\Pi_{\lambda_0})$ in \eqref{eq:range_pi0} is given in \cite[(0.12)]{DG16}.
\end{proof}

\begin{definition}
 We call a pole $\lambda$ of $R(\cdot)$ a \emph{Pollicott-Ruelle resonance on $\mathcal V$}. For  
 $\lambda_0\in \C$ we call 
 \[
  \Res_{\mathcal V}(\lambda_0) :=\{s\in\mathcal D'(S^*\mathcal M,\mathcal V): 
  (-\mathbf X-\lambda_0) s=0, \tu{WF}(s)\subset E_+^*\}
 \]
 the space of \emph{Pollicott-Ruelle resonant states on $\mathcal V$} and for $k\in\N$
 \[
  {\Res_{\mathcal V}}^k(\lambda_0):=\{s\in\mathcal D'(S^*\mathcal M,\mathcal V): 
  (-\mathbf X-\lambda_0)^k s=0, \tu{WF}(s)\subset E_+^*\}
 \]
the space of \emph{generalized Pollicott-Ruelle resonant states on $\mathcal V$} of rank $k$. 
\end{definition}
\begin{rem}~

\begin{enumerate}[leftmargin=*]
 \item  By Proposition~\ref{prop:merom_resolvent}, $\lambda_0\in \C$ is a Pollicott-Ruelle resonance  
 iff $\Res_{\mathcal V}(\lambda_0)\neq 0$.
 \item If $J$ is such that 
 ${\Res_{\mathcal V}}^{J-1}(\lambda_0) \subsetneq {\Res_{\mathcal V}}^J(\lambda_0)
 = {\Res_{\mathcal V}}^{J+1}(\lambda_0)$, then the resolvent has a pole of order $J$.
In this case, there are distributions $s_1,\ldots, s_J$, $s_k\in {\Res_{\mathcal V}}^{k}(\lambda_0)\setminus\{0\}$
such that $s_k=(-\mathbf X-\lambda_0)s_{k+1}$. We say that  $\lambda_0$ lies in a 
\emph{Jordan block of size $J$}. 
\item For any $\lambda_0\in\C$ with $\Re(\lambda_0)>0$ we know that $R(\lambda)$ is
holomorphic in a neighborhood of $\lambda_0$ and conclude $\Res_{\mathcal V}(\lambda_0)=\{0\}$.
\end{enumerate}
\end{rem}
Besides the resonant states we want to define \emph{co-resonant states}. Therefore, let 
$\mathcal V^*$ be the dual bundle of $\mathcal V$. The scalar product on the fibers induces
an antilinear bundle isomorphism
\[
\mathbf d: \Abb{\mathcal V}{\mathcal V^*}{v\in\mathcal V_p}{\langle\cdot,v\rangle_p\in\mathcal (V_p)*}
\]
which allows to transfer the connection on $\mathcal V$ to $\mathcal V^*$ by setting 
\[
 \nabla^*_Y s^*:= \mathbf d(\nabla_Y\mathbf d^{-1}s^*) \tu{ for }s^*\in \Gamma^\infty(\mathcal V^*), 
 Y\in\Gamma^\infty(T(S^*\mathcal M)).
\]
If we define $\mathbf X^*:=\nabla^*_{-X}$, then this is the dual differential operator of $\mathbf X$
in the sense that for $s\in\Gamma^\infty(\mathcal V)$, $t^*\in\Gamma^\infty(\mathcal V^*)$
\[
 \int_{S^*\mathcal M} \langle t^*,\mathbf X s\rangle_{\mathcal V^*_p,\mathcal V_p} d\mu_{\tu{Liouv}}(p)
 =\int_{S^*\mathcal M} \langle \mathbf X^*t^*, s\rangle_{\mathcal V^*_p,\mathcal V_p} d\mu_{\tu{Liouv}}(p).
\]
We define the space of \emph{generalized Pollicott-Ruelle co-resonant states of rank k} as
\[
  {\Res^*_{\mathcal V}}^k(\lambda_0):=\{s^*\in\mathcal D'(S^*\mathcal M,\mathcal V^*): 
  (-\mathbf X^*-\lambda_0)^k s^*=0, \tu{WF}(s^*)\subset E_-^*\}
\]
and we again denote the space of Pollicott-Ruelle co-resonant states by $\Res^*_{\mathcal V}(\lambda_0)
:={\Res^*_{\mathcal V}}^1(\lambda_0)$
\begin{lem}
 $\Res^*_{\mathcal V}(\lambda_0)\neq 0$ iff $\lambda_0$ is a Pollicott-Ruelle resonance.
 Furthermore, if $s\in\Res_{\gam \mathcal V}(\lambda_0)$, $t^*\in\Res^*_{\mathcal V}(\lambda_0)$,
 then the distribution $t^*s \in\mathcal D'(S^*\mathcal M)$ obtained by fiber-wise pairing is invariant under the geodesic flow.
\end{lem}
\begin{proof}
 If we reverse the time direction of the geodesic flow, i.e., if we replace $X$ by $-X$, then
 the flow is still Anosov and the role of stable and unstable bundle are interchanged. We
 can thus apply Proposition~\ref{prop:merom_resolvent} to the time-revesed flow and the bundle 
 $\mathcal V^*$. Then we conclude 
 that $\Res^*_{\mathcal V}(\lambda_0)\neq 0$ iff $\lambda_0$ is a pole of 
 $R^*(\lambda) = (-\mathbf X^*-\lambda)$. This in turn implies that there are sections 
 $t^*\in\Gamma^\infty(\mathcal V^*), s\in\Gamma^\infty(\mathcal V)$ such that the  meromorphic function 
 \[
    \int_{S^*\mathcal M} \langle R^*(\lambda) t^*, s\rangle_{\mathcal V^*_p,\mathcal V_p} 
    d\mu_{\tu{Liouv}}(p)  =\int_{S^*\mathcal M} \langle  t^*, R(\lambda)s\rangle_{\mathcal V^*_p,\mathcal V_p} 
    d\mu_{\tu{Liouv}}(p)  
 \]
has a pole in $\lambda_0$. Now the expression on the right hand side implies that 
$R(\lambda)=\Gamma^\infty(\mathcal V)\to\mathcal D'(S^*\mathcal M,\mathcal V)$ has a pole
at $\lambda_0$, so $\lambda_0$ is a Pollicott-Ruelle resonance. 
The product $t^*s$ is well defined because of the transversal wave front sets and the invariance
follows from 
\[
 X t^*s = (\nabla^*_Xt^*)s + t^*(\nabla_X s) = \lambda_0t^*s - \lambda_0t^*s=0.
\]
\end{proof}
\begin{rem}
 For scalar Pollicott-Ruelle resonances, invariant distributions of the type $t^*s$
 play an important role in \cite{GHW18b} for the study of quantum ergodic properties. It is
 very likely that vector-valued pairings of resonant and co-resonant states are interesting
 objects to study vector-valued generalizations of these results. 
\end{rem}
\subsection{Symmetric spaces}\label{sec:setupnot}
Let $G$ be a connected non-compact semisimple real Lie group of\footnote{Many of the constructions presented in the following have straightforward generalizations to the higher rank case.} real rank $1$ with finite center, and $\g$ its Lie algebra. Let $\theta:\g\to \g$ be a Cartan involution, $\mathfrak{B}:\g\times \g\to \R$ the Killing form, and define an inner product $\eklm{\cdot,\cdot}$ on $\g$ by $(X,Y)\mapsto -\mathfrak{B}(X,\theta Y)$. We will frequently identify $\g=\g^\ast$ using this inner product and write $\norm{\cdot}$ for the associated norm on $\g$ and $\g^\ast$. Let $\k\subset \g$ be the eigenspace of the involution $\theta$ with eigenvalue $1$ and $\p\subset \g$ the eigenspace with eigenvalue $-1$. Choose a maximal abelian subalgebra $\aL\subset \p$. That $G$ has real rank $1$ means that $\dim \aL=1$. Let $\Sigma\subset \aL^\ast$ denote the set of restricted roots with respect to $\aL$. Let $\mathscr{W}=\mathscr{W}(\Sigma)$ be the Weyl group generated by the reflections at the hyperplanes perpendicular to the restricted root vectors. Choose a set $\Sigma_+\subset \Sigma$ of positive restricted roots and define
\[
\nL^\pm:=\bigoplus_{\alpha \in \pm\Sigma_+}\g_\alpha,
\]
where $\g_\alpha=\{X\in \g: [H,X]=\alpha(H)X\;\forall\;H\in \aL\}$ is the root space associated to the root $\alpha$. 
Then one has $\nL^\pm=\theta \nL^\mp$. As $G$ has real rank $1$, there is only one reduced root in $\Sigma_+$ and only one non-trivial element in  the Weyl group $\mathscr{W}$. We write $\aL^\ast_{\C}:=\aL^\ast\otimes_\R \C$ for the complexification of the real vector space $\aL^\ast$. 

\begin{definition}We denote the unique non-trivial element in $\mathscr{W}$ by $w_0$, the unique reduced root in $\Sigma_+$ by $\alpha_0$, and by \bq
H_0\in \aL\label{eq:H0}
\eq the element corresponding to 
\bq
\nu_0:=\frac{\alpha_0}{\norm{\alpha_0}}\label{eq:nu0}
\eq 
under the isomorphism  $\aL\to \aL^\ast$ provided by the Killing form or, equivalently, by the inner product.
\end{definition}
The decomposition of $\g$ into restricted root spaces can be written as the Bruhat decomposition  
$$\g=\m\oplus \nL^+\oplus \aL\oplus \nL^-\qquad \text{orthogonally},$$
where $\m=Z_\k(\aL)$ is the centralizer of $\aL$ in $\k$. 
We also obtain Iwasawa and opposite Iwasawa decompositions of the Lie algebra 
$$\g=\k\oplus\aL\oplus\nL^+= \k\oplus\aL\oplus \nL^- \qquad \text{not orthogonally in general}$$
and on the group level:
$$G=KAN^+= KAN^-.$$
Here $N^\pm\subset G$ are the analytic subgroups with Lie algebras $\nL^\pm$, $K\subset G$ is the analytic subgroup with Lie algebra $\k$ and is maximal compact in $G$, and $A\subset G$ is the analytic subgroup with Lie algebra $\aL$. For each group element $g\in G$ we now have unique Iwasawa ($+$) and opposite Iwasawa ($-$) decompositions
\begin{align}\begin{split}
g&=k^+(g)a^+(g)n^+(g)=k^+(g)\exp(H^+(g))n^+(g)\\
&=k^-(g)a^-(g)n^-(g)=k^-(g)\exp(H^-(g))n^-(g).\label{eq:iwasawadecomp}
\end{split}
\end{align}
This provides us with maps
\[
k^\pm: G\to K,\qquad a^\pm: G\to A,\qquad H^\pm: G\to \aL,\qquad n^\pm: G\to N^\pm,
\]
where $\exp\circ H^\pm=a^\pm$, $\exp$ being the exponential map on $\aL$. 
In addition, define the group $$M:=Z_K(\aL)=Z_K(A),$$
where $K$ acts on $\aL$ by the adjoint action, so that $\m=\text{Lie(M)}$. The groups $N^\pm$ are normalized by $A$ and $M$. Furthermore, the respective exponential maps provide diffeomorphisms $A\cong \aL$, $N^\pm\cong \nL^\pm$. We write $\log$ for the inverse maps.

Besides $M$, let us introduce also the group $M':=N_K(\aL)$, the normalizer of $K$ in $\aL$. Then it is a fundamental result in Lie theory that the Weyl group can be identified with the quotient group
\bq
\mathscr{W}=M'/M.\label{eq:Weyl}
\eq
We define an action of $\mathscr{W}$ on the set $\hat M$ of equivalence classes of irreducible $M$-representations by setting
\bq
w[\tau]:= [m'_w\tau],\quad m'_w\tau(m):=\tau((m'_w)^{-1}m\,m'_w),\quad m\in M,w\in \mathscr{W},[\tau]\in \hat M,\label{eq:actionW111}
\eq
where $m'_w\in M'$ is an arbitrary representant of $w$ in $\mathscr{W}=M'/M$. It is easy to check that the definition of $w[\tau]$ is independent of the choices of $m_w$ and the representant $\tau$ of $[\tau]$. 

\subsubsection{The geodesic vector field}
Given $H\in \aL$, let us write $\X_H$ for the fundamental vector field on $G/M$ associated to $H$ by the action of $A$ on $G/M$ by right multiplication $A\owns a\mapsto (gM\mapsto gaM)$. This action is well-defined due to the definition of $M$ as the centralizer of $A$ in $K$.  One then easily checks that $\X_H$ is invariant under the left regular $G$-action \eqref{eq:actionVF}. There is a distinguished  vector field $\mathfrak{X}_{H_0}\in \Gamma^\infty(E_0)$  associated to the element $H_0\in \aL$ from \eqref{eq:H0}. 
As $\nL^+,\nL^-$ and $\aL$ are each $\Ad(M)$-invariant, $T(G/M)$ splits into a product bundle
\bq
T(G/M)=E_0\oplus E_+\oplus E_-,\qquad 
 E_0=G\times_{\Ad(M)}\aL,\qquad E_\pm= G\times_{\Ad(M)}\nL^\pm.\label{eq:splitting}
\eq
We define the dual splitting $T^*(G/M)=E_0^\ast\oplus E_+^\ast\oplus E_-^\ast$ by the  fiber-wise relations
\bq
 E^\ast_0(E_+\oplus E_-)=0,\qquad E^\ast_\pm(E_0\oplus E_\mp)=0.\label{eq:dualsplitting}
\eq
Note that our convention (\ref{eq:dualsplitting}) differs from that in \cite{GHW18b}. 
\begin{rem}[$G/M$ as a unit cosphere bundle]\label{rem:geodvf}We can regard $G/M$ as the unit cosphere bundle $S^\ast(G/K)$ over the negatively curved Riemannian symmetric space $G/K$. The splitting (\ref{eq:splitting}) then corresponds to the Anosov splitting of the tangent bundle of $S^\ast(G/K)$. The vector field $\mathfrak{X}_{H_0}$ corresponds to the generating vector field of the geodesic flow on $S^\ast(G/K)$. In this picture, the  choice $\Sigma_+$ of a positive restricted root system corresponds to fixing a positive time direction for the geodesic flow. \end{rem}
\subsubsection{Relevant measures and the special element $\rho$}\label{sec:integration1}
By \cite[Chapter VIII]{knapp} the groups $G$, $K$, $A$, and $N^\pm$ are each unimodular. In contrast, the groups $AN^\pm$ are not unimodular, the modular functions being given by
\bq
\Delta_{AN^\pm}(an)=e^{\pm2\rho(\log a)},\quad an\in AN^\pm,\qquad \rho=\frac{1}{2}\sum_{\alpha \in \Sigma_+}(\dim \g_\alpha)\alpha\quad\in \aL^\ast.\label{eq:rho}
\eq
Choose bi-invariant Haar measures $dk$ and $dm$ on $K$ and $M$, respectively, normalized such that $\vol K=\vol M=1$, and a bi-invariant Haar measure $dg$ on $G$. By \cite[Thm.\ 8.36]{knapp} there are unique $G$-invariant measures $d(gM)$ and $d(gK)$ on the homogeneous spaces $G/M$ and $G/K$, respectively, such that for all $f\in \mathrm{C}_c(G)$ we have
\begin{align}
\int_G f(g)\d g&= \int_{G/M}\int_M f(gm) \d m \d (gM)= \int_{G/K}\int_K f(gk) \d k \d (gK).\label{eq:measureGK}
\end{align}
Similarly, there is a unique $K$-invariant measure $d(kM)$ on $K/M$ with the property
\bq
\int_K f(k)\d k= \int_{K/M}\int_M f(km) \d m \d (kM)\qquad \forall\; f\in \mathrm{C}(K).\label{eq:measureKM}
\eq

\subsubsection{The boundary at infinity}\label{sec:boundary}
The subspace $K/M\subset G/M$ can be regarded as the \emph{boundary at infinity} of the Riemannian symmetric space $G/K$, see \cite[Section 3.3]{GHW18b}. 
From the opposite Iwasawa decomposition \eqref{eq:iwasawadecomp} we obtain a surjective map
\bq	
B:G/M\to K/M,\qquad gM\mapsto k^-(g)M,\label{eq:endpointmap}
\eq
which is well-defined as $M$ normalizes $N$ and centralizes $A$. The map $B$ is equivariant with respect to the $G$-action on $G/M$ given by
\bq
g'(gM):=g'gM,\qquad g,g' \in G,\label{eq:Gactionsbase}
\eq
and the $G$-action on $K/M$ given by
\bq
g(kM):=k^-(g k)M,\qquad g \in G,\;k\in K.\label{eq:GactionsbaseK}
\eq
That this is a well-defined $G$-action is straightforward to check using the opposite Iwasawa decomposition. 
\begin{rem}\label{rem:differentB}Note that we use the opposite Iwasawa decomposition in the definitions of the map $B$ and the action \eqref{eq:GactionsbaseK}, in contrast to \cite[Section 3.3]{GHW18b} where the minimal parabolic subgroup $P=MAN$ is used.  The \emph{initial point map} $B_-$ occurring in \cite[(3.1)]{GHW18b} therefore differs from our map $B$. The two conventions are equivalent while ours has the notational advantage that we do not have to deal with an additional Weyl group element in many computations.
\end{rem}
\begin{lem}\label{lem:factorizationGM}
The map
\bqn
F:G/M\longrightarrow G/K\times K/M,\qquad 
gM\longmapsto (gK,B(gM)),
\eqn
is a $G$-equivariant diffeomorphism, where $G/M$ is equipped with the $G$-action \eqref{eq:Gactionsbase} given by left multiplication, and $G/K\times K/M$ is equipped with the diagonal action formed by the left multiplication $G$-action on $G/K$ and the $G$-action \eqref{eq:GactionsbaseK} on $K/M$. The inverse map is given by
\bqn
F^{-1}:G/K\times K/M\longrightarrow G/M, \qquad 
(gK,kM)\longmapsto gk^-(g^{-1}k)M.
\eqn
It fulfills for each $f\in C_c(G/M)$ the transformation formula
\bq
\intop_{G/M}f(gM)\d(gM) =\intop_{G/K}\intop_{K/M}f(F^{-1}(gK,kM))e^{2 \rho(H^-(g^{-1}k))}\d(kM)\d(gK).\label{eq:intstatement}
\eq
\end{lem}
\begin{proof}The map $F$ is well-defined because $M\subset K$, it is smooth and $G$-equivariant because $B$ is smooth and $G$-equivariant. The proof of the transformation formula is given in \cite[Proof of Prop.\ 4.3]{GHW18b} without referring to the inverse. That $F^{-1}$ is well-defined and indeed the inverse of $F$ is a tedious but straightforward computation using the opposite Iwasawa decomposition.
\end{proof}
\subsection{Associated vector bundles}\label{sec:homvb}

Let $V$ be a finite-dimensional complex vector space and $\tau:M\to \text{End}(V)$ a complex representation. 
Writing $G\times_{\tau} V:=G\times V/\sim$ with $(gm,v)\sim (g,\tau(m)v)$ for $m\in M$, we define the associated homogeneous vector bundle
\bq
\V_{\tau}:=(G\times_{\tau} V,\pi_{\V_{\tau}}),\qquad \pi_{\V_{\tau}}([g,v])=gM
\eq
over $G/M$. Its total space carries the $G$-action
\bq
\gamma[g,v]:=[\gamma g,v],\qquad [g,v]\in G\times_{\tau} V,\quad \gamma\in G,\label{eq:GactionV}
\eq
which is a lift of the left-$G$-action $\gamma(gM):=\gamma gM$ on the base and consequently induces the following \emph{left regular $G$-action} action on smooth sections:
\bq
({\gamma} s)(gM):={\gamma}\big(s\big({\gamma}^{-1}gM\big)\big),\qquad s\in \Gamma^\infty(\V_{\tau}),\quad \gamma,g\in G.\label{eq:GactionVsections}
\eq
We will frequently identify a smooth section $s\in \Gamma^\infty(\V)$ with a smooth right-$M$-equivariant function $\bar s\in \Cinft(G,V)$, the right-$M$-equivariance meaning that $\bar s(gm)=\tau(m^{-1})\bar s(g)$. With respect to this identification, the left regular action (\ref{eq:GactionVsections}) reads \bq
({\gamma}\bar s)(g)=\bar s({\gamma}^{-1}g),\qquad g,{\gamma}\in G.  \label{eq:leftregular}
\eq
\begin{rem}\label{rem:dual}The dual vector bundle $\V_\tau^\ast$ can be canonically identified with the homogeneous vector bundle $\V_{\tau^\ast}$ associated to the dual representation $\tau^\ast:M\to \mathrm{End}(V^\ast)$ induced by $\tau$.
\end{rem}

\subsubsection{The canonical connection and the lifted geodesic vector field}
There is a canonical connection on $\V_{\tau}$ that we denote by
\bq
\nabla: \Gamma^\infty(\V_{\tau})\to \Gamma^\infty(\V_{\tau}\otimes T^\ast(G/M))\cong \Gamma^\infty(\mathrm{hom}(T(G/M),\V_{\tau})).
\eq
To describe how $\nabla$ is defined, let us regard a section $s\in \Gamma^\infty(\V)$ as a right-$M$-equivariant function $\bar s\in \Cinft(G,V)$. Moreover, by \eqref{eq:splitting} we can regard a vector field $\X\in \Gamma^\infty(T(G/M))$ as a right-$M$-equivariant function $\bar \X\in \Cinft(G,\nL^+\oplus \aL\oplus {\nL^-})$, that is, $\bar \X(gm)=\Ad(m^{-1})\bar \X(g)$. Then $\nabla$ is defined by the covariant derivative
\begin{align}\begin{split}
\nabla(s)(gM)(\X):=\nabla_{\X}(s)(gM):=\frac{d}{dt}\Big|_{t=0}\bar s\big(ge^{t\bar \X(g)}\big)\label{eq:covariantderiv}
\end{split}
\end{align}
Being a connection, $\nabla$ fulfills
 \bq
\nabla(fs)=df\otimes s+f\nabla(s)\qquad \forall\; f\in \Cinft(G/M),\; s\in \Gamma^\infty(\V_{\tau}).\label{eq:Xextension}
\eq
We obtain analogously as in \eqref{eq:leftregular}  the left regular $G$-action
\bq
\overline{\gamma \X}(g):=\bar \X({\gamma}^{-1}g),\qquad g,{\gamma}\in G,\;\X\in \Gamma^\infty(T(G/M)).\label{eq:actionVF}
\eq
It is straightforward to check that, with respect to the left regular $G$-actions \eqref{eq:GactionVsections} and \eqref{eq:actionVF} on $\Gamma^\infty(\V_\tau)$ and $\Gamma^\infty(T(G/M))$, the connection $\nabla$ transforms according to
\bq
g(\nabla_\X s)=\nabla_{g\X}(gs)\qquad \forall\;g\in G,\;\X\in \Gamma^\infty(T(G/M)),\;s\in \Gamma^\infty(\V_\tau).\label{eq:transf}
\eq
From this one easily deduces that the covariant derivatives $\nabla_{\X_H}$ with $H\in \aL$ commute with the left regular $G$-action \eqref{eq:GactionVsections} on $\Gamma^\infty(\V_\tau)$. Remark \ref{rem:geodvf} motivates the following
\begin{definition}\label{def:liftedgeodvf}
The \emph{lifted geodesic vector field} is
\bq
\Xbf:=\nabla_{\mathfrak{X}_{H_0}}: \Gamma^\infty(\V_{\tau})\to \Gamma^\infty(\V_{\tau}).
\eq
\end{definition}
From (\ref{eq:Xextension}) we conclude that the lifted geodesic vector field satisfies
\bq
\Xbf(fs)=(\mathfrak{X}_{H_0}f)s+f\Xbf(s)\label{eq:Xextensionrk1}
\eq
for all $f\in \Cinft(S(G/K))\cong \Cinft(G/M)$ and  $s\in \Gamma^\infty(\V_{\tau})$, which justifies its name. 

\subsubsection{Extension to distributional sections}
By standard duality constructions we extend the left regular $G$-action \eqref{eq:GactionVsections} as well as the covariant derivatives $ \nabla_{\mathfrak{X}}$ and the operator $\Xbf$ 
by duality to operators acting on distributional sections in $\D'(G/M,\V_{\tau})$. 
For the definition of distributional sections and technical aspects, see Appendix \ref{sec:WF}. With respect to the left regular $G$-actions on $\D'(G/M,\V_{\tau})$ and $\Gamma^\infty(T(G/M))$, the connection $\nabla$ then has the transformation property
\bq
g(\nabla_\X s)=\nabla_{g\X}(gs)\label{eq:transf2}
\eq
for all $g\in G$, $\X\in \Gamma^\infty(T(G/M))$, $s\in \D'(G/M,\V_{\tau})$, as a consequence of \eqref{eq:transf}.
\subsubsection{The boundary vector bundle and its pullback}\label{sec:boundarybundle}
We call the restriction of $\V_{\tau}$ to $K/M\subset G/M$ the \emph{boundary vector bundle}, and denote it by $\V^{\mathcal{B}}_{\tau}$. It is a homogeneous vector bundle associated to $\tau$: $${\V^{\mathcal{B}}_{\tau}} =(K\times_{\tau} V,\pi_{\V^{\mathcal{B}}_{\tau}}),\qquad \pi_{\V^{\mathcal{B}}_{\tau}}([k,v])=kM.$$ 
  Pulling back $\V^{\mathcal{B}}$ along the map $B$ from \eqref{eq:endpointmap}, we obtain the pullback vector bundle $$B^\ast{\V^{\mathcal{B}}_{\tau}}:=(B^\ast(K\times_{\tau} V),\pi_{B^\ast{\V^{\mathcal{B}}_{\tau}}})$$ over $G/M$ with total space
\begin{align}
\nonumber B^\ast(K\times_{\tau} V)&=\{(gM,[k,v])\in G/M\times (K\times_{\tau} V): B(gM)=\pi_{\V^{\mathcal{B}}_{\tau}}([k,v])\}\\
&=\{(gM,[k,v])\in G/M\times (K\times_{\tau} V): k^-(g)M=kM\}\label{eq:pullbackbundle1}
\end{align}
and bundle projection $\pi_{B^\ast{\V^{\mathcal{B}}_{\tau}}}(gM,[k,v])=gM.$ 
The total space $K\times_{\tau} V$ of $\V^{\mathcal{B}}_{\tau}$ carries the $G$-action
\bq 
g[k,v]:=[k^-(gk),v],\qquad g\in G,\;k\in K,\label{eq:Gactionstotal}
\eq
that lifts the $G$-action \eqref{eq:GactionsbaseK} on the base space $K/M$, and the total space $B^\ast(K\times_{\tau} V)$ of $B^\ast{\V^{\mathcal{B}}_{\tau}}$ carries the $G$-action
\bq
{\gamma}(gM,[k,v]):=({\gamma}gM,[k^-({\gamma}k),v]),\qquad
g,\gamma\in G,\;k\in K,\;v\in V,\label{eq:Gactionstotal2}
\eq
that lifts the $G$-action \eqref{eq:Gactionsbase} on the base space $G/M$. Consequently we get induced actions on smooth sections:
\begin{align}
(\gamma s)(kM)&:={\gamma}\big(s\big(\gamma^{-1}(kM)\big)\big),\qquad s\in \Gamma^\infty(\V^{\mathcal{B}}_{\tau}),\label{eq:Gactionstotalsections1}\\
({\gamma}s)(gM,[k,v])&:={\gamma}\big(s\big({\gamma}^{-1}(gM,[k,v])\big)\big),\qquad s\in \Gamma^\infty(B^\ast{\V^{\mathcal{B}}_{\tau}}).
\label{eq:Gactionstotalsections2}
\end{align} 
If we consider a section $s\in \Gamma^\infty(\V^{\mathcal{B}}_{\tau})$ as a right-$M$-equivariant smooth function $\bar s:K\to V$, the action \eqref{eq:Gactionstotalsections1} corresponds to assigning to $\bar s$ for any $g\in G$ the right-$M$-equivariant smooth function $\overline{g s}:K\to V$ given by
\bq
\overline{gs}(k)=\bar s(k^-(g^{-1}k)),\qquad g\in G,\;k\in K.\label{eq:Gactionsectionslanglands}
\eq
\begin{lem}\label{lem:bundleiso}The maps
\begin{align*}
\mathscr{T}:B^\ast(K\times_{\tau} V)&\longrightarrow G\times_{\tau} V&\mathscr{S}:G\times_{\tau} V&\longrightarrow B^\ast(K\times_{\tau} V)\\
(gM,[k,v])&\longmapsto [gk^-(g)^{-1}k,v] &[g,v]&\longmapsto (gM,[k^-(g),v])
\end{align*}
define an isomorphism of vector bundles 
\bq
B^\ast\V_{\tau}^{\mathcal{B}}\stackrel[\mathscr{S}]{\mathscr{T}}{\longrightleftarrows} \V_\tau\label{eq:bundleiso}
\eq
that is $G$-equivariant with respect to the $G$-actions (\ref{eq:GactionV}) and (\ref{eq:Gactionstotal}).
\end{lem}
\begin{proof}
The map $\mathscr{T}$ is well-defined since one has for arbitrary $m,m'\in M$
\begin{align*}
\mathscr{T}(gmM,[km',(m')^{-1}v])=[gmk^-(gm)^{-1}km',(m')^{-1}v]&=[gm(k^-(g)m)^{-1}k,v]\\
&=[gk^-(g)^{-1}k,v]\\
&=\mathscr{T}(gM,[k,v]),
\end{align*}
where we used that the opposite Iwasawa decomposition \eqref{eq:iwasawadecomp} satisfies $k^-(gm)=k^-(g)m$. Using (\ref{eq:pullbackbundle1}) it is easy to check that $\mathscr{T}$ is a morphism of vector bundles. Moreover, $\mathscr{T}$ is $G$-equivariant: 
\begin{align*}
\mathscr{T}({\gamma}(gM,[k,v]))=\mathscr{T}({\gamma}gM,[k^-({\gamma}k),v]))=[{\gamma}gk^-({\gamma}g)^{-1}k({\gamma}k),v]&=[{\gamma}gk^-(g)^{-1}k,v]\\
&={\gamma}\mathscr{T}(gM,[k,v]).
\end{align*}
To get from the second to the third line we used (\ref{eq:pullbackbundle1}) and again the property $k^-(gm)=k^-(g)m$, $m\in M$, of the opposite Iwasawa decomposition. Similarly, we check that the maps $\mathscr{S}$ are well-defined and $G$-equivariant. It remains to show that $\mathscr{T}$, $\mathscr{S}$ are inverses of each other. We compute
\begin{align*}
\mathscr{T}\big(\mathscr{S}([g,v])\big)&=\mathscr{T}(gM,[k^-(g),v])
=[gk^-(g)^{-1}k^-(g),v]
=[g,v],\\
\mathscr{S}\big(\mathscr{T}(gM,[k,v])\big)&=\mathscr{S}_+([gk^-(g)^{-1}k,v])=(gk^-(g)^{-1}kM,[k^-(gk^-(g)^{-1}k),v])=(gM,[k,v]),
\end{align*}
where we used (\ref{eq:pullbackbundle1}) in the last step. Indeed, (\ref{eq:pullbackbundle1}) implies that $k^-(g)^{-1}k\in M$, and by a stanard Iwasawa calculation we get since $M$ normalizes $N^-$ and elements of $M$ commute with elements of $A$, we get 
$k^-(gk^-(g)^{-1}k)=k$ as required.  
\end{proof}

\section{Pollicott-Ruelle resonances on associated vector bundles}\label{sec:res_on_VB}
We begin with technical preparations, which mainly involve generalizing concepts and definitions of \cite[Section 3]{GHW18b} from the scalar case to the vector-valued case. 
\begin{definition}\label{def:repruelle1stband}
For $\lambda \in \aL^\ast_{\C}$ we define\footnote{Note that our set $\mathcal{R}(\lambda)$ agrees with the set $\mathcal{R}_-(\lambda)$ in \cite{GHW18b} in the case of a trivial line bundle.} 
\bqn
\mathcal{R}(\lambda):=\{s\in \D'(G/M,\V_{\tau}):(\Xbf + \lambda(H_0))(s)=0,\;\nabla_{\X}(s)=0\quad\forall\; \X\in \Gamma^\infty(E_-)\}.
\eqn
\end{definition}
Due to the fact that $\dim \aL=1$, one gets in the special case $\lambda=0$ 
\bq
\mathcal{R}(0)=\{s\in \D'(G/M,\V_{\tau}):\nabla_{\X}(s)=0\quad\forall\; \X\in \Gamma^\infty(E_0\oplus E_-)\}.\label{eq:R0}
\eq
\begin{rem}\label{rem:Gactionsrestr}
One easily verifies that for each $\lambda \in \aL^\ast_\C$ the subset $\mathcal{R}(\lambda)\subset\D'(G/M,\V_{\tau})$ is $G$-invariant with respect to the left regular $G$-action,  so the latter restricts to a $G$-action on $\mathcal{R}(\lambda)$. We call that action again simply the \emph{left regular $G$-action}.
\end{rem}
\begin{rem}[Topologies on the introduced spaces]In this paper, we equip all the introduced subspaces of spaces of distributional sections with the subspace topology induced by the standard topology on the spaces of distributional sections (weak convergence w.r.t.\ the pairing with test sections). 
\end{rem}

We can relate the sets $\mathcal{R}(\lambda)$ for $\lambda\neq 0$ to $\mathcal{R}(0)$ as follows. Define $\Phi\in \Cinft(G/M)$ as the function corresponding to the right-$M$-invariant function 
\bq
\bar \Phi(g):= e^{-\nu_0(H^-(g))}\label{eq:defphi}
\eq
with the notation \eqref{eq:iwasawadecomp}. As in  \cite[(3.8), (3.9)]{GHW18b}, it is easy to check using the opposite Iwasawa decomposition that the function $\Phi$ has the properties\footnote{Recall Remark \ref{rem:geodvf} in this context.}
\bq
\mathfrak{X}_{H_0}\Phi=-\Phi,\qquad \X_- \Phi=0\qquad\forall\;\X_-\in \Gamma^\infty(E_-).\label{eq:Phiproperties}
\eq
\begin{lem}\label{lem:homeo293509} For each $\lambda \in \aL_\C^\ast$ there is an isomorphism of topological vector spaces
\begin{align*}
\mathcal{R}(\lambda)&\longrightleftarrows \mathcal{R}(0),\\
s&\longmapsto \Phi^{-\lambda(H_0)}s,\\
\Phi^{\lambda(H_0)}s&\longmapsfrom s.
\end{align*}
\end{lem}
\begin{proof}Fix $\lambda_1,\lambda_2 \in \aL_\C^\ast$ and let $s\in  \mathcal{R}(\lambda_1)\subset \D'(G/M,\V_{\tau})$, which means that
\bq
(\Xbf + \lambda_1(H_0))(s)=0,\;\nabla_{\X_-}(s)=0\;\forall\; \X_-\in \Gamma^\infty(E_-),\label{eq:vanish12424}
\eq
see Definition \ref{def:repruelle1stband}. 
We compute using (\ref{eq:vanish12424}), (\ref{eq:Xextensionrk1}), and (\ref{eq:Xextension}) 
\begin{align*}
\Xbf\big(\Phi^{\lambda_2(H_0)}s\big)&=\mathfrak{X}_{H_0} \big(\Phi^{\lambda_2(H_0)}\big)s+\Phi^{\lambda_2(H_0)}\Xbf(s)\\
&=-\lambda_2(H_0)\Phi^{\lambda_2(H_0)}s-\lambda_1(H_0)\Phi^{\lambda_2(H_0)}s\\
&=-(\lambda_2(H_0)+\lambda_1(H_0))\Phi^{\lambda_2(H_0)}s,
\end{align*}
\[
\nabla_{\X_-}\big(\Phi^{\lambda_2(H_0)}s\big)=\mathfrak{X}_- \big(\Phi^{\lambda_2(H_0)}\big)s+\Phi^{\lambda_2(H_0)}\nabla_{\X_-}(s)=0\qquad \forall\;\X_-\in \Gamma^\infty(E_-).
\]
The claim follows by setting $(\lambda_1,\lambda_2)=(\lambda,-\lambda)$ or  $(\lambda_1,\lambda_2)=(0,\lambda)$, respectively, and taking into account that multiplication of distributional sections with smooth functions is a continuous operation.
\end{proof}
\begin{lem}\label{lem:WF} For each $\lambda\in \aL^\ast_{\C}$ one has
\[
s\in \mathcal{R}(\lambda)\implies \mathrm{WF}(s)\subset E_+^\ast.
\]
Here $\mathrm{WF}(s)$ is the wave front set of the distributional section $s$, see Appendix \ref{sec:WF}.
\end{lem}
\begin{proof}By definition, $s\in  \D'(G/M,\V_{\tau})$ is an element of $\mathcal{R}(\lambda)$ if $s$ fulfills 
\bq
\big(\Xbf+\lambda(H)\big)(s)=0\quad\forall\;H\in \aL\setminus\{0\},\qquad \nabla_{\X_-}(s)=0\;\forall\;\X_-\in\Gamma^\infty(E_-).\label{eq:vanish1}
\eq
Now, the principal symbol of an operator of the form $\Xbf+\lambda(H)$ agrees with the principal symbol of $\Xbf$. Using \eqref{eq:Xextension}, one easily computes that the principal symbol $\mathrm{psmb}(\Xbf)$ of $\Xbf$ is the vector bundle homomorphism $T^*(G/M)\to \mathrm{end}(\V_{\tau})$  given by 
\bqn
\mathrm{psmb}(\Xbf)(\xi)([g,v])=i\xi({\mathfrak{X}_{H_0}}(gM))\cdot [g,v],\qquad [g,v]\in \V_{\tau},
\quad \xi\in T^\ast_{gM}(G/M).
\eqn
Here $\mathrm{end}(\V_{\tau})$ is the endomorphism bundle associated to $\V_{\tau}$. Consequently, $\mathrm{psmb}(\Xbf)$ is fiber-wise invertible precisely at the points in $T^*(G/M)\setminus\mathrm{Ann}(\mathfrak{X}_{H_0})$, where
\[
\mathrm{Ann}(\mathfrak{X}_{H_0})=\{(gM,\xi)\in T^*(G/M): \xi(\mathfrak{X}_{H_0}(gM))=0\}
\]
is the set of covectors annihilating the vectors in the image of the geodesic vector field $\mathfrak{X}_{H_0}$. Since that image spans fiber-wise $E_0$, such a covector lies in the annihilator of $E_0$, which is $E^\ast_+\oplus E^\ast_-$. Arguing analogously for $\nabla_{\X_-}$, $\X_-\in\Gamma^\infty(E_-)$, instead of $\Xbf\equiv \nabla_{\mathfrak{X}_{H_0}}$, we deduce that a covector at which the principal symbols of all the obtained operators is not invertible must lie in the annihilator of $E_-$, which is $E_0^\ast\oplus E_+^\ast$. In total, we get that a covector at which the principal symbols of all the operators $\nabla_{\X_-}$, $\X_-\in \Gamma^\infty(E_-)$, and $\Xbf+\lambda(H)$ are not invertible must be an element of $(E^\ast_+\oplus E^\ast_-)\cap (E_0^\ast\oplus E_+^\ast)=E_+^\ast$. The fibers of $E_+^\ast$ are closed subspaces of the fibers of $T^*(G/M)$, which implies that the fibers of the complement $T^*(G/M)\setminus E_+^\ast$ are open cones. 
It follows that no $\xi\in T^*(G/M)\setminus E_+^\ast$ is an element of $\mathrm{WF}(s)$.
\end{proof}

\subsection{Passing to distributional sections of the boundary vector bundle}
We now consider distributional sections of the boundary vector bundle and their wave front sets\footnote{For the technical aspects, see Appendix \ref{sec:WF} with $G$ replaced by $K$, taking into account that in contrast to the measure $d(gM)$ on $G/M$ the measure $d(kM)$ on $K/M$ is not $G$-invariant.}. 
Our goal consists in establishing a one-to-one correspondence between the elements of the sets $\mathcal{R}(0)$ from (\ref{eq:R0}) and the distributional sections in $\D'(K/M,\V^{\mathcal{B}}_{\tau})$. 
\begin{lem}\label{lem:extcont}
Let $\iota:K/M\to G/M$ be the inclusion and identify $B^\ast\V_{\tau}^{\mathcal{B}}$ with $\V_\tau$ via the $G$-equivariant isomorphism (\ref{eq:bundleiso}). Then $B$ and $\iota$ induce  well-defined continuous pullback maps
 \[
 B^\ast:\D'(K/M,\V^{\mathcal{B}}_{\tau})\longrightarrow \mathcal{R}(0),\qquad
\iota^\ast:\mathcal{R}(0)\longrightarrow \D'(K/M,\V^{\mathcal{B}}_{\tau}).
\]
\end{lem}
\begin{proof}First, we check that the map $B^\ast$ is well-defined on smooth sections (i.e.,\ its image is contained in the specified target space). To this end we compute for $s\in \Gamma^\infty(\V^{\mathcal{B}}_{\tau})$, regarded as a right-$M$-equivariant function $\bar s:K\to V$, and $\X\in \Gamma^\infty(E_0\oplus E_-)$
\bq
\nabla(B^\ast s)(\X)=\nabla_\X(s\circ k^-)=\frac{d}{dt}\Big|_{t=0}\bar s\Big(k^-\big(g\underbrace{e^{t\X(g)}}_{\in AN^-}\big)\Big)=\frac{d}{dt}\Big|_{t=0}\bar s\big(k^-(g))\big)=0.\label{eq:vanish11}
\eq
Being the projections onto the factor $K$ in the product manifold $G=K\times A\times N^-$, the maps $k^-$ are submersions. This implies that the induced maps $B:G/M\to K/M$ are also submersions. Therefore, the pullbacks $B^\ast$ have unique continuous extensions to $\D'(K/M,\V^{\mathcal{B}}_{\tau})$, see \cite[Thm.\ 6.1.2]{hoermanderI}. As $\Gamma^\infty({\V^{\mathcal{B}}_{\tau}})$ is dense in $\D'(K/M,\V^{\mathcal{B}}_{\tau})$, and $\mathcal{R}(0)\subset \D'(G/M,\V_{\tau})$ is closed (being the intersection of inverse images of $0$ of continuous maps), we conclude by continuity that the image of the extension of $B^\ast$ to $\D'(K/M,\V^{\mathcal{B}}_{\tau})$ is still contained in $\mathcal{R}(0)$.  To see the existence of the pullback $\iota^\ast$ we observe that the conormal bundle of $K/M$ in $T^\ast(G/M)$ intersects $E^\ast_-$ only in the zero section. By Lemma \ref{lem:WF} the wave front sets of distributions in $\mathcal{R}(0)$ are contained in $E^\ast_+$, therefore they are disjoint from the conormal bundle of $K/M$ in $T^\ast(G/M)$ and \cite[Corollary 8.2.7]{hoermanderI} implies the existence of the pullbacks $\iota^\ast$. The map $\iota^\ast$ is continuous on the domain $\mathcal{R}(0)$ when the latter is equipped with the subspace topology induced by any of the spaces $\D'_{\Gamma}(G/M,\V_\tau)$ from \cite[Definition 8.2.2]{hoermanderI} for a fiber-wise conical subset $\Gamma\subset T^\ast(G/M)$ with $\Gamma\supset E^\ast_+$ and disjoint from the conormal bundle of $K/M$ in $T^\ast(G/M)$. Let us call that topology on $\mathcal{R}(0)$ for a choice of $\Gamma$ the $\Gamma$\emph{-topology}. The corresponding target space $\D'_{\iota^\ast\Gamma}(K/M,\V^{\mathcal{B}}_{\tau})\subset \D'(K/M,\V^{\mathcal{B}}_{\tau})$, as defined in \cite[(8.2.4)]{hoermanderI}, agrees for any $\Gamma\supset E^\ast_+$ with $\D'(K/M,\V^{\mathcal{B}}_{\tau})$ because one has $\iota^\ast(E_+^\ast)= T^\ast(K/M)$.  For a general $\Gamma\supset E^\ast_+$, the $\Gamma$-topology on $\mathcal{R}(0)$ might be stronger than the standard subspace topology on $\mathcal{R}(0)$ induced from $\D'(G/M,\V_\tau)$. However, when we choose $\Gamma$ so that $E^\ast_+$ is contained in the interior of $\Gamma$ in $T^\ast(G/M)$, then the $\Gamma$-topology and the standard subspace topology on $\mathcal{R}(0)$ agree. To see this, choose such a $\Gamma$. By definition all distributional sections in $\mathcal{R}(0)$ are not only smooth but even constant in the directions $E_0\oplus E_-$. The annihilator of this bundle is precisely $E^\ast_+$, which is contained in the interior of $\Gamma$ and therefore has no touching point, except at the zero section, with any closed cone that is disjoint with $\Gamma$. Therefore, we can use local partial integration with respect to the directions corresponding to $E_0\oplus E_-$ to show that any sequence $\{u_n\}_{n\in\N}\subset \mathcal{R}(0)$ fulfills locally property (ii)' in \cite[Definition 8.2.2]{hoermanderI}, which means that the condition (ii)'  is vacuous and the two topologies on $\mathcal{R}(0)$ agree. 
\end{proof}
\begin{rem}\label{rem:actiononsections2}
We call the $G$-action on $\D'(K/M,\V^{\mathcal{B}}_{\tau})$ obtained by dualizing the $G$-action (\ref{eq:Gactionstotalsections1}) on $\Gamma^\infty(\V^{\mathcal{B}}_{\tau})$ again simply the \emph{left regular $G$-action}.
\end{rem}
With these preparations, we can now state the main result of this section.
\begin{prop}\label{prop:pullbackiso}
The pullback maps from Lemma \ref{lem:extcont} form a $G$-equivariant isomorphism of topological vector spaces
$$\mathcal{R}(0)\stackrel[B^\ast]{\iota^\ast}{\longrightleftarrows} \D'(K/M,\V^{\mathcal{B}}_{\tau}),$$
where $\mathcal{R}(0)$ and $\D'(K/M,\V^{\mathcal{B}}_{\tau})$ are each equipped with the left regular $G$-action as introduced in Remarks \ref{rem:Gactionsrestr} and \ref{rem:actiononsections2}.
\end{prop}
\begin{proof}
The $G$-equivariance of $B^\ast$ follows from the $G$-equivariance of the isomorphism (\ref{eq:bundleiso}).  Since $B\circ \iota=\mathrm{id}_{K/M}$, we immediately get $$\iota^\ast\circ B^\ast=\mathrm{id}_{\D'(K/M,\V^{\mathcal{B}}_{\tau})}.$$
To consider the reversed composition, we observe that by definition of $\mathcal{R}(0)$ and $\nabla$ we have
\bq
\bar s|_{Kan}=\bar s|_K\qquad \forall\;an\in AN^-,\;s\in \mathcal{R}(0),\label{eq:vanishing}
\eq
where $\bar s \in \D'(G,V)$ is the lift of $s$ to an $M$-invariant distribution on $G$ with values in $V$, and the restrictions of the distributions $\bar s$ to the sets $Kan\subset G$ are well-defined due to the wave front set condition involved in the definition of $\mathcal{R}(0)$. Now, by the opposite  Iwasawa decomposition, $\{Kan\}_{an\in AN^-}$ covers all of $G$, hence $\{Kan M\}_{an\in AN^-}$ covers all of $G/M$, which means that one has for every $s\in \mathcal{R}(0)$
\[
B^\ast\circ\iota^\ast s=s\iff (B^\ast\circ\iota^\ast s)|_{Kan M}=s|_{Kan M}\quad \forall\; an\in AN^-.
\]
Here, analogously as above for the lifts $\bar s$, the restrictions of the distributions $s$ to the sets $KanM\subset G/M$ are well-defined. As we already know that $B^\ast\circ\iota^\ast$ maps into $\mathcal{R}(0)$, we deduce from \eqref{eq:vanishing} for every $s\in \mathcal{R}(0)$
\bq
B^\ast\circ\iota^\ast s=s\iff (B^\ast\circ\iota^\ast s)|_{K}=s|_{K}.\label{eq:38208950}
\eq
It now suffices to note that the composition $\iota\circ B:G/M\to G/M$, $gM\mapsto k^-(g)M$, restricts to the identity map on $K\subset G/M$, which implies
\[
B^\ast\circ\iota^\ast s|_{K}=\mathrm{id}_K^\ast s|_{K}=s_K\qquad \forall\;s\in \mathcal{R}(0),
\]
and we get with \eqref{eq:38208950} the desired result $B^\ast\circ\iota^\ast =\mathrm{id}_{\mathcal{R}(0)}$. The $G$-equivariance of $\iota^\ast$ now also follows, for it is the inverse of the $G$-equivariant map $B^\ast$.
\end{proof}
\begin{cor}\label{cor:dense}
The subspace $\mathcal{R}(\lambda)\cap\Gamma^\infty(\V_\tau)$ is dense in $\mathcal{R}(\lambda)$ for each $\lambda\in \aL_\C^\ast$.
\end{cor}
\begin{proof}
Since multiplication by $\Phi^{-\lambda(H_0)}$ maps $\mathcal{R}(\lambda)\cap\Gamma^\infty(\V_\tau)$ to $\mathcal{R}(0)\cap\Gamma^\infty(\V_\tau)$, it suffices by Lemma \ref{lem:homeo293509} to prove the statement for $\lambda=0$. However, the statement for $\lambda=0$ follows from Proposition \ref{prop:pullbackiso}. Indeed, $\Gamma^\infty(\V^{\mathcal{B}}_{\tau})$ is dense in $\D'(K/M,\V^{\mathcal{B}}_{\tau})$, and the linear homeomorphism $B^\ast$ maps smooth sections to smooth sections, so that $B^\ast(\Gamma^\infty(\V^{\mathcal{B}}_{\tau}))\subset \mathcal{R}(0)\cap\Gamma^\infty(\V_\tau)$ is dense in $\mathcal{R}(0)$.
\end{proof}
\subsection{Pollicott-Ruelle resonances and resonant states}
In order to define Pollicott-Ruelle resonances and related objects, we first pass to compact quotients of the spaces $G/K$ and $G/M$.

\subsubsection{Transition to a compact locally symmetric space}
Let $\Gamma\subset G$ be a cocompact torsion-free discrete subgroup. Due to the equivariance of the projection of the homogeneous vector bundle $\V_{\tau}$ with respect to the $G$-action (\ref{eq:GactionV}) we get an induced associated vector bundle $\Gamma\backslash\V_{\tau}$ over $\Gamma\backslash G/M$ with total space $\Gamma\backslash(G\times_{\tau}V)$. The situation is similar for the bundles $E_0$, $E_\pm$, leading to induced bundles $\Gamma\backslash E_0$ and $\Gamma\backslash E_\pm$ over $\Gamma\backslash G/M$. Identifying $T(\Gamma\backslash G/M)$ with $\Gamma\backslash T(G/M)$ we obtain the splitting
\[
T(\Gamma\backslash G/M)=\Gamma\backslash (E_0\oplus E_+\oplus E_-)=\Gamma\backslash E_0\oplus \Gamma\backslash E_+\oplus \Gamma\backslash E_-,
\]
and analogously for $T^\ast(\Gamma\backslash G/M)$. A vector field $\X$ on $\Gamma\backslash G/M$ can be regarded as a left-$\Gamma$-invariant, right-$M$-equivariant smooth function $\bar \X\in \Cinft(G,\nL^+\oplus \aL\oplus {\nL^-})$, and from (\ref{eq:covariantderiv}) it is easy to see that $\nabla_\X$ is left-$\Gamma$-equivariant for left-$\Gamma$-invariant vector fields $\X$. Thus, $\nabla$ induces a connection ${_\Gamma}\nabla$ on $\Gamma\backslash\V_{\tau}$, with associated covariant derivatives ${_\Gamma}\nabla_{\X}$, where $\X\in \Gamma^\infty(T(\Gamma\backslash G/M))$, and the covariant derivatives ${_\Gamma}\nabla_{\X}$ extend by duality to operators on distributional sections of $\Gamma\backslash\V_{\tau}$. 
We obtain this way the induced lifted geodesic vector field\footnote{Recall that the geodesic vector field $\mathfrak{X}_{H_0}\in \Gamma^\infty(T(G/M))$ corresponds to the constant function $\overline{\mathfrak{X}}_{H_0}\in\Cinft(G,\nL^+\oplus \aL\oplus {\nL^-})$ with value $H_0$ which is trivially a left-$\Gamma$-invariant function.}
 \[
{_\Gamma\Xbf}:\D'(\Gamma\backslash G/M,\Gamma\backslash \V_{\tau})\to \D'(\Gamma\backslash G/M,\Gamma\backslash \V_{\tau}).
\]
\begin{definition}\label{def:quotientR}By analogy to Definition \ref{def:repruelle1stband}, we define  on the quotient $\gam G/M$
\bqn
{_\Gamma}\mathcal{R}(\lambda):=\big\{s\in \D'(\gam G/M,\gam\V_{\tau}):\big({_\Gamma\Xbf} + \lambda(H)\big)(s)=0,\;{_\Gamma}\nabla_{\X}(s)=0\;\forall\;\X\in \Gamma^\infty(\Gamma\backslash E_-)\big\}.
\eqn
\end{definition}
We then have
\bq
 {_\Gamma}\mathcal{R}(\lambda)\cong {^\Gamma}\mathcal{R}(\lambda):=\{s\in \mathcal{R}(\lambda): \gamma s=s\;\forall\; \gamma\in \Gamma\},\label{eq:quotientR}
\eq
the isomorphism being defined by considering $\Gamma$-invariant distributional sections of $\V_{\tau}$ as distributional sections of $\gam\V_{\tau}$, and vice versa. 
\subsubsection{Definition of Pollicott-Ruelle resonances on associated vector bundles, (generalized) resonant states, and the \emph{first band}}\label{sec:defruelle}
\begin{definition}\label{def:ruelle1stband}
We call $\lambda \in \C$ a \emph{Pollicott-Ruelle resonance on $\gam\V_\tau$} if the set
\begin{align*}
\mathrm{Res}_{\,\Gamma\backslash\V_{\tau}}(\lambda)&:=\{s\in \D'(\Gamma\backslash G/M,\Gamma\backslash \V_{\tau}):({_\Gamma}\Xbf + \lambda)(s)=0,\;\mathrm{WF}(s)\subset \Gamma\backslash E_+^\ast\}
\end{align*}
of \emph{resonant states on $\gam\V_{\tau}$ of the resonance $\lambda$} is non-trivial, i.e., contains non-zero elements.  We say that  
a Pollicott-Ruelle resonance $\lambda$ on $\gam\V_\tau$ is \emph{in the first band} or a \emph{first band resonance} if the set
\[
\mathrm{Res}^{0}_{\,\gam\V_{\tau}}(\lambda):=\{s\in \mathrm{Res}_{\,\Gamma\backslash\V_{\tau}}(\lambda):{_\Gamma}\nabla_{\X}(s)=0\;\forall\;\X\in \Gamma^\infty(\Gamma\backslash E_-)\}
\]
of \emph{first band resonant states of the resonance $\lambda$} is non-trivial\footnote{Caution: not all resonant states of a first band resonance have to be first band resonant states.}. For $k\in \{1,2,3,\ldots\}$ we define
\begin{align*}
\mathrm{Res}_{\,\Gamma\backslash\V_{\tau}}\hspace*{-0.5em}{^k}(\lambda)&:=\{s\in \D'(\Gamma\backslash G/M,\Gamma\backslash \V_{\tau}):({_\Gamma}\Xbf + \lambda)^k(s)=0,\;\mathrm{WF}(s)\subset \Gamma\backslash E_+^\ast\},
\end{align*}
the sets of \emph{generalized\footnote{Note that $\mathrm{Res}_{\,\Gamma\backslash\V_{\tau}}\hspace*{-0.5em}{^1}(\lambda)=\mathrm{Res}_{\,\Gamma\backslash\V_{\tau}}(\lambda)$, so generalized resonant states of rank $1$ are just resonant states.} resonant states of rank $k$}, and the sets
\begin{align*}
\mathrm{Res}^{0}_{\,\Gamma\backslash\V_{\tau}}\hspace*{-0.25em}{^k}(\lambda)&:=\{s\in \mathrm{Res}_{\,\Gamma\backslash\V_{\tau}}\hspace*{-0.5em}{^k}(\lambda):{_\Gamma}\nabla_{\X}(s)=0\;\forall\;\X\in \Gamma^\infty(\Gamma\backslash E_-)\}
\end{align*}
of \emph{generalized first band resonant states of rank $k$}. We say that there is a \emph{Jordan block} of size $j\in\{1,2,3,\ldots\}$ for the resonance $\lambda$ with a \emph{Jordan basis} formed by generalized resonant states $s_1,\ldots,s_j$ if one has
\[
s_k\in \mathrm{Res}_{\,\Gamma\backslash\V_{\tau}}\hspace*{-0.5em}{^k}(\lambda)\setminus\{0\}\quad\forall\;k\in\{1,\ldots,j\},\qquad ({_\Gamma}\Xbf +\lambda)(s_{k})=s_{k-1}\quad\forall\;k\in\{2,\ldots,j\}.
\]
A \emph{first band Jordan block} is a Jordan block as above with
$s_k\in \mathrm{Res}^{0}_{\,\Gamma\backslash\V_{\tau}}\hspace*{-0.25em}{^k}(\lambda)\setminus\{0\}$ for all $k$. 
A Jordan block of size $j$ is \emph{non-trivial} if $j\geq 2$. 
\end{definition}
If $s$ is a distributional section of the bundle $\gam\V_{\tau}$ and we identify it with a $\Gamma$-equivariant distributional section $\tilde s$ of $\V_{\tau}$,  one has $\mathrm{WF}(s)=\gam \mathrm{WF}(\tilde s)$. Thus, it follows from Lemma \ref{lem:WF} and the $\Gamma$-invariance of the operators $\Xbf $ and $\nabla_{\X}$, $X\in \Gamma^\infty(E_\pm)$, that the wave front set properties required in the definition of Pollicott-Ruelle resonant states are automatically fulfilled for resonant states in the first band, leading to
\begin{rem}\label{rem:ident23590332532636}
With the element $H_0\in \aL$ from \eqref{eq:H0}, one has for each $\lambda\in\aL_\C^\ast$
\bqn
\mathrm{Res}^{0}_{\,\gam\V_{\tau}}(\lambda(H_0))\cong{^\Gamma}\mathcal{R}(\lambda),
\eqn
where the isomorphism is established by considering a distributional section of the bundle $\gam\V_{\tau}$ as a $\Gamma$-equivariant distributional section of $\V_{\tau}$.
\end{rem}
Let us finally mention two easy yet important naturality properties. 
\begin{lem}\label{lem:naturality}If $\tau,\tau'$ are two equivalent complex $M$-representations, then there are for every $\lambda \in \C,\;k\in \N$ linear isomorphisms
\begin{align*}
\mathrm{Res}_{\,\Gamma\backslash\V_{\tau}}(\lambda)&\cong\mathrm{Res}_{\,\Gamma\backslash\V_{\tau'}}(\lambda), & \mathrm{Res}_{\,\Gamma\backslash\V_{\tau}}\hspace*{-0.5em}{^k}(\lambda)&\cong \mathrm{Res}_{\,\Gamma\backslash\V_{\tau'}}\hspace*{-0.5em}{^k}(\lambda),\\
\mathrm{Res}^{0}_{\,\gam\V_{\tau}}(\lambda)&\cong\mathrm{Res}^{0}_{\,\gam\V_{\tau'}}(\lambda), & \mathrm{Res}^{0}_{\,\Gamma\backslash\V_{\tau}}\hspace*{-0.25em}{^k}(\lambda)&\cong \mathrm{Res}^{0}_{\,\Gamma\backslash\V_{\tau'}}\hspace*{-0.25em}{^k}(\lambda).
\end{align*}
\end{lem}
\begin{proof}
From the above definition, it is obvious that if there is a vector bundle isomorphism $\V_{\tau}\to\V_{\tau'}$, compatible with the canonical connections $\nabla$ and $\nabla'$ on $\V_{\tau}$ and $\V_{\tau'}$, respectively, then the claim holds. It now suffices to observe that an intertwining operator  establishing an equivalence between $\tau$ and $\tau'$ induces precisely such a vector bundle isomorphism $\V_{\tau}\to\V_{\tau'}$.
\end{proof}
\begin{lem}\label{lem:naturality2}If $\tau,\tau'$ are two complex $M$-representations and $\tau\oplus \tau'$ is their direct sum, then there are for every $\lambda \in \C,\;k\in \N$ linear isomorphisms
\begin{alignat*}{5}
\mathrm{Res}_{\,\Gamma\backslash\V_{\tau\oplus\tau'}}(\lambda)&\cong\mathrm{Res}_{\,\Gamma\backslash\V_{\tau}\oplus \Gamma\backslash\V_{\tau'}}(\lambda)&&\cong\mathrm{Res}_{\,\Gamma\backslash\V_{\tau}}(\lambda)&&\oplus\mathrm{Res}_{\,\Gamma\backslash\V_{\tau'}}(\lambda),\\
\mathrm{Res}^{0}_{\,\gam\V_{\tau\oplus\tau'}}(\lambda)&\cong\mathrm{Res}^{0}_{\Gamma\backslash\V_{\tau}\oplus \Gamma\backslash\V_{\tau'}}(\lambda)&&\cong\mathrm{Res}^{0}_{\,\gam\V_{\tau}}(\lambda)&&\oplus\mathrm{Res}^{0}_{\,\gam\V_{\tau'}}(\lambda),\\
\mathrm{Res}_{\,\Gamma\backslash\V_{\tau\oplus\tau'}}\hspace*{-0.5em}{^k}(\lambda)&\cong\mathrm{Res}_{\Gamma\backslash\V_{\tau}\oplus \Gamma\backslash\V_{\tau'}}\hspace*{-0.5em}{^k}(\lambda)&&\cong \mathrm{Res}_{\,\Gamma\backslash\V_{\tau}}\hspace*{-0.5em}{^k}(\lambda)&&\oplus\mathrm{Res}_{\,\Gamma\backslash\V_{\tau'}}\hspace*{-0.5em}{^k}(\lambda),\\
\mathrm{Res}^{0}_{\,\Gamma\backslash\V_{\tau\oplus\tau'}}\hspace*{-0.25em}{^k}(\lambda)&\cong \mathrm{Res}^{0}_{\Gamma\backslash\V_{\tau}\oplus \Gamma\backslash\V_{\tau'}}\hspace*{-0.25em}{^k}(\lambda) &&\cong \mathrm{Res}^{0}_{\,\Gamma\backslash\V_{\tau}}\hspace*{-0.25em}{^k}(\lambda)&&\oplus\mathrm{Res}^{0}_{\,\Gamma\backslash\V_{\tau'}}\hspace*{-0.25em}{^k}(\lambda).
\end{alignat*}
\end{lem}
\begin{proof}
It suffices to observe that the canonical connections on $\Gamma\backslash\V_{\tau\oplus\tau'}$ and $\Gamma\backslash\V_{\tau}\oplus \Gamma\backslash\V_{\tau'}$ are the sum of the individual canonical connections on $\Gamma\backslash\V_{\tau}$, $\Gamma\backslash\V_{\tau'}$ under the obvious necessary identifications.
\end{proof}
\subsection{Description of (generalized) resonant states by their transformation behavior at the boundary}\label{sec:princseries}
We would like to have a precise description of the image of the embedding
\bq
\mathrm{Res}^{0}_{\,\gam\V_{\tau}}(\lambda(H_0))\cong{^\Gamma}\mathcal{R}(\lambda)\stackrel{\text{incl.}}{\hooklongrightarrow}\mathcal{R}(\lambda)\mathrel{\mathop{\longrightarrow}^{\cdot\Phi^{-\lambda(H_0)}}_{\cong}}\mathcal{R}(0)\mathrel{\mathop{\longrightarrow}^{\iota^\ast}_{\cong}}\D'(K/M,\V^{\mathcal{B}}_{\tau})\label{eq:composite}
\eq
provided by Lemma \ref{lem:homeo293509} and Proposition \ref{prop:pullbackiso}. This would be a trivial problem if the isomorphism
\bq
Q_{\lambda}:\mathcal{R}(\lambda)\mathrel{\mathop{\longrightarrow}^{\cdot\Phi^{-\lambda(H_0)}}_{\cong}}\mathcal{R}(0)\mathrel{\mathop{\longrightarrow}^{\iota^\ast}_{\cong}}\D'(K/M,\V^{\mathcal{B}}_{\tau})\label{eq:composite2}
\eq
were $\Gamma$-equivariant with respect to the $G$-actions on $\mathcal{R}(\lambda)$ and $\D'(K/M,\V^{\mathcal{B}}_{\tau})$ that we introduced so far $-$ then, the image of \eqref{eq:composite} would simply consist of the $\Gamma$-invariant distributions in $\D'(K/M,\V^{\mathcal{B}}_{\tau})$. That is, however, not the case. While the final map $\iota^\ast$ is $G$-equivariant by Proposition \ref{prop:pullbackiso}, the preceding multiplication with $\Phi^{-\lambda(H_0)}$ is not $\Gamma$-equivariant because the function $\Phi$ is not $\Gamma$-invariant. Indeed, by \eqref{eq:defphi} and\footnote{Recall Remark \ref{rem:differentB}.} \cite[(3.12)]{GHW18b} we have for $g,\gamma\in G$
\bq
\bar \Phi(\gamma g)=e^{-\nu_0(H^-(\gamma k^-(g)))}\bar \Phi(g).\label{eq:Phitransf}
\eq
The idea is now to introduce a new $G$-action on $\D'(K/M,\V^{\mathcal{B}}_{\tau})$ to get equivariance of the map \eqref{eq:composite2}. This leads us to the study of so-called \emph{principal series representations}.
\subsubsection{Principal series representations}
As explained in \cite[Chapter 7]{knapp2}, the principal  series representations of the Lie group $G$ can be realized in three different ways or ``pictures''. As we look for a new $G$-action on $\D'(K/M,\V^{\mathcal{B}}_{\tau})$, we focus on the so-called \emph{compact picture}.  The principal series representation of $G$ associated to the $M$-representation $\tau$ and an element $\lambda \in \aL_\C^\ast$ acts on smooth sections by\footnote{We follow Olbrich's convention as in \cite[between Satz 2.8 and Satz 2.9]{olbrichdiss}. In \cite[p.\ 169]{knapp2}, the definition differs from ours in such a way that $\lambda$ is replaced by $-\lambda$.}
\bq
\pi^{\tau,\lambda}_\mathrm{comp}(g)(s)(kM):=e^{(\lambda + \rho)(H^-(g^{-1}k))}(gs)(kM)\label{eq:compactp},\quad s\in \Gamma^\infty(\V^\B_\tau),\; kM\in K/M,
\eq
and extends by continuity\footnote{Note that one could alternatively extend the definition by duality to distributional sections. One would then arrive at a definition which differs from our continuity extension in such a way that $\lambda$ is replaced by $-\lambda$ because of the transformation behavior of the measure $d(kM)$ on $K/M$ which is not $G$-invariant.} to a representation $\pi^{\tau,\lambda}_\mathrm{comp}:G\to \mathrm{End}(\D'(K/M,\V^{\mathcal{B}}_{\tau}))$. Here the element $\rho$ was defined in \eqref{eq:rho}, and $gs$ refers to the action \eqref{eq:Gactionsectionslanglands}. The significance of the principal series representations for our purposes becomes clear from 
\begin{lem}\label{lem:ruellerep}
For each $\lambda\in \aL^\ast_\C$, the isomorphism $Q_{\lambda}$ from \eqref{eq:composite2} is $G$-equivariant when $\mathcal{R}(\lambda)$ carries the left regular $G$-action and $\D'(K/M,\V^{\mathcal{B}}_{\tau})$ is equipped with the  principal series representation $\pi^{\tau,-\lambda-\rho}_\mathrm{comp}$. Consequently, the embedding \eqref{eq:composite} induces an isomorphism
\bq
{^\Gamma}Q_{\lambda}:\mathrm{Res}^{0}_{\,\gam\V_{\tau}}(\lambda(H_0))\stackrel{\cong}{\longrightarrow}{^\Gamma}\mathrm{H}^{-\infty}_{\tau,-\lambda-\rho},\label{eq:ruellerepiso}
\eq
where ${^\Gamma}\mathrm{H}^{-\infty}_{\tau,-\lambda-\rho}\subset \D'(K/M,\V^{\mathcal{B}}_{\tau})$ denotes the $\Gamma$-invariant distributional vectors in the principal series representation $\pi^{\tau,-\lambda-\rho}_\mathrm{comp}$.
\end{lem}
\begin{proof}
By definition of $\Phi$ one has
\bq
\bar \Phi(k)=1\qquad \forall\;k\in K,\label{eq:1k}
\eq
and \eqref{eq:vanishing} then yields for $s\in \mathcal{R}(\lambda)$
\begin{align*}
\bar \Phi(g)^{-\lambda(H_0)}\bar s(g)&=\big(\bar \Phi^{-\lambda(H_0)}\bar s\big)(g)=\big(\bar \Phi^{-\lambda(H_0)}\bar s\big)(k^-(g))=\underbrace{\bar \Phi(k^-(g))^{-\lambda(H_0)}}_{=1}\bar s(k^-(g))\\
&=\bar s(k^-(g))\qquad\forall\; g\in G.
\end{align*}
In particular, we get
\begin{align}\begin{split}
\bar \Phi(\gamma^{-1}k)^{-\lambda(H_0)}\bar s(\gamma^{-1}k)
&=\bar s(k^-(\gamma^{-1}k))\\
&=\underbrace{\Phi(k^-(\gamma^{-1}k))^{-\lambda(H_0)}}_{=1}\bar s(k^-(\gamma^{-1}k))\\
&\equiv \gamma\Big(\iota^\ast\big(\Phi^{-\lambda(H_0)}s\big)\Big)(kM)\qquad\forall\;k\in K,\;\gamma \in G.\end{split}\label{eq:inv2K}
\end{align}
With the relations \eqref{eq:Phitransf}, \eqref{eq:1k}, and \eqref{eq:inv2K} we compute for $s\in \mathcal{R}(\lambda)$, $\gamma\in G$, $k\in K$
\begin{align*}
\overline{\iota^\ast\big(\Phi^{-\lambda(H_0)}(\gamma s)\big)}(k)&=\underbrace{\bar\Phi(k)^{-\lambda(H_0)}}_{=1}\bar s(\gamma^{-1}k)\\
&=\bar \Phi(\gamma^{-1}k)^{\lambda(H_0)}\,\overline{\gamma\big(\iota^\ast\big(\Phi^{-\lambda(H_0)}s\big)\big)}(k)\\
&=\Big(e^{-\nu_0(H^-(\gamma^{-1} k))}\underbrace{\bar \Phi(k)}_{=1}\Big)^{\lambda(H_0)}\,\overline{\gamma\big(\iota^\ast\big(\Phi^{-\lambda(H_0)}s\big)\big)}(k)\\
&=e^{-\lambda(H_0)\nu_0(H^-(\gamma^{-1} k))}\,\overline{\gamma\big(\iota^\ast\big(\Phi^{-\lambda(H_0)}s\big)\big)}(k)\\
&=e^{-\lambda(H^-(\gamma^{-1} k))}\,\overline{\gamma\big(\iota^\ast\big(\Phi^{-\lambda(H_0)}s\big)\big)}(k),
\end{align*}
where we used the definition of $\nu_0\in \aL^\ast$ as the element dual to $H_0\in \aL$ in the last step. Recalling Remark \ref{rem:ident23590332532636}, the proof is finished.
\end{proof}
\subsubsection{Characterization of first band Jordan blocks}
Just as Lemma \ref{lem:ruellerep} establishes a one-to-one correspondence between first band Pollicott-Ruelle resonant states and those distributional sections of the boundary vector bundle $\V^{\mathcal{B}}_{\tau}$ that transform under the action of the group $\Gamma$ according to a principal series representation, there is a description of \emph{generalized} first band Pollicott-Ruelle resonant states in terms of distributional sections of $\V^{\mathcal{B}}_{\tau}$.
\begin{prop}\label{prop:jordanchar}For a Pollicott-Ruelle resonance $\lambda\in \C$ in the first band and generalized first band resonant states  $s_0,\ldots,s_{j-1}\in\bigcup_{k\geq 1}\mathrm{Res}^{0}_{\,\Gamma\backslash\V_{\tau}}\hspace*{-0.25em}{^k}(\lambda)\setminus\{0\}$, $j\in \{1,2,\ldots\}$, the following two statements are equivalent.\begin{enumerate}[leftmargin=*]
\item The family $s_0,\ldots,s_{j-1}$ forms a Jordan basis for a first band Jordan block. 
\item The distributional sections $\omega_0,\ldots,\omega_{j-1}\in \D'(K/M,\V^{\mathcal{B}}_{\tau})$ defined by
\[
\omega_k:= \iota^\ast\bigg(\Phi^{-\lambda}\sum_{l=0}^{k}\frac{(\log \Phi)^{k-l}}{(k-l)!}\pi_\Gamma^\ast(s_l)\bigg),\quad k\in \{0,\ldots,j-1\},
\]
where  $\pi_\Gamma:G/M\to \gam G/M$ is the canonical projection, $\Phi\in \Cinft(G/M)$ was defined in \eqref{eq:defphi}, and $\iota^\ast$ is as in Proposition \ref{prop:pullbackiso}, have the transformation property
\[
\gamma\, \omega_k=N_{\gamma}^{-\lambda}\sum_{l=0}^k\frac{(\log N_{\gamma})^{k-l}}{(k-l)!}\omega_l\qquad \forall\; k\in \{1,\ldots,j-1\},\;\gamma\in \Gamma,
\]
where $\gamma$ acts by the principal series representation \eqref{eq:compactp} and $N_{\gamma}\in \Cinft(K/M)$ is induced by the $M$-invariant function $\bar N_{\gamma}\in \Cinft(K)$ defined by
\bq
\bar N_{\gamma}(k)=e^{-\nu_0(H^-(\gamma k))}\label{eq:Ndef}.
\eq
\end{enumerate}
\end{prop}
\begin{proof}Recalling \eqref{eq:Phitransf} and taking into account Remark \ref{rem:differentB}, the proof is completely analogous to \cite[proof of Proposition 3.6]{GHW18b}.
\end{proof}

\section{First band resonant states and generalized common eigen\-spaces}\label{sec:fbpoisson}
In this section we establish a correspondence between first band Pollicott-Ruelle resonances and certain generalized eigenvalues, as well as between the corresponding first band resonant states and generalized eigensections.  The main technical tool is a vector-valued Poisson transform developed  by Olbrich in his dissertation \cite{olbrichdiss}. 
From now on we suppose that the $M$-representation $\tau$ is irreducible.  Let $\sigma:K\to \mathrm{End}(W)$ be an irreducible unitary representation of the group $K$ on a finite-dimensional Hermitian vector space $W$ such that $[\tau]$ occurs 
in the restriction $\sigma|_M$, and let $\W_\sigma$ be the vector bundle over $G/K$ associated to $\sigma$. For a section $s\in \Gamma^\infty(\W_\sigma)$, let us write $\bar s$ for the corresponding right-$K$-equivariant function $G\to W$. Furthermore, in the following we will make use of the action of $K$ on $\mathrm{End}(W)$ given by
\bq
kE:= \sigma(k)\circ E \circ \sigma(k)^{-1},\qquad k\in K,\; E\in \mathrm{End}(W)\label{eq:endaction}
\eq
and the action of the Weyl group $\mathscr{W}$ on $\mathrm{End}_M(W)$ given by
\bq
wE:=m'_w E,\qquad w=m'_wM\in \mathscr{W},\;E\in \mathrm{End}_M(W).\label{eq:endWaction}
\eq
Finally, as in the previous Section, $\Gamma\subset G$ will denote a cocompact torsion-free discrete subgroup.
\subsection{The Poisson transform and generalized common eigen\-spaces}
Before we introduce Olbrich's vector-valued Poisson transform, let us briefly compare it to other available vector-valued Poisson transforms and provide  some references for background information.
\subsubsection{Comparison between Olbrich's and other vector-valued Poisson transforms}\label{sec:poissonadv}
A scalar Poisson transform maps distributions (or hyperfunctions) on $K/M$ to functions on $G/K$ that are joint eigenfunctions of the algebra of invariant differential operators. Broadly and informally speaking, a vector-valued Poisson transform relates distributional sections of the boundary vector bundle $\V^{\mathcal{B}}_{\tau}$ over $K/M$ to sections of a bundle of the form $\W_\sigma$ over $G/K$ which are eigensections of invariant differential operators on $\W_\sigma$. There exist several different vector-valued Poisson transforms in the literature that differ in the generality of the allowed vector bundles $\V^{\mathcal{B}}_{\tau}$ and also in the question which differential operators on $\W_\sigma$ have to be considered. We want to give a short overview and explain why we work with Olbrich's Poisson transform. 

As we are interested in general vector-valued Pollicott-Ruelle resonant states of the geodesic flow our initial focus lies on the vector bundle $\V_\tau$ over $G/M$ and the corresponding boundary vector bundle $\V^{\mathcal{B}}_{\tau}$ over $K/M$. We therefore want to start with an arbitrary irreducible unitary $M$-representation $\tau$ (indeed any associated vector bundle can be decomposed into a sum of irreducibles).
This is precisely the setting of Olbrich's vector-valued Poisson transform.  The drawback of this level of generality is that one has to generalize the concept of a joint eigenspace of a commutative family of differential operators to a certain \emph{generalized common eigenspace} of the algebra $D(G,\sigma)$ of invariant differential operators on $\W_\sigma$. These generalized common eigenspaces do not consist of true joint eigensections in the usual sense unless one makes additional assumptions, such as those of Section \ref{sec:specialcase}. 

There are other vector-valued Poisson transforms  that achieve to avoid this difficulty at the cost of working with less general bundles $\V^{\mathcal{B}}_{\tau}$: 

For $G$ of arbitrary real rank and a fixed irreducible unitary $K$-representation $\sigma$, Yang \cite{yang} extends the classical results of  Helgason \cite{helgason70a,helgason76} (see also \cite[II.4.4.D]{helgason84}) and Kashiwara et al.\ \cite{kashiwara78} from the scalar to the vector-valued setting, with a Poisson transform that maps distributional sections of a boundary vector bundle bijectively onto the space of joint eigensections of the commutative subalgebra $\mathscr{D}(Z(\mathcal{U}(\g)))\subset D(G,\sigma)$ arising from the center of the universal enveloping algebra of $\g$, cf.\ Section \ref{sec:studysmb}. In particular, Yang avoids generalizing the notion of a common eigenspace.  His approach is in some sense contrary to ours as he starts with a $K$-representation $\sigma$ and then considers boundary vector bundles associated to the $M$-representations given for appropriate $\Lambda\in \g_\C^\ast$ by the direct sum of all irreducible $M$-representations occurring in $\sigma|_M$ with infinitesimal character $\Lambda$. 
We recommend the introduction of Yang's paper for further reading regarding the variety of available vector-valued Poisson transforms. 

The rank one case of Yang's setting is studied by van der Ven  \cite{vanderven}. For $\g$ of real rank one, i.e.,\ as in our article, the algebra $\mathscr{D}(Z(\mathcal{U}(\g)))$ is generated by the Casimir operator, which is geometrically the Bochner Laplacian associated to the canonical connection on $\W_\sigma$. Thus, the common eigenspaces reduce to eigenspaces of this operator.

Gaillard \cite{gaillard1,gaillard2} considers the interesting special case of differential forms over real and complex hyperbolic spaces of arbitrary dimension. He obtains very precise invertibility results for a Poisson transform acting on currents or hyperforms on the boundary at infinity.  

Finally, Dyatlov, Faure, and Guillarmou \cite{dfg} prove invertibility of a particular Poisson transform acting on distributional sections of certain\footnote{More precisely, the bundles resulting from restricting the tensor bundles occurring in Section \ref{sec:band} to $K/M$ in the special case where $G=\SO(n+1,1)_0$. In this case, the construction of the tensor bundles is less complicated as $\alpha_0$ is the only positive restricted root.} tensor bundles over the boundary at infinity of a real hyperbolic space of arbitrary dimension. As we already mentioned in the introduction, their results were a main motivation for the present article.

\subsubsection{Olbrich's vector-valued Poisson transform}

The Olbrich-Poisson transform, in the following simply called  \emph{Poisson transform}, associated to the irreducible unitary $M$-representation $\tau$, the compatible\footnote{Recall that this means that $[\tau]$ occurs in $\sigma|_M$.} irreducible unitary $K$-representation $\sigma$, and an element $\lambda \in \aL^\ast_\C$ can be defined as the map 
\bq
\mathbb{P}^{\sigma}_{\tau,\lambda}:\mathrm{Hom}_M(V,W)\otimes_\C\D'(K/M,\V^{\mathcal{B}}_{\tau})\to \Gamma^\infty(\W_\sigma)\label{eq:Poisson}
\eq
obtained by continuous\footnote{As already pointed out for the principal series representations, one could alternatively extend the definition by duality to distributional sections, which would flip the sign of $\lambda$ due to the transformation behavior of the measure $d(kM)$ on $K/M$.} extension of the map on $\mathrm{Hom}_M(V,W)\otimes_\C\Gamma^\infty(\V_\tau^\B)$ given by
\bq
\overline{\mathbb{P}^{\sigma}_{\tau,\lambda}(h\otimes \omega)}(g):=\int_K e^{(-\lambda+\rho)(H^-(g^{-1}k))}\sigma(k^-(g^{-1}k))h(\bar \omega(k))\d k\label{eq:Poissoneq}
\eq
and extended linearly. To understand the significance of the Poisson transform, we need to make a few additional definitions. First of all, note that $\Gamma^\infty(\W_\sigma)$ carries a left regular $G$-action defined analogously as the left regular $G$-action \eqref{eq:GactionVsections} on $\Gamma^\infty(\V_\tau)$.
The associative complex algebra with unit consisting of the $G$-invariant differential operators on smooth sections of $\W_\sigma$ is then defined as
\[
D(G,\sigma):=\big\{D:\Gamma^\infty(\W_\sigma)\to \Gamma^\infty(\W_\sigma)\;\text{linear differential operator}\;|\;D(gs)=gD(s)\;\forall\; g\in G\big\}.
\]
This algebra is non-commutative in general. By \cite[Folg.\ 2.5]{olbrichdiss}, it is a finitely generated algebra with at most $N=\dim \mathrm{End}_M(W)-1$ generators.
\begin{example}The (positive) Bochner Laplacian $\Delta$ associated to the canonical connection on $\W_\sigma$ over  the Riemannian symmetric space $G/K$ is always an element of $D(G,\sigma)$. \label{ex:bochner}
\end{example} The notion of a common eigenspace of a commuting family of differential operators is generalized in the following
\begin{definition}[{Generalized common eigenspace, {\cite[Def.\ 2.1]{olbrichdiss}}}]\label{def:eigenspaces}Let $\chi:D(G,\sigma)\to \mathrm{End}(U)$ be a finite-dimensional complex representation of the algebra $D(G,\sigma)$ on a complex vector space $U$. We define the generalized common eigenspace $E_\chi(\mathcal W_\sigma)\subset \Gamma^\infty(\mathcal W_\sigma)$ of $D(G,\sigma)$ associated to $\chi$ by
\begin{align*}
E_\chi(\mathcal W_\sigma):&=\{\mathscr{H}(u):u\in U,\;\mathscr{H}\in \mathrm{Hom}_{D(G,\sigma)}(U,\Gamma^\infty(\W_\sigma))\}\\
&=\{\mathscr{H}(u):u\in U,\;\mathscr{H}:U\to\Gamma^\infty(\W_\sigma)\;\mathrm{linear},\; D\mathscr{H}(u)=\mathscr{H}(\chi(D)u)\;\forall\;D\in D(G,\sigma)\}.
\end{align*}
\end{definition}
It is straightforward to check that for equivalent representations $\chi,\chi'$ of $D(G,\sigma)$ the resulting generalized eigenspaces are equal, i.e.,  $E_\chi(\mathcal W_\sigma)=E_{\chi'}(\mathcal W_\sigma)$.

\begin{rem}\label{rem:elemobs}The space $E_\chi(\W_\sigma)$ has obviously the following two properties:\begin{itemize}[leftmargin=*]
\item If for some operator $D\in D(G,\sigma)$ the endomorphism $\chi(D)$ is a scalar multiple of the identity, all sections in $E_\chi(\W_\sigma)$ are eigensections of the operator $D$.
\item If $\chi$ is a $1$-dimensional representation, all the endomorphisms $\chi(D)$ are scalar multiples of the identity, so the statement above applies to every $D\in D(G,\sigma)$. Consequently, the generalized common eigenspace $E_\chi(\W_\sigma)$ consists entirely of eigensections of all operators in $D(G,\sigma)$ in this case.\end{itemize}
\end{rem}

\subsubsection{The relevant representations}
The representations of $D(G,\sigma)$ to which we shall apply the above constructions are those considered in the following
\begin{prop}[{{\cite[Def.\ 2.10 and (9)]{olbrichdiss}}}]\label{prop:existencechi}
For each $\lambda\in \aL_\C^\ast$ there is an algebra representation
\bq
\chi_{\tau,\lambda}:D(G,\sigma)\to \mathrm{End}(\mathrm{Hom}_M(V,W))\label{eq:chirep}
\eq
that is irreducible if the principal series representation $\pi^{\tau,\lambda}_\mathrm{comp}$ is irreducible, and if $\tau'$ is an irreducible $M$-representation equivalent to $\tau$, then $\chi_{\tau',\lambda}$ is equivalent to $\chi_{\tau,\lambda}$. \qed
\end{prop}
The construction of $\chi_{\tau,\lambda}$ will be explicitly given in Lemma~\ref{lem:smbaction}. For the moment only its existence and not the precise form is important.
\begin{definition}\label{def:eigencomnot}
We shall use the notation
\begin{align}
\nonumber E_{[\tau],\lambda}(\W_\sigma)&:=E_{\chi_{\tau,\lambda}}(\W_\sigma),\\
{^\Gamma}E_{[\tau],\lambda}(\W_\sigma)&:=\{s\in E_{[\tau],\lambda}(\W_\sigma): \gamma s=s\;\forall\; \gamma\in \Gamma\},\label{eq:not3950}\\
E_{[\tau],\lambda}(\Gamma\backslash\W_\sigma)&:=\big\{s\in \Gamma^\infty(\Gamma\backslash\W_\sigma) \,|\; s\; \mathrm{is\; induced\; by\;} \tilde s\in {^\Gamma}E_{[\tau],\lambda}(\W_\sigma)\big\}.\label{eq:EpmGamma}
\end{align}
\end{definition}
A tight relation between generalized common eigenspaces associated to the representations $\chi_{\tau,\lambda}$ and the Poisson transform is given by the following result.
\begin{lem}[{{\cite[Lemma.\ 3.3]{olbrichdiss}}}]\label{lem:Poissonequivariant}
For every $\lambda\in \aL_\C^\ast$ one has the relation
\begin{multline*}
\mathbb{P}^{\sigma}_{\tau,\lambda}\big(\chi_{\tau,\lambda}(D)(h)\otimes \pi^{\tau,\lambda}_\mathrm{comp}(g)(\omega)\big)=D(g\, \mathbb{P}^{\sigma}_{\tau,\lambda}(h\otimes \omega))\\
 \forall\; h\in \mathrm{Hom}_M(V,W),\;\omega\in \D'(K/M,\V^{\mathcal{B}}_{\tau}),\;D\in D(G,\sigma),\;g\in G.
\end{multline*}
In particular, the image of the Poisson transform $\mathbb{P}^{\sigma}_{\tau,\lambda}$ is contained in $E_{[\tau],\lambda}(\W_\sigma)$. 
\qed
\end{lem}
As a direct consequence we get a statement complementing Lemma \ref{lem:ruellerep}.
\begin{cor}\label{cor:poisson1}
For each element $\lambda\in \aL_\C^\ast$, the Poisson transform $\mathbb{P}^{\sigma}_{\tau,\lambda}$ induces maps
\begin{align}
\mathcal{P}^{\sigma}_{\tau,\lambda}:\mathrm{Hom}_M(V,W)\otimes_\C\D'(K/M,\V^{\mathcal{B}}_{\tau})&\longrightarrow E_{[\tau],\lambda}(\W_\sigma)\label{eq:poissonnoGamma}\\
{^\Gamma}\mathcal{P}^{\sigma}_{\tau,\lambda}:\mathrm{Hom}_M(V,W)\otimes_\C{^\Gamma}\mathrm{H}^{-\infty}_{\tau,\lambda}&\longrightarrow {^\Gamma} E_{[\tau],\lambda}(\W_\sigma)\cong E_{[\tau],\lambda}(\Gamma\backslash\W_\sigma)\label{eq:poissonGamma}
\end{align}
which we also call Poisson transforms. Depending on the context, ${^\Gamma}\mathcal{P}^{\sigma}_{\tau,\lambda}$ will be regarded as a map with codomain ${^\Gamma} E_{[\tau],\lambda}(\W_\sigma)$ or $E_{[\tau],\lambda}(\Gamma\backslash\W_\sigma)$ in the following.  \qed
\end{cor}
A well-known basic fact in representation theory is
\begin{lem}\label{lem:homhom}
The dimension of the complex vector space $\mathrm{Hom}_M(V,W)$ is given by \[\dim_\C \mathrm{Hom}_M(V,W)=[\sigma|_M:\tau].\]
\end{lem}
\subsection{Natural pushforward maps}\label{sec:pushfwd}
Let $\Pi:G/M\to G/K$ and ${^\Gamma}\Pi:\gam G/M\to \gam G/K$ be the canonical projections. Every homomorphism $h\in \mathrm{Hom}_M(V,W)$ induces canonically vector bundle morphisms 
\begin{equation}\label{eq:homh}
h_\Pi:\Abb{ \V_\tau}{\W_\sigma,}{{[g,v]}}{{[g,h(v)],}} \quad\qquad ^\Gamma h_\Pi:\Abb{\gam \V_\tau}{ \gam \W_\sigma,}{\Gamma [g,v]}{\Gamma [g,h(v)].}
\end{equation}
\begin{definition}[Natural pullbacks and pushforwards]\label{def:pushforwards}Given $h\in \mathrm{Hom}_M(V,W)$, let $h_\Pi:\V_\tau \to \W_\sigma$ and  ${^\Gamma}h_{\Pi}:\gam\V_\tau \to \gam\W_\sigma$ be the canonically induced vector bundle morphisms as described above.\begin{itemize}[leftmargin=*]
\item For a section $t\in \Gamma^\infty(\W^\ast_\sigma)$, corresponding to a right-$K$-equivariant smooth function $\bar t: G\to W^\ast$, we define the \emph{pullback of $t$ along $h_\Pi$}  as the section $h_\Pi^\ast (t)\in\Gamma^\infty(\V^\ast_\tau)$ corresponding to the right-$M$-equivariant  function 
$
h^\ast\circ \bar t: G\to V^\ast,
$ 
where $h^\ast\in\mathrm{Hom}_M(W^\ast,V^\ast)$ is the adjoint of $h$.
\item Analogously, we define $^ \Gamma h_\Pi^\ast: \Gamma^\infty(\gam\W_\sigma^\ast)\to \Gamma(\gam\V_\tau^\ast)$.
\item By duality, we define the pushforward maps
\begin{align}\begin{split}
{h_\Pi}_\ast :  \D'(G/M,\V_\tau)&\to \D'(G/K,\W_\sigma),\\
{{^\Gamma}h_{\Pi}}_\ast: \D'(\gam G/M,\gam\V_\tau)&\to\D'(\gam G/K,\gam \W_\sigma).\label{eq:pushfwd}\end{split}
\end{align}
\end{itemize}

\end{definition}
With these preparations we can prove the main result of this section.
\begin{thm}[Pushforwards of first band resonant states are generalized eigensections]\label{thm:main}For each $\lambda\in \aL_\C^\ast$ and $h\in \mathrm{Hom}_M(V,W)$, the natural pushforward ${{^\Gamma}h_{\Pi}}_\ast$ bi-restricts to a linear map
\bq
{{^\Gamma}h_{\Pi}}_\ast: 
\mathrm{Res}^{0}_{\,\gam\V_{\tau}}(\lambda(H_0))\longrightarrow E_{[\tau],-\lambda-\rho}(\Gamma\backslash\W_\sigma)\label{eq:Possnat3}
\eq
which is injective iff the Poisson transform ${^\Gamma}\mathcal{P}^{\sigma}_{\tau,-\lambda-\rho}(h\otimes_\C\cdot)$ is injective. Moreover, the obtained linear map
\begin{align}\begin{split}
\mathrm{Hom}_M(V,W)\otimes_\C\mathrm{Res}^{0}_{\,\gam\V_{\tau}}(\lambda(H_0))&\longrightarrow E_{[\tau],-\lambda-\rho}(\Gamma\backslash\W_\sigma),\\
h\otimes s&\longmapsto {{^\Gamma}h_{\Pi}}_\ast(s),\label{eq:Possnat4444}\end{split}
\end{align}
is surjective iff the Poisson transform ${^\Gamma}\mathcal{P}^{\sigma}_{\tau,-\lambda-\rho}$ is surjective.
\end{thm}
\begin{rem}We will see in Theorem \ref{thm:main2} that all generalized eigensections on $\gam G/K$ are smooth, so Theorem \ref{thm:main} implies that the natural pushforwards of first band resonant states on $\gam\V_{\tau}$ are smooth sections of $\Gamma\backslash\W_\sigma$.
\end{rem}
\begin{proof}[Proof of Theorem \ref{thm:main}]
By Lemma \ref{lem:ruellerep}  and Corollary \ref{cor:poisson1}, it suffices to prove
\bqn
{^\Gamma}\mathcal{P}^{\sigma}_{\tau, -\lambda-\rho }(h\otimes {^\Gamma}Q_{\lambda}(s))={{^\Gamma}h_{\Pi}}_\ast(s)\qquad \forall\; s\in \mathrm{Res}^{0}_{\,\gam\V_{\tau}}(\lambda(H_0)).
\eqn
This identity is just the restriction of the statement
\bq
\mathcal{P}^{\sigma}_{\tau, -\lambda-\rho }(h\otimes Q_{\lambda}(s))={h_\Pi}_\ast(s)\qquad \forall\; s\in \mathcal{R}(\lambda)\label{eq:Possnat1}\\
\eq
 to the $\Gamma$-invariant elements in $\mathcal{R}(\lambda)$, therefore we only need to prove \eqref{eq:Possnat1}. To this end, we recall Corollary \ref{cor:dense} which says that $\Gamma^\infty(\V_\tau)\cap\mathcal{R}(\lambda)$ is dense in $\mathcal{R}(\lambda)$. As the maps $\mathcal{P}^{\sigma}_{\tau, -\lambda-\rho }$, $Q_{\lambda}$, and ${h_\Pi}_\ast$ are continuous, \eqref{eq:Possnat1} is equivalent to
\bq
\mathcal{P}^{\sigma}_{\tau, -\lambda-\rho }(h\otimes Q_{\lambda}(s))={h_\Pi}_\ast(s)\qquad \forall\; s\in \Gamma^\infty(\V_\tau)\cap\mathcal{R}(\lambda).\label{eq:Possnat4}
\eq
It remains to prove \eqref{eq:Possnat4}. Let $s\in \Gamma^\infty(\V_\tau)\cap\mathcal{R}(\lambda)$ and choose a test section $t\in \Gamma^\infty_c(\W_\sigma^\ast)$. With the defining relations  \eqref{eq:Poissoneq},  \eqref{eq:composite2}, and \eqref{eq:defphi}, we get
\begin{align*}
\mathcal{P}^{\sigma}_{\tau, -\lambda-\rho }(h\otimes Q_{\lambda}(s))(t)&=\int_{G/K}\eklm{\mathcal{P}^{\sigma}_{\tau, -\lambda-\rho }(h\otimes Q_{\lambda}(s))(gK),t(gK)}\d(gK)\\
&=\int_{G}\int_{K}e^{(\lambda+ 2\rho)(H^-(g^{-1}k))} \eklm{\sigma(k^-(g^{-1}k))h(\bar s(k)),\bar t(g)}\d k\d g.
\end{align*}
Here we used that $\overline{Q_{\lambda}(s)}(k)=s(k)$  for all $k\in K$. As $Q_{\lambda}(s)\in \mathcal{R}(0)$, we have by \eqref{eq:vanishing} and the identity $k^-(gk^-(g^{-1}k))=k$, which one checks easily using the opposite Iwasawa decomposition, the formula
\[
\bar s(k)=\bar \Phi(gk^-(g^{-1}k))^{-\lambda(H_0)}\bar s(gk^-(g^{-1}k))\qquad \forall\; k\in K.
\]
On the other hand, with the identity $H^-(gk^-(g^{-1}k))=-H^-(g^{-1}k)$ that one checks again easily using the opposite Iwasawa decomposition, we compute
\[
\bar \Phi(gk^-(g^{-1}k))^{-\lambda(H_0)}=e^{\lambda(H^-(gk^-(g^{-1}k)))}=e^{-\lambda(H^-(g^{-1}k))}.
\]
This yields
\begin{align*}
\mathcal{P}^{\sigma}_{\tau, -\lambda-\rho }(h\otimes Q_{\lambda}(s))(t)&=\int_{G}\int_{K}e^{2\rho(H^-(g^{-1}k))}\eklm{\sigma(k^-(g^{-1}k))h(\bar s(gk^-(g^{-1}k))),\bar t(g)} \d k\d g.
\end{align*}
In order to use Lemma \ref{lem:factorizationGM}, we employ the trivial identities
\begin{align*}
\sigma(k^-(g^{-1}k))&=\sigma(g^{-1}gk^-(g^{-1}k)),\\
\bar t(g)&= t(gK)=t(gk^-(g^{-1}k)K)=t(\Pi(gk^-(g^{-1}k)M)).
\end{align*}
Applying the Lemma, we finally get
  \begin{align*}
\mathcal{P}^{\sigma}_{\tau, -\lambda-\rho }(h\otimes Q_{\lambda}(s))(t)&=\int_{G/M} \eklm{h(s(gM)),t(\Pi(gM)))}\d (gM)\\
&=\int_{G/M} \eklm{s(gM),h^\ast(t(\Pi(gM)))}\d (gM)\\
&\equiv\int_{G/M} \eklm{s(gM),h_{\Pi}^\ast(t)(gM)}\d (gM)\qquad \equiv\; {h_{\Pi}}_\ast(s)(t).
\end{align*}
\end{proof}
\subsection{Bijectivity and injectivity of the Poisson transform}\label{sec:regres}
Theorem \ref{thm:main} gives us a strong motivation to study the injectivity and surjectivity of the Poisson transform, depending on the parameter $\lambda\in \C$. Considering the relevance of this problem we shall introduce a notation that is adapted to it. When dealing with bijectivity questions it is convenient to enlarge the domain of the Poisson transform $\mathcal{P}^{\sigma}_{\tau,\lambda}$, as defined in \eqref{eq:poissonnoGamma}. Indeed, $\mathcal{P}^{\sigma}_{\tau,\lambda}$  extends by continuity to a map 
\bq
\mathscr{P}^{\sigma}_{\tau,\lambda}:\mathrm{Hom}_M(V,W)\otimes_\C\Gamma^{-\omega}(K/M,\V^{\mathcal{B}}_{\tau})\longrightarrow E_{[\tau],\lambda}(\W_\sigma)\label{eq:PoissonHF}
\eq
acting on hyperfunction-sections of $\V^{\mathcal{B}}_{\tau}$ instead of distributional sections. Hereby it is not important to know how the space $\Gamma^{-\omega}(K/M,\V^{\mathcal{B}}_{\tau})$ is technically defined. It suffices for our purposes to know that $\Gamma^{-\omega}(K/M,\V^{\mathcal{B}}_{\tau})$ is a topological vector space containing the space $\D'(K/M,\V^{\mathcal{B}}_{\tau})$ of distributional sections, and therefore also the space $\Gamma^\infty(\V^{\mathcal{B}}_{\tau})$ of smooth sections, as a dense subset. It is then clear that the principal series representation $\pi^{\tau,\lambda}_\mathrm{comp}$ extends by continuity to $\Gamma^{-\omega}(K/M,\V^{\mathcal{B}}_{\tau})$, and that $\mathscr{P}^{\sigma}_{\tau,\lambda}$ has the analogous equivariance property as stated in Lemma \ref{lem:Poissonequivariant}. Consequently, we get an induced extended Poisson transform 
\bq
{^\Gamma}\mathscr{P}^{\sigma}_{\tau,\lambda}: \mathrm{Hom}_M(V,W)\otimes_\C{^\Gamma}\mathrm{H}^{-\omega}_{\tau,\lambda}\longrightarrow {^\Gamma} E_{[\tau],\lambda}(\W_\sigma)\cong E_{[\tau],\lambda}(\Gamma\backslash\W_\sigma),\label{eq:PoissonHFGamma}
\eq
where ${^\Gamma}\mathrm{H}^{-\omega}_{\tau,\lambda}$ denotes the $\Gamma$-invariant hyperfunction-vectors in the principal series representation $\pi^{\tau,\lambda}_\mathrm{comp}$. Note that the extended transforms \eqref{eq:PoissonHF}, \eqref{eq:PoissonHFGamma} still have the same codomains as the original Poisson transforms  $\mathcal{P}^{\sigma}_{\tau,\lambda}$, ${^\Gamma}\mathcal{P}^{\sigma}_{\tau,\lambda}$. The point of the extension to hyperfunction-sections is that on these larger domains the transforms have a higher chance of being surjective. 
We shall need the extended Poisson transforms $\mathscr{P}^{\sigma}_{\tau,\lambda}$, ${^\Gamma}\mathscr{P}^{\sigma}_{\tau,\lambda}$ mainly as tools to deduce properties of  $\mathcal{P}^{\sigma}_{\tau,\lambda}$, ${^\Gamma}\mathcal{P}^{\sigma}_{\tau,\lambda}$. \begin{definition}\label{def:regular}For $\lambda\in \C$ we say that
\begin{itemize}[leftmargin=*]
\item $\lambda$  is \emph{weakly regular} if the Poisson transform $\mathcal{P}^{\sigma}_{\tau,-\lambda\nu_0-\rho}(h\otimes_\C \cdot)$ is injective for at least one $h\in \mathrm{Hom}_M(V,W)$.
\item $\lambda$  is \emph{semi-regular} if the Poisson transform  $\mathcal{P}^{\sigma}_{\tau,-\lambda\nu_0-\rho}$ is injective.
\item $\lambda$  is \emph{regular} if the extended Poisson transform $\mathscr{P}^{\sigma}_{\tau,-\lambda\nu_0-\rho}$ is bijective.
\item $\lambda$  is \emph{$\Gamma$-regular} if the Poisson transform  ${^\Gamma}\mathcal{P}^{\sigma}_{\tau,-\lambda\nu_0-\rho}$ is bijective.
\end{itemize}
\end{definition}
We negate $\lambda\nu_0$ and shift by the element $\rho$ in the above definition because the same occurs in Theorem \ref{thm:main}. \begin{rem}From Lemma \ref{lem:transferbijectivity} below it follows that a regular parameter is $\Gamma$-regular. \end{rem}
\begin{thm}[Rough regularity result]\label{thm:rough} There is a number $N_{\sigma}\in \N$, independent of $\tau$, and a finite set $A_{\sigma,\tau}\subset \R$ such that for every $\lambda \in \C\setminus\frac{1}{N_{\sigma}}\Z$ the parameter $\lambda-\norm{\rho}$ is semi-regular and for every $\lambda \in \C\setminus(A_{\sigma,\tau}\cup\frac{1}{N_{\sigma}}\Z)$ the parameter $\lambda-\norm{\rho}$ is regular.  
\end{thm}
\begin{rem}We say \emph{rough} regularity result because much more precise results than Theorem \ref{thm:rough} can be proved, see \cite{olbrichdiss}. However, this requires a lot of additional technical machinery, not only in the proof but also in the statements. For the sake of clarity, we do not discuss the problem of optimizing Theorem \ref{thm:rough} in this paper. \end{rem}
We prepare the proof of Theorem \ref{thm:rough} in two Lemmas. 
\begin{lem}\label{lem:transferbijectivity}
Suppose that $\mathscr{P}^{\sigma}_{\tau,\lambda}$ is injective. Then all $\Gamma$-invariant hyperfunction-vectors in the principal series representation $\pi^{\tau,\lambda}_\mathrm{comp}$ are in fact distributional vectors. More precisely, one has \bq
{^\Gamma}\mathrm{H}^{-\omega}_{\tau,\lambda}={^\Gamma}\mathrm{H}^{-\infty}_{\tau,\lambda},\label{eq:gammahypdistr}
\eq and consequently
$ 
{^\Gamma}\mathscr{P}^{\sigma}_{\tau,\lambda}={^\Gamma}\mathcal{P}^{\sigma}_{\tau,\lambda}.
$ 
\end{lem}
\begin{proof}Olbrich's result \cite[Satz 4.13]{olbrichdiss} says that the restriction of $\mathscr{P}^{\sigma}_{\tau,\lambda}$ to distributional sections, which is by definition $\mathcal{P}^{\sigma}_{\tau,\lambda}$, maps $\mathrm{Hom}_M(V,W)\otimes_\C \D'(K/M,\V^{\mathcal{B}}_{\tau})$ into the subset of $E_{[\tau],\lambda}(\W_\sigma)$ given by all generalized eigensections that are of a very particular \emph{moderate growth}, see \cite[inequality in first half of Satz 4.13]{olbrichdiss}. Now, as the group $\Gamma$ is cocompact, every $\Gamma$-invariant generalized eigensection in $E_{[\tau],\lambda}(\W_\sigma)$ is trivially of moderate growth in that sense because it corresponds to a smooth section on the compact quotient $\gam G/K$.  Under the assumption that $\mathscr{P}^{\sigma}_{\tau,\lambda}$ is injective, define
\[
{^\Gamma}I:= \mathrm{im}\big(\mathscr{P}^{\sigma}_{\tau,\lambda}\big)\cap {{^\Gamma}E_{[\tau],\lambda}}(\W_\sigma),
\]
so that ${\mathscr{P}^{\sigma}_{\tau,\lambda}}^{-1}$ is defined on the set ${^\Gamma}I$. Then we have by  the $G$-equivariance of  ${\mathscr{P}^{\sigma}_{\tau,\lambda}}$ and the argument above
$$\mathrm{Hom}_M(V,W)\otimes_\C{^\Gamma}\mathrm{H}^{-\omega}_{\tau,\lambda}={\mathscr{P}^{\sigma}_{\tau,\lambda}}^{-1}\big({^\Gamma}I\big)=\mathrm{Hom}_M(V,W)\otimes_\C {^\Gamma}\mathrm{H}^{-\infty}_{\tau,\lambda},$$
which implies \eqref{eq:gammahypdistr} provided that $\mathrm{Hom}_M(V,W)\neq \{0\}$. As we assume $[\sigma|_M:\tau]\geq 1$ globally in this section, one has $\mathrm{Hom}_M(V,W)\neq \{0\}$ by Lemma \ref{lem:homhom}.
\end{proof}
\begin{lem}\label{lem:injcrit}Suppose that the principal series representation $\pi^{\tau,\lambda}_\mathrm{comp}$ is irreducible. Then the extended Poisson transform $\mathscr{P}^{\sigma}_{\tau,\lambda}$ is injective. Moreover, one then has
\bq
\dim \mathrm{im}\big(\mathscr{P}^{\sigma}_{\tau,\lambda}\big)\geq [\sigma|_M:\tau]^2.\label{eq:dimim}
\eq
\end{lem}
\begin{proof}
The inequality \eqref{eq:dimim} is shown in \cite[Thm.\ 3.6]{olbrichdiss}. As we demand $[\sigma|_M:\tau]\neq 0$, \eqref{eq:dimim} shows that $\mathscr{P}^{\sigma}_{\tau,\lambda}$ is not the zero map. 
By Lemma \ref{lem:Poissonequivariant} and the density of $\D'(K/M,\V^{\mathcal{B}}_{\tau})$ in $\Gamma^{-\omega}(K/M,\V^{\mathcal{B}}_{\tau})$, the extended Poisson transform \eqref{eq:PoissonHF} fulfills
\begin{multline}
\mathscr{P}^{\sigma}_{\tau,\lambda}\big(\chi_{\tau,\lambda}(D)(h)\otimes \pi^{\tau,\lambda}_\mathrm{comp}(g)(\omega)\big)=D(g\, \mathbb{P}^{\sigma}_{\tau,\lambda}(h\otimes \omega))\\
 \forall\; h\in \mathrm{Hom}_M(V,W),\;\omega\in \Gamma^{-\omega}(K/M,\V^{\mathcal{B}}_{\tau}),\;D\in D(G,\sigma),\;g\in G,\label{eq:152519201252}
\end{multline}
so $\ker \mathscr{P}^{\sigma}_{\tau,\lambda}$ is a $D(G,\sigma)\times G$-invariant subspace of $\mathrm{Hom}_M(V,W)\otimes_\C\Gamma^{-\omega}(K/M,\V^{\mathcal{B}}_{\tau})$. By Proposition \ref{prop:existencechi}, the $D(G,\sigma)$-representation $\chi_{\tau,\lambda}$ is irreducible, so when extending $\pi^{\tau,\lambda}_\mathrm{comp}$ to an irreducible representation of the group algebra $\C G$, the representation $\chi_{\tau,\lambda}\otimes \pi^{\tau,\lambda}_\mathrm{comp}$ of the algebra $D(G,\sigma)\times \C G$ is irreducible (cp.\ \cite[proof of Lemma 4.29]{olbrichdiss}). From this we conclude that the only $D(G,\sigma)\times G$-invariant subspaces of $\mathrm{Hom}_M(V,W)\otimes_\C\Gamma^{-\omega}(K/M,\V^{\mathcal{B}}_{\tau})$ are $\{0\}$ and the entire space. Thus, that space is either entirely equal to the kernel of $\mathscr{P}^{\sigma}_{\tau,\lambda}$, or the kernel is zero. Because we already know that $\mathscr{P}^{\sigma}_{\tau,\lambda}$ is not the zero map, only the possibility $\ker \mathscr{P}^{\sigma}_{\tau,\lambda}=\{0\}$ remains.
 \end{proof}
\begin{proof}[Proof of Theorem \ref{thm:rough}]
According to \cite[Lemma B2]{kroetzkuitopdamschlichtkrull}, there exists an $N_{\sigma}\in \N$, independent of $\tau$ and $\Gamma$, such that $\pi^{\tau,\lambda}_\mathrm{comp}$ is irreducible when $\lambda(H_0) \not\in \frac{1}{N_{\sigma}}\Z$. By Lemma \ref{lem:injcrit}, this implies that  $\mathscr{P}^{\sigma}_{\tau,\lambda}$ is injective for $\lambda(H_0) \not\in \frac{1}{N_{\sigma}}\Z$. Thus, $\mathcal{P}^{\sigma}_{\tau,\lambda}$ is injective as the restriction of the injective map $\mathscr{P}^{\sigma}_{\tau,\lambda}$, and the first statement of Theorem \ref{thm:rough} is proved. To prove also the second statement, we use the main result of Olbrich's dissertation \cite[Thm.\ 4.34]{olbrichdiss}, which implies that there is a finite set $A_{\sigma,\tau}\subset \R$, independent of  $\Gamma$, such that for $\lambda(H_0)\not\in A_{\sigma,\tau}$ the extended Poisson transform $\mathscr{P}^{\sigma}_{\tau,\lambda}$ is not injective if it is not surjective. This implies that $\mathscr{P}^{\sigma}_{\tau,\lambda}$ is bijective for $\lambda(H_0)\not\in A_{\sigma,\tau}\cup\frac{1}{N_{\sigma}}\Z$.
\end{proof}

\subsection{Study of the relevant generalized common eigen\-spaces}\label{sec:studysmb}
The natural pushforward maps from Theorem \ref{thm:main} relate the first band Pollicott-Ruelle resonant states on the vector bundle $\gam \V_\tau$ with generalized common eigenspaces $E_{[\tau],\lambda}(\Gamma\backslash\W_\sigma)$ associated to the algebra $D(G,\sigma)$ of $G$-invariant differential operators acting on sections of the vector bundle $\W_\sigma$. This motivates a more detailed study of these generalized eigenspaces and  the algebra $D(G,\sigma)$, which we shall carry out in this section using the techniques developed in \cite{olbrichdiss}. The main idea is to construct for every operator $D\in D(G,\sigma)$ a polynomial function $$\mathrm{smb}_I(D):\aL^\ast\to\mathrm{End}_M(W)$$
that allows us to describe the action of the finite-dimensional representation $\chi_{\tau,\lambda}$ explicitly. The subscript $I$ stands for Iwasawa, and the notation $\mathrm{smb}_I$ is supposed to indicate that  $D\mapsto \mathrm{smb}_I(D)$ can be regarded as a kind of symbol map that we construct using the opposite Iwasawa decomposition. 

\subsubsection{Invariant differential operators and the universal enveloping algebra}\label{sec:diffopunivers} The $\Ad(K)$-action on $\g$ extends to a $K$-action on the universal enveloping algebra $\mathcal{U}(\g)$. By extension of the infinitesimal action of the Lie algebra $\g$, the $K$-invariant elements in  $\mathcal{U}(\g)$ act on $\Gamma^\infty(\W_\sigma)$ by differential operators. Let us denote this action by
\[
\mathscr{D}:\mathcal{U}(\g)^K \owns u\longmapsto \mathscr{D}(u):\Gamma^\infty(\W_\sigma)\to \Gamma^\infty(\W_\sigma).
\]
Then one has $\mathscr{D}(\mathcal{U}(\g)^K)\subset D(G,\sigma)$. By \cite[Section 2.2]{olbrichdiss} we have in fact
\[
\mathscr{D}(\mathcal{U}(\g)^K)=D(G,\sigma),
\]
which means that every $G$-invariant differential operator is obtained by the action of at least one $K$-invariant element in the universal enveloping algebra. 
\begin{example}[Casimir elements and the Bochner Laplacian]\label{ex:casimir}The (positive) \emph{Casimir operator}  $\mathcal{C}_{\g}$ associated to the Lie algebra $\g$ is
\bq \label{eq:def_casimir}
\mathcal{C}_{\g}=\mathscr{D}(-\Omega_\g),\qquad \Omega_\g=\sum_{i=1}^{\dim \g}X_i\overline X_i \in Z(\mathcal{U}(\g))\subset\mathcal{U}(\g),
\eq
where $\Omega_\g$ is called \emph{Casimir element}. Here $X_i$ is an arbitrary basis of $\g$ and $\overline X_i$ the associated dual basis of $\g$  w.r.t\ the Killing form $\mathfrak{B}$.
In addition we introduce the Casimir elements of the compact groups $K$ and $M$ with respect to the $\Ad(K)$-invariant inner product $\eklm{\cdot,\cdot}|_{\k\times \k}=-\mathfrak{B}|_{\k\times \k}$ and the $\Ad(M)$-invariant inner product $\eklm{\cdot,\cdot}|_{\m\times \m}=-\mathfrak{B}|_{\m\times \m}$, respectively:
\bq
\Omega_{\k}:=\sum_{i=1}^{\dim \mathfrak{k}}K_i^2\in \mathcal{U}(\k)^K,\qquad 
\Omega_{\m}:=\sum_{i=1}^{\dim \mathfrak{m}}M_i^2\in \mathcal{U}(\m)^M,\label{eq:OmegaM}
\eq
where $K_i$ and $M_i$ are  arbitrary orthonormal bases of $\k$ and $\m$ with respect to  $\eklm{\cdot,\cdot}|_{\k\times \k}$ and $\eklm{\cdot,\cdot}|_{\m\times \m}$, respectively. As in \cite[Prop.\ 5.24]{knapp} for $\Omega_\g$, one sees that $\Omega_{\k}$ and $\Omega_{\m}$ 
are independent of the choices of bases. The operators \begin{align}
\mathcal{C}_{\k}^\sigma&:=\sigma(\Omega_{\k})\in \mathrm{End}_{K}(W)=\mathrm{span}_{\C}(\mathrm{id}_W)\label{eq:casimirK1}\\
\mathcal{C}_{\m}^\sigma&:=\sigma(\Omega_{\m})\in \mathrm{End}_{M}(W)\label{eq:casimirM1}
\end{align}
are therefore well-defined. Note that $\mathcal{C}_{\k}^\sigma$ is just a multiple of the identity by Schur's Lemma.   One can show 
that the Bochner Laplacian $\Delta$ from Example \ref{ex:bochner} is obtained as
\bq
\Delta=-\mathscr{D}(\Omega_\g+\Omega_\k)=\mathcal{C}_{\g}+\mathscr{D}(-\Omega_\k)
=\mathcal{C}_{\g}-c_\sigma\cdot  \mathrm{id}_{\,\Gamma^\infty(\W_\sigma)},\label{eq:bochnerD}
\eq
which means that the Bochner Laplacian is the action of the Casimir element in $\mathcal{U}(\g)$ shifted by a scalar $c_\sigma\in \R$ known as the \emph{Casimir invariant} of the representation $\sigma$, defined by $\mathcal{C}_{\k}^\sigma=c_\sigma\cdot \mathrm{id}_W$. We shall determine $c_\sigma$ explicitly in Corollary \ref{cor:eigDelta1}.
\end{example}
\subsubsection{Consequences of the Iwasawa decomposition on the description of generalized common eigenspaces} \label{sec:decompuniv}
The opposite Iwasawa decomposition $\g=\k\oplus \aL\oplus \nL^-$ induces by the Theorem of Poincar\'e-Birkhoff-Witt \cite[III., Thm.\ 3.8]{knapp} a decomposition as linear spaces\footnote{The symbol ``$\cdot$'' refers here to the multiplication in $\mathcal{U}(\g)$.}
\bq
\mathcal{U}(\g)=\big(\mathcal{U}(\k)\cdot \mathcal{U}(\aL)\big) \oplus \big(\mathcal{U}(\g)\cdot \nL^-\big).\label{eq:birkhoffwitt}
\eq
By refining the Cartan decomposition as in \eqref{eq:refinedcartan}, one sees that for $X\in \k,H\in\aL$ the Lie bracket $[X,H]$ is an element of the orthogonal complement $\aL^{\perp_\p}\subset \p$ of $\aL$ in $\p$, which implies that the natural surjective map $\mathcal{U}(\k)\otimes \mathcal{U}(\aL)\to \mathcal{U}(\k)\cdot \mathcal{U}(\aL)$, $X\otimes H\mapsto X\cdot H$, is injective and thus an isomorphism of vector spaces. Composing it with the flip $\mathcal{U}(\k)\otimes \mathcal{U}(\aL)\cong\mathcal{U}(\aL)\otimes \mathcal{U}(\k)$, $X\otimes H\mapsto H\otimes X$, we obtain a vector space isomorphism
\[
I:\mathcal{U}(\k)\cdot \mathcal{U}(\aL)\cong \mathcal{U}(\aL)\otimes \mathcal{U}(\k).
\]
Let $p_-:\mathcal{U}(\g)\to \mathcal{U}(\k)\cdot \mathcal{U}(\aL)$ be the projection onto the first summand in \eqref{eq:birkhoffwitt}. By \cite[3.5.6.(2)]{wallach} the composition $I\circ p_-$ restricts to an antihomomorphism of algebras
\[
I\circ p_-|_{\mathcal{U}(\g)^K}:\mathcal{U}(\g)^K\to \mathcal{U}(\aL)\otimes \mathcal{U}(\k)^M.
\]
We can use $\rho$ from \eqref{eq:rho} to define automorphisms $v_{\rho}:\mathcal{U}(\aL)\to \mathcal{U}(\aL)$ by requiring
\bq
v_{\rho}(H)=H+\rho(H),\qquad H\in \aL.\label{eq:vrho}
\eq
Similarly as in Section \ref{sec:diffopunivers}, consider the antihomomorphism
 $
\mathrm{opp}:\mathcal{U}(\k)^M\to \mathcal{U}(\k)^M
$ 
induced by the involution $X\mapsto -X$ on $\k$. Then we get an algebra homomorphism
\[
\widetilde{\mathrm{smb}}_I:=(v_{\rho}\otimes(\sigma\circ\mathrm{opp}))\circ I\circ p_-|_{\mathcal{U}(\g)^K}: \mathcal{U}(\g)^K\to S(\aL)\otimes_\C\mathrm{End}_M(W),
\]
where we used that $\mathcal{U}(\aL)=S(\aL)$ because $\aL$ is abelian\footnote{Here $S(\aL)$ denotes the symmetric tensor algebra.}. Now, since $S(\aL)\otimes_\C\mathrm{End}_M(W)$ is canonically isomorphic to the algebra of polynomial functions $\aL^\ast\to \mathrm{End}_M(W)$, we can consider $\widetilde{\mathrm{smb}}_I$ as an algebra homomorphism
\[
\widetilde{\mathrm{smb}}_I: \mathcal{U}(\g)^K\to \{\text{polynomial functions }\aL^\ast\to \mathrm{End}_M(W)\}
\]
with the property that $\widetilde{\mathrm{smb}}_I(u)(\aL^\ast)\subset \sigma(\mathcal{U}(\k)^M)$ for all $u\in \mathcal{U}(\g)^K$. 
The significance of this homomorphism lies in the following observations of Olbrich. 
\begin{lem}[{\cite[Satz 2.8, Satz 2.13]{olbrichdiss}}]\label{lem:smbaction}
The assignment
\[
D(G,\sigma)\owns D\longmapsto \mathrm{smb}_I(D):=\widetilde{\mathrm{smb}}_I(u),\quad u\in \mathcal{U}(\g),\;\mathscr{D}(u)=D,
\]
is a well-defined injective algebra homomorphism
\[
\mathrm{smb}_I:D(G,\sigma)\hookrightarrow \{\mathrm{polynomial\;functions\;}\aL^\ast\to \mathrm{End}_M(W)\}.
\]
Moreover, for each $\lambda\in \aL^\ast_\C$ and $D\in D(G,\sigma)$, the action of the finite-dimensional representation $\chi_{\tau,\lambda}:D(G,\sigma)\to \mathrm{End}(\mathrm{Hom}_M(V,W))$ from \eqref{eq:chirep} on $D$ can be described as
\bq
\chi_{\tau,\lambda}(D)=\text{post-composition with }\,\mathrm{smb}_I(D)(\lambda).\label{eq:repsmb2}
\eq
Stated in terms of formulas, \eqref{eq:repsmb2} means
\[
\chi_{\tau,\lambda}(D)(h)=\mathrm{smb}_I(D)(\lambda)\circ h\qquad \forall\;h\in \mathrm{Hom}_M(V,W).
\]
Here the polynomial function $\mathrm{smb}_I(D)$ is understood to be extended from $\aL^\ast$ to $\aL^\ast_\C$. 
\qed
\end{lem}
This description of the representations $\chi_{\tau,\lambda}$ has the following useful implication.
\begin{cor}\label{cor:ordinaryeigenspaces}
For some $D\in D(G,\sigma)$ and $\lambda\in\aL^\ast_\C$, suppose that one has $$\mathrm{smb}_I(D)(\lambda)\in \sigma(\mathcal{U}(\m)^M)\subset \mathrm{End}_M(W).$$ 
Then the generalized common eigenspace $E_{[\tau],\lambda}(\W_\sigma)$ is a subspace of an ordinary eigenspace of $D$, the eigenvalue $\mu_{\sigma,\tau,\lambda}(D)$ being given by 
$
\tau(u)=\mu_{\sigma,\tau,\lambda}(D)\cdot \mathrm{id}_V$,  
where $u\in\mathcal{U}(\m)^M$ is an arbitrary element with $\mathrm{smb}_I(D)(\lambda)=\sigma(u)$.
\end{cor}
\begin{proof}Suppose that $\mathrm{smb}_I(D)(\lambda)\in \sigma(\mathcal{U}(\m)^M)$. 
By Remark \ref{rem:elemobs} and Lemma \ref{lem:smbaction}, it suffices to show that there is a number $\mu_{\sigma,\tau,\lambda}(D)\in \C$ with
\bq
\mathrm{smb}_I(D)(\lambda)\circ h=\mu_{\sigma,\tau,\lambda}(D)\cdot h\qquad \forall\; h\in \mathrm{Hom}_M(V,W).\label{eq:rel24789}
\eq 
For any element $u\in \mathcal{U}(\m)\subset\mathcal{U}(\k)$ one has $\sigma(u)\circ h=\sigma|_M(u)\circ h$. On any irreducible component $\sigma|_M(u)$
acts by a scalar, and as $h$ is $M$-equivariant its range has to be contained in those irreducible components that are equivalent to $(\tau,V)$. Consequently, we get \eqref{eq:rel24789} with some $\mu_{\sigma,\tau,\lambda}(D)\in \C$.
\end{proof}
\subsection{Application to the Bochner Laplacian}\label{sec:smbID}
In the following we work out in detail the constructions from the previous section for the Bochner Laplacian $\Delta\in D(G,\sigma)$.
\begin{prop}\label{prop:laplacesmb}For the Bochner Laplacian $\Delta\in D(G,\sigma)$, the polynomial function $\mathrm{smb}_I(\Delta)$ defined in Lemma \ref{lem:smbaction} is given by
 \bqn
\mathrm{smb}_I(\Delta)(\lambda)=\big(\norm{\rho}^2 -\lambda(H_0)^2-c_\sigma \big)\,\mathrm{id}_W + \mathcal{C}_{\m}^\sigma\;\in\mathrm{End}_M(W),\quad \lambda \in\aL^\ast.
\eqn
Here $\rho$ and $\mathcal{C}_{\m}^\sigma$ have been defined in  \eqref{eq:rho} and \eqref{eq:casimirM1}, respectively, and $c_\sigma\in \R$ is the \emph{Casimir invariant} of the representation $\sigma$, as defined in the paragraph after \eqref{eq:bochnerD}.
\end{prop}
\begin{proof}
Recall from \eqref{eq:bochnerD} in Example \ref{ex:casimir}  that the Bochner Laplacian is obtained as
\bq
\Delta=-\mathscr{D}(\Omega_\g+\Omega_\k)=\mathcal{C}_{\g}+\mathscr{D}(-\Omega_\k).
\eq 
We now have by definition of $\mathrm{smb}_I$
\[
\mathrm{smb}_I(\Delta)\equiv\widetilde{\mathrm{smb}}_I(-\Omega_\g-\Omega_\k)=v_{\rho}\otimes(\sigma\circ\mathrm{opp}))\circ p_-(-\Omega_\g-\Omega_\k).
\]
To compute this polynomial more explicitly, we refine the Cartan decomposition according to
\bq
\g=\k\oplus\p=\m\oplus \m^{\perp_\k}\oplus\aL\oplus \aL^{\perp_\p},\label{eq:refinedcartan}
\eq
 where $\m^{\perp_\k}\subset \k$ is the orthogonal complement of $\m$ in $\k$ and $\aL^{\perp_\p}\subset \p$ is the orthogonal complement of $\aL$ in $\p$. We then choose a basis for $\g$ of the form 
\begin{multline}
\{X_1,\ldots,X_{\dim \g}\}:=\Big\{H_0,M_1,\ldots,M_{\dim \m},
\frac{1}{\sqrt{2}}\big(N_1+\theta N_1\big) ,\ldots,\frac{1}{\sqrt{2}}\big(N_{\dim \nL^-}+\theta N_{\dim \nL^-}\big),\\
\frac{1}{\sqrt{2}}\big(N_1-\theta N_1\big) ,\ldots,\frac{1}{\sqrt{2}}\big(N_{\dim \nL^-}-\theta N_{\dim \nL^-}\big)\Big\},\label{eq:basis1}
\end{multline}
where $\{M_1,\ldots,M_{\dim \m}\}$ is an orthonormal basis of $\m$, $H_0$ the unit vector in $\aL$ introduced in \eqref{eq:H0}, and $\{N_1,\ldots,N_{\dim \nL^-}\}$ an orthonormal basis of $\nL^-$ consisting of restricted-root vectors, meaning that for each $i\in \{1,\ldots,\dim \nL^-\}$ there is a restricted root $\alpha_i\in -\Sigma_{+}$ such that the vector $N_i$ lies in the restricted-root space $\g_{\alpha_i}$. Note that the elements to the right of $M_{\dim \m}$ in the first line of \eqref{eq:basis1} form an orthonormal basis of $\m^{\perp_\k}$ and the elements in the last line of \eqref{eq:basis1} form an orthonormal basis of $\aL^{\perp_\p}$. 
By \eqref{eq:def_casimir} the Casimir element in $\mathcal{U}(\g)$ can be expressed by
\begin{align*}
\Omega_\g&=H_0^2 - \sum_{i=1}^{\dim \m}M_i^2 - \frac{1}{2}\sum_{i=1}^{\dim \nL^-}(N_i+\theta N_i)^2 +  \frac{1}{2}\sum_{i=1}^{\dim \nL^-}(N_i-\theta N_i)^2\\
&=H_0^2 - \sum_{i=1}^{\dim \m}M_i^2 - \sum_{i=1}^{\dim \nL^-}\big(\theta N_i N_i+N_i\theta N_i\big).
\end{align*}
Note that we have
\bq
p_-\bigg(H_0^2 - \sum_{i=1}^{\dim \m}M_i^2\bigg)=H_0^2 - \sum_{i=1}^{\dim \m}M_i^2,\qquad 
p_-(\Omega_\k)=\Omega_\k,\label{eq:proj00}
\eq
since $\Omega_k$ and $H_0^2 - \sum_{i=1}^{\dim \m}M_i^2$ are elements in $\mathcal{U}(\aL)\cdot\mathcal{U}(\k)$. On the other hand, because $\theta N_i N_i\in \mathcal{U}(\g)\otimes\nL^-$, we get 
\bq
p_-\big(\theta N_i N_i\big)=0\qquad \forall\; i\in \{1,\ldots,\dim \nL^-\}\label{eq:proj01}.
\eq
To deal with the summands $N_i\theta N_i$ we write $N_i\theta N_i=[N_i,\theta N_i]+\theta N_iN_i$ 
and apply a basic Lie-theoretic result \cite[Prop.\ 6.52a]{knapp} which says in our context that
\[
[N_i,\theta N_i]=-H_{\alpha_{i}}\in \aL,
\]
where $H_{\alpha_{i}}\in \aL$ is the element corresponding to $\alpha_{i}$ under the isomorphism $\aL\cong \aL^\ast$ provided by the inner product.  We thus have
\[
N_i\theta N_i=\underbrace{-H_{\alpha_{i}}}_{\in\,\aL\subset \mathcal{U}(\aL)\oplus \mathcal{U}(\k)}+\quad\underbrace{\theta N_iN_i}_{\in\, \mathcal{U}(\g)\otimes\nL^-},
\]
from which we read off that $
p_-\big(N_i\theta N_i\big)=-H_{\alpha_{i}}$ for all $i\in \{1,\ldots,\dim \nL^-\}$.  
Combining this with \eqref{eq:proj00} and  \eqref{eq:proj01},  we conclude
\bqn
p_-(\Omega_\g)=H_0^2 + \sum_{i=1}^{\dim \nL^-}H_{\alpha_{i}}- \sum_{i=1}^{\dim \m}M_i^2.
\eqn
Now, for any element $\lambda\in\aL^\ast$ one has
\[
\lambda\bigg(\sum_{i=1}^{\dim \nL^-}H_{\alpha_{i}}\bigg)=\eklm{\sum_{i=1}^{\dim \nL^-}\alpha_{i},\lambda}=\eklm{\sum_{\alpha\in -\Sigma_+}(\dim \g_\alpha) \alpha,\lambda}= -2\eklm{\rho,\lambda}=- 2\lambda(H_\rho),
\]
where $H_\rho\in \aL$ is the vector dual to $\rho\in \aL^\ast$ in the sense that $\rho=\eklm{H_\rho,\cdot}$. This shows
\bqn
p_-(\Omega_\g)=H_0^2 -2H_\rho- \sum_{i=1}^{\dim \mathfrak{m}}M_i^2.
\eqn
Due to the relations 
\begin{align*}
v_{\rho}(H_0^2)&=v_{\rho}(H_0)^2=(H_0+\rho(H_0))^2,\qquad 
v_{\rho}(H_\rho)=H_\rho+\rho(H_\rho)=H_\rho+\norm{\rho}^2,\\
\mathrm{opp}\big(M_i^2\big)&=\mathrm{opp}(M_i)^2=(-M_i)^2=M_i^2\qquad \forall\; i\in \{1,\ldots,\dim \m\}
\end{align*}
and the similarly checked identity $\mathrm{opp}(\Omega_\k)=\Omega_\k$, we arrive at
\begin{align*}
\mathrm{smb}_I(\Delta)&=-(v_{\rho}\otimes(\sigma\circ\mathrm{opp}))\circ p_-(\Omega_\g+\Omega_\k)\\
&=-\big((H_0+\rho(H_0))^2-2H_\rho-2\norm{\rho}^2 \big)\,\mathrm{id}_W+\underbrace{\sum_{i=1}^{\dim \mathfrak{m}}\sigma(M_i)\circ\sigma(M_i)}_{\equiv \mathcal{C}_{\m}^\sigma}\;-\;\sigma(\Omega_\k)\\
&=-\big((H_0+\rho(H_0))^2-2H_\rho-2\norm{\rho}^2 \big)\,\mathrm{id}_W + \mathcal{C}_{\m}^\sigma - \mathcal{C}_{\k}^\sigma.
\end{align*}
To finish the proof, it suffices to observe 
$
 \rho(H_0)=\norm{\rho}$, $H_\rho=\norm{\rho}H_0
$ 
which holds because $\{H_0\}$ is an orthonormal basis of $\aL$ and $\rho$ is a positive linear combination of positive restricted roots.\end{proof}
Proposition \ref{prop:laplacesmb} and Corollary \ref{cor:ordinaryeigenspaces} imply that each of the generalized common eigenspaces $E_{[\tau],\lambda}(\W_\sigma)$ introduced in Definition \ref{def:eigenspaces}  is a subspace of an ordinary eigenspace of the Bochner Laplacian $\Delta\in D(G,\sigma)$. In order to give a formula for the associated eigenvalues, we first need some representation theoretic preparations. 
\subsubsection{Casimir invariants and highest weights}\label{sec:invariantsweights}
In the following we consider root systems of complexified compact Lie algebras; see \cite[p.\ 253-254]{knapp} for their definition. Let $\t^\k\subset \k$, $\t^\m\subset \m$ be maximal abelian subalgebras of the compact Lie algebras $\k$ and $\m$, respectively, and choose systems $\Delta^\k_+$, $\Delta^\m_+$ of positive roots in the respective root systems $\Delta(\k_\C,\t^\k_\C)$ and $\Delta(\m_\C,\t^\m_\C)$. Analogously as $\rho$ was defined in \eqref{eq:rho}, define
\bq
\delta_\k:=\frac{1}{2}\sum_{\alpha \in \Delta^\k_+}\alpha\;\in i(\t^\k)^\ast\subset {\t^\k_\C}^\ast,\qquad 
\delta_\m:=\frac{1}{2}\sum_{\alpha \in \Delta^\m_+}\alpha\;\in i(\t^\m)^\ast\subset {\t^\m_\C}^\ast.\label{eq:rhom}
\eq
By definition, the compact group $K$ is connected. In contrast, the compact group $M$ can be disconnected in general\footnote{For example, if $G=\SL(2,\R)$ with $K=\SO(2)$, one has $M\cong\Z_2$.}. Let $M_0$ be the identity component of $M$. Then $\tau|_{M_0}$ is a representation of $M_0$ on $V$ that might be reducible in general. Therefore, let
\[
V=V^0_1\oplus \cdots\oplus V^0_{N_0},\qquad \tau|_{M_0}=\tau^0_1\oplus\cdots\oplus \tau^0_{N_0}
\]
be the decomposition of $\tau|_{M_0}$ into irreducible $M_0$-representations, i.e., $\tau^0_j$ is a representation $M_0\to \mathrm{End}(V^0_j)$. The theorem of the highest weight for compact connected Lie groups \cite[Thm.\ 5.110]{knapp} says that to the equivalence classes of irreducible representations $[\sigma]\in \hat K$ and $[\tau^0_j]\in \hat M_0$ there correspond highest weights\footnote{In common literature one finds also another notation convention in which the imaginary unit is absorbed into the definition of a highest weight, which is then an element of a real instead of a purely imaginary dual Lie algebra.}    $$\omega_{[\sigma]}\in i(\t^\k)^\ast\subset {\t^\k_\C}^\ast,\qquad \omega_{[\tau^0_j]}\in i(\t^\m)^\ast\subset {\t^\m_\C}^\ast$$
with respect to the orderings in $i(\t^\k)^\ast$ and $i(\t^\m)^\ast$ that we fixed by choosing the positive systems $\Delta^\k_+$ and $\Delta^\m_+$. 
Let now $c_\sigma$, $c_\tau\in \R$ be the \emph{Casimir invariants} of the $K$-repre\-sentation $\sigma$ and the $M$-representation $\tau$ with respect to the bilinear forms $\eklm{\cdot,\cdot}|_{\k\times \k}$ and $\eklm{\cdot,\cdot}|_{\m\times \m}$, respectively, defined by the relations
\bq
\sigma(\Omega_{\k})=c_\sigma\cdot\mathrm{id}_W,\qquad \tau(\Omega_{\m})=c_\tau\cdot\mathrm{id}_V,\label{eq:defcs}
\eq
where the Casimir elements $\Omega_\k\in \mathcal{U}(\k)^K$, $\Omega_\m\in \mathcal{U}(\m)^M$ have been defined in \eqref{eq:OmegaM}. In addition, denote for $j\in \{1,\ldots,N_0\}$ by $c_{\tau^0_j}\in \R$ the Casimir invariant of the $M_0$-representation $\tau_j^0$ with respect to the bilinear form $\eklm{\cdot,\cdot}|_{\m\times \m}$, defined by
\[
\tau^0_j(\Omega_{\m})=c_{\tau_j^0}\cdot\mathrm{id}_{V^0_j}.
\]
\begin{lem}\label{lem:casimirinvariants}
The Casimir invariants with respect to $\eklm{\cdot,\cdot}|_{\m\times \m}$ of the irreducible $M$-representation $\tau$ and all the irreducible $M_0$-re\-pre\-sen\-tations $\tau^0_j$ agree. They are given by
\[
c_\tau=c_{\tau_j^0}=\norm{i\delta_\m}^2-\big\Vert i\omega_{[\tau^0_j]}+i\delta_\m\big\Vert^2,
\]
which is the same number for all $j\in \{1,\ldots,N_0\}$. Similarly, the Casimir invariant with respect to $\eklm{\cdot,\cdot}|_{\k\times \k}$ of the irreducible $K$-representation $\sigma$ is given by
\[
c_\sigma=\norm{i\delta_\k}^2-\big\Vert i\omega_{[\sigma]}+i\delta_\k\big\Vert^2.
\]
\end{lem}
\begin{proof}
As the $M$-invariant Casimir element $\Omega_\m\in \mathcal{U}(\m)^M$ and the $M$-invariant bilinear form $\eklm{\cdot,\cdot}|_{\m\times \m}$ are both also $M_0$-invariant, a basic result in representation theory \cite[Prop.\ 4 on p.\ 358]{bourbaki7-9} says that
\[
c_{\tau^0_j}=\norm{i\delta_\m}^2-\Vert i\omega_{[\tau^0_j]}+i\delta_\m\Vert^2,\qquad j\in \{1,\ldots,N_0\}.
\]
Now, since the infinitesimal actions of $M$ and $M_0$ on $\m$ agree, we have
\bq
\tau(\Omega_\m)=\tau_1^0(\Omega_\m)\oplus\cdots\oplus\tau_{N_0}^0(\Omega_\m):V\to V,\label{eq:82951029501}
\eq
where $\tau(\Omega_\m)$ as well as the $\tau^0_{j}(\Omega_\m)$ act as scalar multiples of the identity on their corresponding vector spaces. The relation \eqref{eq:82951029501} can be valid only if all the scalars are the same. This proves the first half of the Lemma. The formula for $c_{\sigma}$ is another direct application of \cite[Prop.\ 4 on p.\ 358]{bourbaki7-9} since $K$ is connected and $\eklm{\cdot,\cdot}|_{\k\times \k}$, $\Omega_{\k}$ are both $K$-invariant.
\end{proof}
With the preparations from the previous subsection, we can prove
\begin{cor}[{Compare \cite[Lemma 2.14]{olbrichdiss}}]\label{cor:eigDelta1}For every $\lambda \in \aL^\ast_\C$ one has\
\[
E_{[\tau],\lambda}(\W_\sigma)\subset \mathrm{Eig}(\Delta,\mu_{[\sigma],[\tau],\lambda}),\qquad E_{[\tau],\lambda}(\gam\W_\sigma)\subset \mathrm{Eig}(\Delta_\Gamma,\mu_{[\sigma],[\tau],\lambda}),
\]
where the right hand sides denote the eigenspaces of the Bochner Laplacian $\Delta\in D(G,\sigma)$ and the induced Bochner Laplacian $\Delta_\Gamma$, acting on smooth sections of $\gam\W_\sigma$, with the eigenvalue
\bq
\mu_{[\sigma],[\tau],\lambda}=\norm{\rho}^2-\lambda(H_0)^2+\norm{i\omega_{[\sigma]}+i\delta_\k}^2-\big\Vert i\omega_{[\tau_1^0]}+i\delta_\m\big\Vert^2+\norm{i\delta_\m}^2-\norm{i\delta_\k}^2.\label{eq:evlaplace}
\eq
Here $\rho$ and $\delta_\k,\delta_\m\in i\k^\ast$ have been defined in  \eqref{eq:rho} and \eqref{eq:rhom}, respectively, and $\omega_{[\sigma]},\omega_{[\tau_1^0]}\in i\k^\ast$ are the highest weights associated to $[\sigma]\in \hat K$ and $[\tau_1^0]\in \hat M_0$ as introduced in Section \ref{sec:invariantsweights}. 
\end{cor}
\begin{proof}
Corollary \ref{cor:ordinaryeigenspaces} and Proposition \ref{prop:laplacesmb} yield 
$
E_{[\tau],\lambda}(\W_\sigma)\subset \mathrm{Eig}(\Delta,\mu_{[\sigma],[\tau],\lambda})
$ 
with \[
\mu_{[\sigma],[\tau],\lambda}=\norm{\rho}^2-\lambda(H_0)^2+c_{\tau}-c_{\sigma},
\]
where the Casimir invariants $c_{\tau}$ and $c_{\sigma}$ were introduced in \eqref{eq:defcs}. The claimed formula for the eigenvalue now follows from Lemma \ref{lem:casimirinvariants}. To get the statement on the $\Gamma$-quotient, it suffices to recall \eqref{eq:EpmGamma}.
\end{proof}
\begin{cor}\label{cor:eigfinite}The generalized common eigenspaces
$E_{[\tau],\lambda}(\gam\W_\sigma)$ are finite-dimen\-sional and consist of smooth sections of $\gam\W_\sigma$.
\end{cor}
\begin{proof}
By Corollary \ref{cor:eigDelta1}, the  $E_{[\tau],\lambda}(\gam\W_\sigma)$ are subspaces of eigenspaces of $\Delta_\Gamma$. As an elliptic operator on the compact manifold $\gam G/K$, $\Delta_\Gamma$ has finite-dimensional eigenspaces. By elliptic regularity, the eigenvectors of $\Delta_\Gamma$ are smooth sections.
\end{proof}
As an immediate consequence of Theorem \ref{thm:main} and Corollary \ref{cor:eigDelta1} we obtain
\begin{thm}\label{thm:main2}
Let $\lambda\in \C$ be a Pollicott-Ruelle resonance on $\gam\V_\tau$ in the first band. Then, for every first band resonant state $s\in \mathrm{Res}^{0}_{\,\gam\V_{\tau}}(\lambda)$, the section
\[
{{^\Gamma}h_{\Pi}}_\ast(s)\in E_{[\tau], -\lambda\nu_0-\rho }(\Gamma\backslash\W_\sigma)
\]
is an eigensection of the Bochner Laplacian $\Delta_\Gamma$ with the eigenvalue
\[
\norm{\rho}^2-(\lambda+\norm{\rho})^2+\norm{i\omega_{[\sigma]}+i\delta_\k}^2-\big\Vert i\omega_{[\tau_1^0]}+i\delta_\m\big\Vert^2+\norm{i\delta_\m}^2-\norm{i\delta_\k}^2.
\]
In particular, ${{^\Gamma}h_{\Pi}}_\ast(s)$ is a smooth section.
\qed
\end{thm}
\begin{cor}\label{cor:re_lambda_-rho}
 Let $\mathcal V_\tau$ be the vector bundle 
 associated to a (not necessarily irreducible) finite-dimensional unitary $M$-representation $\tau$. If $\lambda\in\C$ is a first band Pollicott-Ruelle resonance on $\gam\mathcal V_\tau$, then either $\Im(\lambda)=0$ or $\Re(\lambda)=-\norm{\rho}$.
\end{cor}
\begin{proof}
 By decomposition into irreducibles, we can reduce the problem to 
 an irreducible representation. The statement then follows from 
 Theorem~\ref{thm:main2} and the fact that all eigenvalues of the Bochner Laplacian are real-valued. 
\end{proof}
\section{Band structure}\label{sec:band}
In the following we show how the description of the first band of vector-valued 
Pollicott-Ruelle resonances from Theorem~\ref{thm:main2} allows to deduce a result of a general
band structure. In this section, let $\tau$ be a fixed arbitrary finite-dimensional 
unitary representation of $M$. As
we do not require irreducibility of $\tau$, $\mathcal V_\tau$ is an arbitrary associated vector bundle over $G/M$. We shall prove
\begin{thm}\label{thm:band}
If $\lambda\in \C$  is a Pollicott-Ruelle resonance on  $\gam \mathcal V_\tau$, then either $\Im(\lambda) = 0$ or 
$\Re(\lambda) \in - \|\rho\| - \N_0\|\alpha_0\|$. Furthermore, if 
$s\in \tu{Res}_{\gam\mathcal V_\tau}(-\|\rho\| + ir)$, $r\in \R\setminus\{0\}$, then 
for all $\X \in \Gamma^\infty(\gam E_-)$ we have $\X s=0$, i.e.,
\begin{equation}\label{eq:first_band_is_first_band}
\tu{Res}_{\gam\mathcal V_\tau}(-\|\rho\| + ir) = \tu{Res}^0_{\gam\mathcal V_\tau}(-\|\rho\| + ir).
\end{equation}
\end{thm}
A central tool for the proof of Theorem \ref{thm:band} are horocycle operators. Their
definition is a suitable generalization of the operators introduced in \cite{dfg}
for real hyperbolic spaces. However, as we now deal with arbitrary rank one spaces, 
we have to introduce two different horocycle operators associated to the two different restricted 
roots $\alpha_0$ and $2\alpha_0$. For their definition, recall that one has  
$T(\Gamma\backslash G/M) = \gam E_0\oplus \gam E_+\oplus \gam E_-$ and 
\[
E_- = G \times_M \mathfrak n_- = \underbrace{G\times_M \g_{-\alpha_0}}_{=:E_{-\alpha_0}}
\oplus \underbrace{G\times_M\g_{-2\alpha_0}}_{=:E_{-2\alpha_0}}.
\] 
After complexification we consider the bundle homomorphisms
\[
 \tu{pr}_{\gam (E_{\alpha}^\ast)_\C}:T^*(\Gamma\backslash G/M)_\C \to \gam (E^*_{\alpha})_\C, \qquad \alpha\in\{-\alpha_0,-2\alpha_0\}
\]
given by  fiber-wise restriction to the subbundle $\gam (E_{\alpha})_\C$. By composition they induce maps 
\[
\widetilde{\tu{pr}}_{\gam (E_{\alpha}^\ast)_\C}:\Gamma^\infty(\gam\mathcal V_{\tau}\otimes T^*(\Gamma\backslash G/M)_\C) \to \Gamma^\infty(\gam\mathcal V_{\tau}\otimes \gam (E^*_\alpha)_\C)\cong \Gamma^\infty(\gam(\mathcal V_{\tau}\otimes (E^*_\alpha)_\C)).
\]
Let $(\tau',V')$ be another arbitrary finite-dimensional unitary $M$-representation with associated vector bundle  
$\mathcal V_{\tau'}$.\footnote{The bundle 
$\mathcal V_{\tau'}$ will change frequently in the following arguments. For example, we shall put $\mathcal V_{\tau'}=\mathcal V_{\tau}\otimes (E^\ast_{-\alpha_0})^{\otimes m}_\C\otimes (E^\ast_{-2\alpha_0})^{\otimes l}_\C$ with various $m,l\in \N_0$. } 
 We define
the \emph{horocycle operators} for $\alpha\in\{-\alpha_0,-2\alpha_0\}$ by
\[
 \mathcal U_\alpha:= \widetilde{\tu{pr}}_{\gam (E_{\alpha}^\ast)_\C} \circ \,{_\Gamma}\nabla_\C : 
 \Gamma^\infty(\gam \mathcal V_{\tau'}) \to \Gamma^\infty(\gam(\mathcal V_{\tau'}\otimes (E^*_\alpha)_\C)).
\]
Here ${_\Gamma}\nabla_\C$ is the complex linear extension of the canonical connection on the bundle $\gam \mathcal V_{\tau'}$. 
Using the horocycle operators we can rewrite the definition of the 
space of first band Pollicott-Ruelle resonant states (see Definition \ref{def:ruelle1stband}) on $\gam\V_\tau$ as
\begin{equation}
 \label{eq:first_band_horocycle_description}
 \Res^0_{\Gamma\backslash \mathcal V_\tau}(\lambda) = 
 \{s\in \Res_{\Gamma\backslash \mathcal V_\tau}(\lambda)~| ~\mathcal U_{-\alpha_0}s=0 \tu{ and }
 \mathcal U_{-2\alpha_0}s=0\}.
\end{equation}
The horocycle operators fulfill the following important commutation relations.
\begin{lem}\label{lem:horcyclic_operators_properties}
 Recall that $\mathfrak X_{H_0}\in\Gamma^\infty( T(\Gamma\backslash G/M))$ denotes  
 the geodesic vector field. We have for $\alpha\in\{-\alpha_0, -2\alpha_0\}$
 the identities\footnote{Note that there is a slight abuse of notation in formulating
 these formulas as commutators: Stricly speaking, the two instances of 
 $\nabla_{\mathfrak X_{H_0}}$ in the expression $\nabla_{\mathfrak X_{H_0}}\mathcal U_\alpha -  
 \mathcal U_\alpha \nabla_{\mathfrak X_{H_0}}$ are not the same operators: 
 In the first expression it acts on sections of $\mathcal V_{\tau'}\otimes (E_\alpha^*)_\C$ and in the 
 second on $\mathcal V_{\tau'}$.}
 \begin{eqnarray}
  [\nabla_{\mathfrak X_{H_0}}, \mathcal U_\alpha] &=& \alpha(H_0)\mathcal U_\alpha  \label{eq:horo_commutator_1}\\
  {[}\mathcal U_\alpha, \mathcal U_{-2\alpha_0}] &=& 0  \label{eq:horo_commutator_2}
 \end{eqnarray}
\end{lem}
\begin{proof}
 The statement follows from the definiton of $\nabla$, the 
 Leibniz rule, and the Lie algebra identities $[H_0, X_\alpha] = \alpha(H_0)X_\alpha$ and $[X_\alpha, X_{-2\alpha_0}]=0$
 for $\alpha=-\alpha_0,-2\alpha_0$ and $X_\alpha\in \g_\alpha$ by straightforward calculations.
\end{proof}
As a consequence of the commutation relations we obtain the following result on the resonant states.
\begin{lem}\label{lem:horo_first_band_shift}
 Let $\lambda\in \C$ be a Pollicott-Ruelle resonance on $\gam \mathcal V_\tau$ and take $s\in \Res_{\gam\mathcal V_\tau}(\lambda)\setminus \{0\}$.  
 Then there are unique $k,l\in\N_0$ such that 
 \[
 \mathcal U_{-\alpha_0}^k\mathcal U_{-2\alpha_0}^l s \in \Res^0_{\gam(\mathcal V_\tau\otimes 
 (E^*_{-\alpha_0})_\C^{\otimes k} \otimes (E^*_{-2\alpha_0})_\C^{\otimes l})}(\lambda +k\alpha_0(H_0) + 
 2l\alpha_0(H_0))\setminus\{0\}.
 \]
\end{lem}
\begin{proof}
Let $s\in  \Res_{\gam \mathcal V_\tau}(\lambda)\setminus \{0\}$. As a direct consequence of \eqref{eq:horo_commutator_1} and the fact that taking covariant derivatives cannot enlarge the wave front set, we 
get that $\mathcal U_{-2\alpha_0}^m s \in 
\Res_{\gam(\mathcal V_\tau\otimes (E_{-2\alpha_0}^*)_\C^{\otimes m})}(\lambda + 2m\alpha_0(H_0))$. We know, however, that on an arbitrary 
Hermitian vector bundle there are no Pollicott-Ruelle resonances with 
positive real part (see Proposition~\ref{prop:merom_resolvent}). 
Thus, there is a maximal integer $l\in \N_0$ such that $\mathcal U_{-2\alpha_0}^l s\neq 0$.
By the same argument but starting with 
$\mathcal U_{-2\alpha_0}^l s \in \Res_{\gam(\mathcal V_\tau\otimes(E_{-2\alpha_0}^*)_\C^{\otimes l})}(\lambda + 2l\alpha_0(H_0))\setminus\{0\}$ and applying $\mathcal U_{-\alpha_0}$, 
we deduce that there has to be a maximal $k\in \N_0$ such 
that $\mathcal U_{-\alpha_0}^k \mathcal U_{-2\alpha_0}^l s \neq 0$. 
We already know, by the commutation relations, that  
$$\mathcal U_{-\alpha_0}^k\mathcal U_{-2\alpha_0}^l s \in \Res_{\gam(\mathcal V_\tau\otimes 
 (E^*_{-\alpha_0})_\C^{\otimes k} \otimes (E^*_{-2\alpha_0})_\C^{\otimes l})}(\lambda +k\alpha_0(H_0) + 
 2l\alpha_0(H_0)),$$
 so it only remains to show that the section is a first band resonant state. We use the description \eqref{eq:first_band_horocycle_description} of the first band. The fact that $\mathcal U_{-\alpha_0}(\mathcal U_{-\alpha_0}^k\mathcal U_{-2\alpha_0}^l s) = 0$ follows from the choice of $k$, and the fact that 
 $\mathcal U_{-2\alpha_0}(\mathcal U_{-\alpha_0}^k\mathcal U_{-2\alpha_0}^l s) = 0$ follows from \eqref{eq:horo_commutator_2} and the 
 choice of $l$. 
\end{proof}
\begin{proof}[Proof of Theorem \ref{thm:band}]
 If $\lambda$ is a Pollicott-Ruelle resonance on $\gam\mathcal V_\tau$, then we know by 
 Lemma~\ref{lem:horo_first_band_shift}  that there are
 $k,l\in\N_0$ such that $\lambda+(k+2l)\alpha_0(H_0)$ is a 
 first band Pollicott-Ruelle resonance on the bundle 
 $\gam\big(\mathcal V_\tau\otimes(E_{-\alpha_0}^*)_\C^{\otimes k}\otimes(E_{-2\alpha_0}^*)_\C^{\otimes l}\big)$. Applying Corollary~\ref{cor:re_lambda_-rho} to this bundle finishes the proof.
\end{proof}

\section{Convenient special cases}\label{sec:specialcase}
 In this section we derive a simple and more concrete description of the generalized eigenspaces $E_{[\tau],\lambda}(\W_\sigma)$ and $E_{[\tau],\lambda}(\Gamma\backslash\W_\sigma)$ under additional assumptions (Assumptions \ref{ass:1} and \ref{ass:2} below) that define a class of convenient special cases. Among them are many interesting examples such as associated vector bundles over real and complex hyperbolic spaces. As in Section  \ref{sec:fbpoisson}, we assume that the $M$-representation $\tau$ is irreducible and  $\sigma:K\to \mathrm{End}(W)$ is an irreducible unitary representation with $[\sigma|_M:\tau]\geq 1$.  We begin with
 \begin{assumption2}\label{ass:1}One has $[\sigma|_M:\tau]=1$.
\end{assumption2}
\begin{rem}\label{rem:choosesigma}By the Frobenius reciprocity theorem \cite[Thm.\ 1.14]{knapp2} one has $$[\sigma|_M:\tau]=[\mathrm{ind}^K_M\tau:\sigma],$$
where $\mathrm{ind}^K_M\tau$ is the natural induction of the $M$-representation $\tau$ to a $K$-representation. Now, given some irreducible $M$-representation $\tau$, we see that one can find an irreducible $K$-representation $\sigma$ such that Assumption \ref{ass:1} is fulfilled if, and only if,  $\mathrm{ind}^K_M\tau$ possesses an irreducible subrepresentation that occurs in it with multiplicity one.
\end{rem}
\begin{rem}[{\cite[Satz 2.9]{olbrichdiss}}]If Assumption \ref{ass:1} holds for \emph{every} irreducible representation $\tau$ of $M$ with $[\tau]$ occurring in $\sigma|_M$, then the algebra $D(G,\sigma)$ is commutative.  Conversely, if the algebra $D(G,\sigma)$ is commutative, then Assumption \ref{ass:1} is fulfilled for every irreducible unitary $M$-representation $\tau$ such that $[\tau]$ occurs in $\sigma|_M$.
\end{rem}
\begin{rem}[{\cite[p.\ 16]{olbrichdiss}}]\label{rem:ass1}
The algebra $D(G,\sigma)$ is commutative for every irreducible unitary $K$-representation $\sigma$ iff $G/K$ is a real or complex hyperbolic space, that is, $G=\SO(n+1,1)_0$ or  $G=\SU(n+1,1)$ for some $n\in \N$. In particular, Assumption \ref{ass:1} is fulfilled for arbitrary $\sigma,\tau$ when $G=\SO(n+1,1)_0$ or  $G=\SU(n+1,1)$, $n\in \N$.
\end{rem}
\begin{assumption2}\label{ass:2}The equivalence class $[\tau]\in \hat M$ is $\mathscr{W}$-invariant.
\end{assumption2}
Here we refer to the Weyl group action \eqref{eq:actionW111} on $\hat M$.
\begin{rem}[{\cite[proof of Lemma 4.32]{olbrichdiss}}]\label{rem:ass2}Assumption \ref{ass:2} is fulfilled for every $[\tau]\in \hat M$ unless $G$ is locally isomorphic to $\mathrm{Spin}(1,2n+1)$ for some $n\in \N$. 
\end{rem}
Let us now see how Assumptions \ref{ass:1} and \ref{ass:2} allow us to give a more explicit description of the generalized common eigenspaces.
\begin{cor}\label{cor:es1}If Assumption \ref{ass:1} holds, the generalized common eigenspaces $E_{[\tau],\lambda}(\W_\sigma)$ are ordinary common eigenspaces of all the operators in $D(G,\sigma)$:
\begin{align*}
E_{[\tau],\lambda}(\W_\sigma)&=\{s\in \Gamma^\infty(\W_\sigma):Ds=\mu_{[\sigma],[\tau],\lambda}(D)s\},\; \mu_{[\sigma],[\tau],\lambda}(D)\in \C\;\forall\; D\in D(G,\sigma),\\
E_{[\tau],\lambda}(\gam\W_\sigma)&=\{s\in \Gamma^\infty(\gam\W_\sigma):s\mathrm{\;is\;induced\;by\;}\tilde s\in E_{[\tau],\lambda}(\W_\sigma)^\Gamma\}.
\end{align*}
\end{cor}
\begin{proof}
Recalling \eqref{eq:chirep} and Definition \ref{def:eigencomnot}, apply Lemma \ref{lem:homhom} and Remark \ref{rem:elemobs}.
\end{proof}
The eigenvalues $\mu_{[\sigma],[\tau],\lambda}(\Delta)$ of the Bochner Laplacian are given explicitly in Corollary \ref{cor:eigDelta1}. For the other operators in $D(G,\sigma)$ it turns out that under the additional Assumption \ref{ass:2} the computation of their eigenvalues is very simple: For every operator that is algebraically independent of $\Delta$ in $D(G,\sigma)$, the eigenvalue $\mu_{[\sigma],[\tau],\lambda}(D)$ is $0$.
\begin{thm}\label{thm:es2}
Under Assumptions \ref{ass:1} and \ref{ass:2} one can choose generators $\Delta,D_1,\ldots,D_{N}$ of the algebra $D(G,\sigma)$, where $N:=\dim \mathrm{End}_M(W)-1$, such that
\[
E_{[\tau],\lambda}(\W_\sigma)=\mathrm{Eig}(\Delta,\mu_{[\sigma],[\tau],\lambda}(\Delta))\cap\bigcap_{i=1}^{N}\ker D_i.
\]
In particular, the generalized common eigenspace $E_{[\tau],\lambda}(\W_\sigma)$ depends on $\lambda$ only via the eigenvalue $\mu_{[\sigma],[\tau],\lambda}(\Delta)$ of the Bochner Laplacian $\Delta$ given in Corollary \ref{cor:eigDelta1}.
\end{thm}
The proof of Theorem \ref{thm:es2} will be based on the following
\begin{lem}\label{lem:evensmb}
If Assumption \ref{ass:1} holds for $[\tau]$, it also holds for $w_0[\tau]$. Moreover, under Assumptions \ref{ass:1} and \ref{ass:2}, one has
\[
\mathrm{smb}_I(D)(\lambda)\circ h=\mathrm{smb}_{I}(D)(-\lambda)\circ h\qquad \forall\;h\in\mathrm{Hom}_M(V,W),\;D\in D(G,\sigma),\;\lambda \in \aL^\ast.
\]
\end{lem}
\begin{proof}First, we introduce the matrix element of the \emph{Knapp-Stein intertwining operator}
\bq
j_{\sigma}(\lambda):=\int_{N^+}e^{(\lambda+\rho)(H^-(n^+))}\sigma(k^-(n^+)^{-1})\d n^+\in \mathrm{End}_M(W),\label{eq:j}
\eq
where $\lambda\in \aL_\C^\ast$ is required to fulfill $\mathrm{Re}\,\lambda(H_0)>0$. The function $j_{\sigma}$ possesses a continuation to a  meromorphic function $\aL_\C^\ast\to \mathrm{End}_M(W)$, and as such we shall regard it from now on. By \cite[Lemma 5.5]{bunke-olbrich2000} and \cite[Lemma 2.18]{olbrichdiss}, there is a meromorphic function $\eta_{[\tau]}:\aL_\C^\ast\to\C$ such that we have for all $h\in \mathrm{Hom}_M(V,W)$
\[
j_{\sigma}(\lambda)\circ w_0(j_{\sigma}(w_0\lambda))\circ h=\eta_{[\tau]}(\lambda)\, h,
\]
compare also \cite[Thm.\ 14.12]{knapp2}, \cite[Thm.\ 7]{knapp-stein1970}. Here we use the action \eqref{eq:endWaction} of the Weyl group on $\mathrm{End}_M(W)$, and $w_0\lambda=-\lambda$. In particular, this implies that for $\lambda$ outside a  countable set of isolated points in $\aL_\C^\ast$, the endomorphism $\mathrm{Hom}_M(V,W)\to \mathrm{Hom}_M(V,W)$ given by post-composition with $j_{\sigma}(\lambda)$ is invertible. By \cite[Satz 2.19]{olbrichdiss} one has
\bq
\mathrm{smb}_I(D)(w_0\lambda)\circ j_\sigma(w^{-1}_0 \lambda)=j_\sigma(w^{-1}_0 \lambda)\circ (w_0\,\mathrm{smb}_I(D)(\lambda)),\label{eq:25122155125}
\eq
which is to be understood as an identity of meromorphic functions on $\aL_\C^\ast$ with values in $\mathrm{End}_M(W)$. Under Assumption \ref{ass:1} every $M$-equivariant endomorphism $W\to W$ restricts by Schur's Lemma to a multiple of the identity on the vector space $V_{[\tau]}\subset W$ on which $\sigma|_M$ acts as an $M$-representation equivalent to $\tau$, and the image of any $h\in \mathrm{Hom}_M(V,W)$ is contained in $V_{[\tau]}$. Therefore, the endomorphisms $\mathrm{smb}_I(D)(w_0\lambda)$ and $j_\sigma(w^{-1}_0 \lambda)$ commute on the image of any $h\in \mathrm{Hom}_M(V,W)$:
\bq
\mathrm{smb}_I(D)(w_0\lambda)\circ j_\sigma(w^{-1}_0 \lambda)\circ h= j_\sigma(w^{-1}_0 \lambda)\circ\mathrm{smb}_I(D)(w_0\lambda)\circ h.\label{eq:25122155125325325}
\eq
Combining \eqref{eq:25122155125} with \eqref{eq:25122155125325325} yields for every $h\in \mathrm{Hom}_M(V,W)$
\[
j_\sigma(w^{-1}_0 \lambda)\circ \mathrm{smb}_I(D)(w_0\lambda)\circ h=j_\sigma(w^{-1}_0 \lambda)\circ (w_0\,\mathrm{smb}_I(D)(\lambda))\circ h.
\]
With the above observation that, for $\lambda$ outside a countable set of isolated points in $\aL_\C^\ast$, the endomorphism $\mathrm{Hom}_M(V,W)\to \mathrm{Hom}_M(V,W)$ given by post-composition with $j_{\sigma}(\lambda)$ is invertible, and recalling $w_0\lambda=-\lambda$, we conclude that 
\bq
\mathrm{smb}_I(D)(-\lambda)\circ h=(w_0\,\mathrm{smb}_I(D)(\lambda))\circ h\qquad \forall\; h\in \mathrm{Hom}_M(V,W)\label{eq:checkcond123532535}
\eq
 holds for all $\lambda\in \aL_\C^\ast$ outside a countable set of isolated points. But \eqref{eq:checkcond123532535} is an identity between polynomials in $\lambda$, consequently it holds for all $\lambda\in \aL_\C^\ast$. Consider now the inclusions
\[
M\subset M'\subset K,
\]
where $M,M'$ are as in \eqref{eq:Weyl}, in particular $M$ is an index $2$ subgroup of $M'$ because $G$ has real rank $1$. The corresponding restrictions $\sigma|_{M'},\sigma|_M$ of the irreducible $K$-representation $\sigma$ split into irreducibles, respectively, according to
\[
\sigma|_{M'}=\bigoplus_{j=1}^{N'}\tau_j',\qquad \sigma|_{M}=\bigoplus_{j=1}^{N'}\tau_j'|_{M}=\bigoplus_{j=1}^{N'}\bigoplus_{i=1}^{N_j}\tau_{j,i},
\]
where $\tau_{i,j}\cong \tau$ for at least one tuple $(i,j)$ because $[\sigma|_M:\tau]\geq 1$. More precisely, Clifford's theorem \cite{clifford37} says
\bq
N_j\in \{1,2\} \quad\forall\; j\in\{1,\ldots,N'\}\label{eq:Clifford1}
\eq
and
\bq
N_j=2\implies \exists\; m'_j\in M':\tau_{j,2}=\tau_{j,1}^{m'_j},\label{eq:Clifford2}
\eq
where
\bq
\tau_{j,1}^{m'_j}(m):=\tau_{j,1}\big(m'_j m(m'_j)^{-1}\big),\qquad m\in M,\label{eq:Clifford3}
\eq
is the conjugate of the representation $\tau_{j,1}$ by the element $m'_j$. Thus, each of the irreducible $M'$-representations $\tau_j'$ splits when restricted to $M$ into at most $2$ irreducible $M$-representations, and if it splits into two, then they are conjugate to each other via some element in $M'$. Now, under Assumption \ref{ass:1}  there is exactly one index $j_0\in \{1,\ldots,N'\}$ such that $[\tau'_{j_0}|_M:\tau]\neq0$. We now either have $N_{j_0}=2$ or $N_{j_0}=1$. If $N_{j_0}=2$, then $w_0[\tau]$ occurs in $\sigma|_{M}$, and by applying the arguments above to $w_0[\tau]$, we see that it can occur only with multiplicity $1$, so that Assumption \ref{ass:1} is fulfilled for $w_0[\tau]$. On the other hand, if $N_{j_0}=1$, $\tau_{j_0}'|_M$ is irreducible, consequently one has $w_0[\tau]=[\tau]$, in particular Assumptions \ref{ass:1} and  \ref{ass:2} are fulfilled for $w_0[\tau]$. This proves the first claim. To prove the remaining statements, we show that if $N_{j_0}=1$, which is equivalent to Assumptions \ref{ass:1} and  \ref{ass:2},  one has
\bq
(w_0E)\circ h=E\circ h\qquad \forall\; E\in \mathrm{End}_M(W),\;h\in \mathrm{Hom}_M(V,W).\label{eq:w0trivial}
\eq
Indeed, let $V_0\subset W$ be the vector space on which $\sigma|_{M'}$ acts as the irreducible $M'$-representation $\tau'_{j_0}$. Then $\sigma|_{M}=\tau'_{j_0}|_M$ acts on $V_0$ as the irreducible $M$-representation $\tau'_{j_0}|_M\cong \tau$. By Assumption \ref{ass:1} this is the only irreducible $M$-representation equivalent to $\tau$ occurring in $\sigma|_M$. Therefore, Schur's Lemma implies that every $M$-equivariant endomorphism $W\to W$ restricts to a multiple of the identity on $V_0$ and that one has
\bq
\mathrm{image}(h)\subset V_0\qquad \forall\;h\in \mathrm{Hom}_M(V,W).\label{eq:imhom}
\eq
Recall that the action of $w_0$ on $\mathrm{End}_M(W)$ was defined in \eqref{eq:endWaction} and \eqref{eq:endaction} by
\[
w_0E=\sigma(m'_{w_0})\circ E\circ \sigma(m'_{w_0})^{-1},\qquad w_0=m'_{w_0}M\in \mathscr{W},\;E\in \mathrm{End}_M(W).
\]
Now we have $\sigma(m'_{w_0})|_{V_0}=\tau'_{j_0}(m'_{w_0})|_{V_0}$ for every $m'_{w_0}\in M'$. As $\tau'_{j_0}(m'_{w_0})$ maps $V_0$ into $V_0$ and any $E\in \mathrm{End}_M(W)$ restricts to a multiple of the identity on $V_0$, we get
\[
\sigma(m'_{w_0})\circ E\circ \sigma(m'_{w_0})^{-1}|_{V_0}=\tau'_{j_0}(m'_{w_0})\circ E\circ \tau'_{j_0}(m'_{w_0})^{-1}|_{V_0}=E|_{V_0},
\]
because $\tau'_{j_0}(m'_{w_0})$ commutes with a multiple of the identity. Together with \eqref{eq:imhom}, this shows \eqref{eq:w0trivial}. Combining this with \eqref{eq:checkcond123532535}, we are done.
\end{proof}
\begin{proof}[Proof of Theorem \ref{thm:es2}]
Lemma \ref{lem:smbaction} says that the representation $\chi_{\tau,\lambda}$ of $D\in D(G,\sigma)$ is given by post-composing $M$-equivariant endomorphisms $V\to W$ with $\mathrm{smb}_I(D)(\lambda)$. Let \begin{multline*}
\mathrm{symm}_{\mathscr{W}}:\{\mathrm{polynomial\;functions\;}\aL^\ast\to \mathrm{End}_M(W)\}\\\longrightarrow \{\mathscr{W}\mathrm{\text{-}invariant\;polynomial\;functions\;}\aL^\ast\to \mathrm{End}_M(W)\}
\end{multline*}
be the projection given by
\[
\mathrm{symm}_{\mathscr{W}}(p)(\lambda)=\frac{1}{|\mathscr{W}|}\sum_{w\in \mathscr{W}}p(w\lambda)=\frac{1}{2}(p(\lambda)+p(w_0\lambda)),\qquad \lambda \in \aL^\ast.
\]
Lemma \ref{lem:evensmb} tells us that under Assumptions \ref{ass:1} and \ref{ass:2} the action of the representation $\chi_{\tau,\lambda}$ on $D(G,\sigma)$ is determined by $\mathrm{symm}_{\mathscr{W}}(\mathrm{smb}_I(D))(\lambda)$. For the Bochner Laplacian $\Delta$ we get from Proposition \ref{prop:laplacesmb} 
\[
\mathrm{symm}_{\mathscr{W}}(\mathrm{smb}_I(\Delta))=\mathrm{smb}_I(\Delta).
\]
Moreover, Proposition \ref{prop:laplacesmb}, \eqref{eq:rel24789}, and Corollaries \ref{cor:ordinaryeigenspaces}, \ref{cor:eigDelta1} show that the composition of $\mathrm{smb}_I(\Delta)$ with an $M$-equivariant endomorphism $V\to W$ agrees with the composition with the quadratic polynomial function
\[
\lambda \mapsto\big(\norm{\rho}^2-\lambda(H_0)^2+\norm{i\omega_{[\sigma]}+i\delta_\k}^2-\big\Vert i\omega_{[\tau_1^0]}+i\delta_\m\big\Vert^2+\norm{i\delta_\m}^2-\norm{i\delta_\k}^2 \big)\,\mathrm{id}_W
\]
on $\aL^\ast$. By \cite[Folg.\ 2.5]{olbrichdiss}, the algebra $D(G,\sigma)$ is finitely generated with at most $N=\dim \mathrm{End}_M(W)-1$ generators. Thus, let $\Delta,\tilde D_1,\ldots,\tilde D_{N}\in D(G,\sigma)$ be generators of the algebra $D(G,\sigma)$. 
Let $V_0\subset W$ be the vector space on which $\sigma|_M$ acts as an $M$-representation equivalent to $\tau$. Then, by Schur's Lemma, we can choose a basis of $\mathrm{End}_M(W)$ of the form
\[
E_0:=\mathrm{id}_W,E_1,\ldots,E_N
\]
such that
\bq
E_i|_{V_0}=0\qquad \forall\;i\in \{1,\ldots,N\}.\label{eq:nullE}
\eq
Every $\mathscr{W}$-invariant polynomial function $p:\aL^\ast\to \mathrm{End}_M(W)$ is of the form 
\[
p(\lambda)=\sum_{k=0}^{(\text{deg } p)/2}\sum_{j=0}^Nc_{j,k}(p)\norm{\lambda}^{2k}E_{j},\qquad c_{j,k}(p)\in 
\C,\qquad \lambda\in \aL^\ast.
\]
Taking \eqref{eq:nullE} and \eqref{eq:imhom} into account, we get for any $h\in \mathrm{Hom}_M(V,W)$,  $i\in \{1,\ldots,N\}$, $\lambda\in \aL^\ast$ the relations
\begin{align*}
\mathrm{smb}_I(\tilde D_i)(\lambda)\circ h&=\mathrm{symm}_{\mathscr{W}}(\mathrm{smb}_I(\tilde D_i)(\lambda))\circ h\\
&=\sum_{k=0}^{(\text{deg } \mathrm{smb}_I(\tilde D_i))/2}\sum_{j=0}^Nc_{j,k}(\tilde D_i)\norm{\lambda}^{2k}E_{j}\circ h\\
&=\sum_{k=0}^{(\text{deg } \mathrm{smb}_I(\tilde D_i))/2}c_{0,k}(\tilde D_i)\norm{\lambda}^{2k}E_{0}\circ h.
\end{align*}
For $i\in \{1,\ldots,N\}$, define a polynomial $p_i$ in one complex variable $\lambda$ by requiring
\begin{multline*}
p_i\big(\norm{\rho}^2-\lambda(H_0)^2+\norm{i\omega_{[\sigma]}+i\delta_\k}^2-\big\Vert i\omega_{[\tau_1^0]}+i\delta_\m\big\Vert^2+\norm{i\delta_\m}^2-\norm{i\delta_\k}^2 \big)\\
=\sum_{k=0}^{(\text{deg } \mathrm{smb}_I(\tilde D_i))/2}c_{0,k}(\tilde D_i)\lambda(H_0)^{2k},
\end{multline*}
and set $D_i:=\tilde D_i-p_i(\Delta).$ 
As $p_i(\Delta)$ lies in the subalgebra of $D(G,\sigma)$ generated by $\Delta$ and the operators $\tilde D_i$ are disjoint from this algebra, the family $\Delta,D_1,\ldots,D_N$ generates the algebra $D(G,\sigma)$. However, by construction one has
\[
\mathrm{smb}_I(D_i)(\lambda)\circ h=0\qquad \forall\; h\in \mathrm{Hom}_M(V,W), \;i\in \{1,\ldots,N\},\;\lambda\in \aL^\ast,
\]
which by Lemma \ref{lem:smbaction} means $\mu_{[\sigma],[\tau],\lambda}(D_i)\equiv\chi_{\tau,\lambda}(D_i)=0$ for all $i$ and $\lambda$. 
\end{proof}

\section{Jordan blocks for resonances in the first band}
Knowing that the generalized common eigenspaces
$E_{[\tau],\lambda}(\gam\W_\sigma)$ are subspaces of ordinary eigenspaces of the Bochner Laplacian $\Delta_\Gamma$ on $\gam G/K$, we can proceed similarly as in \cite[Theorem 3]{GHW18b} to decide the existence of non-trivial first band Jordan blocks, as introduced in Definition \ref{def:ruelle1stband}, for all first band Pollicott-Ruelle resonances that fulfill the weakest of the regularity assumptions introduced in Definition \ref{def:regular}. 
\begin{thm}\label{thm:jordan} Let $\lambda\in \C$ be a weakly regular first band Pollicott-Ruelle resonance on $\gam\V_\tau$.\begin{enumerate}[leftmargin=*]
\item If $\lambda\neq -\norm{\rho}$, there are no non-trivial first band Jordan blocks for the resonance $\lambda$.
\item If $\lambda=-\norm{\rho}$, all first band Jordan blocks for the resonance $\lambda$ are at most of size $2$.\\ If one can\footnote{Compare Remark \ref{rem:choosesigma}.} find $\sigma$ such that Assumption \ref{ass:1} is fulfilled, then the following holds.
\begin{itemize}\item[2\,a.] If $\tau$ fulfills Assumption \ref{ass:2}, all first band Jordan blocks  for the resonance $\lambda=-\norm{\rho}$ are exactly of size $2$. 
\item[2\,b.] If $\tau$ does not fulfill Assumption \ref{ass:2}, there are no non-trivial first band Jordan blocks for the resonance $\lambda=-\norm{\rho}$.\end{itemize}
\end{enumerate}
\end{thm}
\begin{proof}
We adapt the proof of \cite[Thm.\ 3]{GHW18b} to our setting and notation. Let us begin with statement 1. Suppose that a non-trivial first band Jordan block exists for a first band Pollicott-Ruelle resonance $\lambda_0\in \C$. Then there are non-zero $s_0,s_1\in \D'(\Gamma\backslash G/M,\Gamma\backslash \V_{\tau})$ with 
\begin{align*}
(\Gam\Xbf+ \lambda_0)s_0&=0,\qquad 
(\Gam\Xbf+ \lambda_0)s_1=s_0,\qquad {_\Gamma}\nabla_{\X}(s_0)={_\Gamma}\nabla_{\X}(s_1)=0\quad\forall\;\X\in \Gamma^\infty(\Gamma\backslash E_-).
\end{align*}
Choose a homomorphism $h\in \mathrm{Hom}_M(V,W)$ such that $\mathcal{P}^{\sigma}_{\tau,-\lambda\nu_0-\rho}(h\otimes_\C \cdot)$ is injective. Theorem \ref{thm:main2} then says that the section $\phi_0:={{^\Gamma}h_{\Pi}}_\ast(s_0)\in E_{[\tau],-\lambda_0\nu_0-\rho}(\Gamma\backslash\W_\sigma)$
fulfills
\bq
(\Delta_\Gamma -\mu(\lambda_0))\phi_0=0,\label{eq:laplace24215}
\eq
\[
\mu(\lambda_0)=\norm{\rho}^2-(\lambda_0+\norm{\rho})^2+\norm{i\omega_{[\sigma]}+i\delta_\k}^2-\big\Vert i\omega_{[\tau_1^0]}+i\delta_\m\big\Vert^2+\norm{i\delta_\m}^2-\norm{i\delta_\k}^2.
\]
As $\mathcal{P}^{\sigma}_{\tau,-\lambda\nu_0-\rho}(h\otimes_\C \cdot)$ is injective, also the pushforward ${{^\Gamma}h_{\Pi}}_\ast$ is injective by Theorem \ref{thm:main}. Therefore, because $s_0\neq 0$, we have $\phi_0\neq 0$. Define $
\phi_1:={{^\Gamma}h_{\Pi}}_\ast(s_1)$. 
Let $\tilde s_0,\tilde  s_1\in \D'(G/M,\V_{\tau})$ and $\tilde \phi_0,\tilde \phi_1\in \D'(G/K,\W_\sigma)$ be the $\Gamma$-invariant distributional sections corresponding to $s_0,s_1,\phi_0$, and $\phi_1$, respectively, as in \eqref{eq:quotientR}. Similarly as in Proposition \ref{prop:jordanchar} we now introduce distributional sections in $\D'(K/M,\V^{\mathcal{B}}_{\tau})$ by setting
\begin{align*}
\omega_{0}&:= \iota^\ast\big(\Phi^{-\lambda_0}\tilde s_0\big)=\iota^\ast\big(\Phi^{-\lambda}\Phi^{\lambda-\lambda_0}\tilde s_0\big)=Q_{\lambda}(\Phi^{\lambda-\lambda_0}\tilde s_0)\quad \forall\;\lambda\in  \C,\\
\omega_{1}&:= \iota^\ast\Big(\Phi^{-\lambda_0}\big((\log \Phi)\tilde s_0+\tilde s_1\big)\Big).
\end{align*}
The relation $B^\ast\circ \iota^\ast=\mathrm{id}_{\mathcal{R}(0)}$ from Proposition \ref{prop:pullbackiso} gives the identity
\bq
\tilde s_1=\Phi^{\lambda_0}\big(B^\ast \omega_{1}-(\log \Phi)B^\ast \omega_{0}\big),\label{eq:1892358023237273}
\eq
and by \eqref{eq:Possnat1} we have
\bq
\mathcal{P}^{\sigma}_{\tau,-\lambda\nu_0-\rho}(h\otimes \omega_{0})={h_\Pi}_\ast(\Phi^{\lambda-\lambda_0}s_0)\qquad \forall\;\lambda\in  \C.\label{eq:18924738883237273}
\eq
By pairing with test sections and using the density of $\Gamma^\infty(\V^{\mathcal{B}}_{\tau})$ in $\D'(K/M,\V^{\mathcal{B}}_{\tau})$, we can take on both sides of \eqref{eq:18924738883237273} the derivative with respect to $\lambda$ at $\lambda=\lambda_0$, which yields
\bqn
\frac{\partial}{\partial \lambda}\Big|_{\lambda_0}\mathcal{P}^{\sigma}_{\tau,-\lambda\nu_0-\rho}(h\otimes \omega_{0})={h_\Pi}_\ast\big((\log \Phi)s_0\big)={h_\Pi}_\ast\big(\Phi^{\lambda_0}(\log \Phi)B^\ast \omega_{0}\big).
\eqn
Thus, when applying the pushforward ${h_\Pi}_\ast$ to \eqref{eq:1892358023237273}, we get
\begin{align}
\nonumber\tilde \phi_1={h_\Pi}_\ast(\tilde s_1)&={h_\Pi}_\ast\Big(\Phi^{\lambda_0}\big(B^\ast \omega_{1})\Big)-\frac{\partial}{\partial \lambda}\Big|_{\lambda_0}\mathcal{P}^{\sigma}_{\tau,-\lambda\nu_0-\rho}(h\otimes \omega_{0})\\
&=\mathcal{P}^{\sigma}_{\tau,-\lambda\nu_0-\rho}(h\otimes\omega_{1})-\frac{\partial}{\partial \lambda}\Big|_{\lambda_0}\mathcal{P}^{\sigma}_{\tau,-\lambda\nu_0-\rho}(h\otimes \omega_{0}).\label{eq:32579837912759}
\end{align}
Here we applied equation \eqref{eq:Possnat1} in the last step, using that $\Phi^{\lambda_0}\big(B^\ast \omega_{1})\in \mathcal{R}(\lambda_0\nu_0)$. By Lemma \ref{lem:Poissonequivariant} we have $$\mathcal{P}^{\sigma}_{\tau,-\lambda\nu_0-\rho}(h\otimes\omega_{1})\in E_{[\tau],-\lambda\nu_0-\rho}(\W_\sigma),$$
so we deduce from \eqref{eq:32579837912759} the formula
\bq
(\Delta -\mu(\lambda_0))\tilde \phi_1=-(\Delta -\mu(\lambda_0))\frac{\partial}{\partial \lambda}\Big|_{\lambda_0}\mathcal{P}^{\sigma}_{\tau,-\lambda\nu_0-\rho}(h\otimes \omega_{0}).
\eq
Again by Lemma \ref{lem:Poissonequivariant}, we have
\bq
(\Delta -\mu(\lambda))\mathcal{P}^{\sigma}_{\tau,-\lambda\nu_0-\rho}(h\otimes \omega_{0})=0\qquad \forall\; \lambda \in   \C,\label{eq:1892473888323727336532}
\eq
and taking, as above, the derivative with respect to $\lambda$ at $\lambda=\lambda_0$ yields
\[
\Big(-\frac{\partial}{\partial \lambda}\Big|_{\lambda_0}\mu(\lambda)\Big)\mathcal{P}^{\sigma}_{\tau,-\lambda\nu_0-\rho}(h\otimes \omega_{0})
+(\Delta -\mu(\lambda_0))\frac{\partial}{\partial \lambda}\Big|_{\lambda_0}\mathcal{P}^{\sigma}_{\tau,-\lambda\nu_0-\rho}(h\otimes \omega_{0})=0.
\] 
Computing the derivative of  $\mu(\lambda)$ using \eqref{eq:laplace24215} and combining \eqref{eq:18924738883237273} with \eqref{eq:1892473888323727336532}, we arrive at
\[
(\Delta -\mu(\lambda_0))\tilde \phi_1=2(\lambda_0+\norm{\rho}){h_\Pi}_\ast(s_0)=2(\lambda_0+\norm{\rho})\tilde \phi_0,
\]
which is equivalent to
\[
(\Delta_\Gamma -\mu(\lambda_0))\phi_1=2(\lambda_0+\norm{\rho})\phi_0.
\]
By \eqref{eq:laplace24215} we also have
\bqn
(\Delta_\Gamma -\mu(\lambda_0))\phi_0=0.
\eqn
For $\lambda_0+\norm{\rho}\neq 0$, the last two equations say that there is a non-trivial Jordan block associated to the eigenvalue $\mu(\lambda_0)$ of the Laplacian $\Delta_\Gamma$. This is a contradiction because the operator $\Delta_\Gamma$ on the compact manifold $\gam G/K$ is symmetric in  
$\L^2(\gam G/K,\gam \W_\sigma)$ and therefore does not possess any non-trivial Jordan blocks in its spectrum. This finishes the proof of statement $1$. To prove the similar statement 2b we recall Remark \ref{rem:ass2} which says that if $\tau$ violates Assumption $2$, the group $G$ is locally isomorphic to $\mathrm{Spin}(1,2n+1)$ for some $n\in \N$, and one can then show that the representation $\tau$ embeds into the tensor product of certain Spin-representations. This has the consequence that there is a Dirac type operator\footnote{The operator $iD_1$ considered in Section \ref{eq:algebraex} is an example of such an operator.} $D\in D(G,\sigma)$,   
symmetric in $\L^2(G/K,\W_\sigma)$, such that one has under Assumption \ref{ass:1}
\[
E_{[\tau],\lambda\nu_0}(\W_\sigma)\subset \mathrm{Eig}(D,\mu_D(\lambda))\qquad \forall\; \lambda \in \C,
\]
where $\lambda\mapsto \mu_D(\lambda)$ is a non-constant polynomial of degree $1$, see \cite[remark before Def.\ 4.33]{olbrichdiss}. 
Statement 2b can now be proved by applying the same arguments as above but with $\Delta$, $\Delta_\Gamma$, $\mu(\lambda_0)$ replaced by the invariant differential operator $D$, the induced operator $D_\Gamma$,  and its eigenvalue $\mu_D(\lambda_0)$. To prove also Statement 2a using the ideas in the proof of \cite[Thm.\ 3]{GHW18b}, 
we need to find an appropriate generalization of the family of scattering operators $S_\mu$ occurring there, where $\mu=\lambda+\norm{\rho}$. In the situation of \cite[Thm.\ 3]{GHW18b}, it is crucial to the definition of $S_\mu$ that the image of the scalar Poisson transform at $\mu$, consisting of an eigenspace $\E_\mu$ of the Laplace Beltrami operator, agrees with the space $\E_{-\mu}$. In our situation, we require Assumptions \ref{ass:1} and \ref{ass:2} and have the analogous relation
\[
E_{[\tau],\mu\nu_0}(\W_\sigma)=E_{[\tau],-\mu\nu_0}(\W_\sigma),
\]
as follows from Lemmas \ref{lem:evensmb} and \ref{lem:smbaction}. Now, by Theorem \ref{thm:rough}, the Poisson transform $\mathcal{P}^{\sigma}_{\tau,\mu\nu_0}$ is bijective for all $\mu$ outside a discrete set in $\C$, therefore we can find a neighborhood $U\subset \C$ of $0$ such that $\mathcal{P}^{\sigma}_{\tau,\mu\nu_0}$ is bijective for all $\mu \in U\setminus\{0\}$. We can then define the scattering operator
\[
\mathcal{S}^{\sigma}_{\tau,\mu}:=\mathcal{P}^{\sigma}_{\tau,-\mu\nu_0}\Big|_{\mathrm{im}(\mathcal{P}^{\sigma}_{\tau,-\mu\nu_0})}^{-1}\circ \mathcal{P}^{\sigma}_{\tau,\mu\nu_0},\qquad \mu\in U\subset \C,
\]
\[
\mathcal{S}^{\sigma}_{\tau,\mu}:\mathrm{Hom}_M(V,W)\otimes_\C\D'(K/M,\V^{\mathcal{B}}_{\tau})\longrightarrow\mathrm{Hom}_M(V,W)\otimes_\C\D'(K/M,\V^{\mathcal{B}}_{\tau}).
\]
The space $\mathrm{Hom}_M(V,W)$ is one-dimensional due to  Assumption \ref{ass:1}, so we can regard the scattering operator as a map
\[
\mathcal{S}^{\sigma}_{\tau,\mu}:\D'(K/M,\V^{\mathcal{B}}_{\tau})\longrightarrow\D'(K/M,\V^{\mathcal{B}}_{\tau}).
\] 
It is bijective with inverse $
\big(\mathcal{S}^{\sigma}_{\tau,\mu}\big)^{-1}=\mathcal{S}^{\sigma}_{\tau,-\mu}$, and one has $\mathcal{S}^{\sigma}_{\tau,0}=\mathrm{id}$.  We can now finish the proof of Statement 2a analogously as in \cite[Thm.\ 3]{GHW18b} using Proposition \ref{prop:jordanchar} and Lemma \ref{lem:Poissonequivariant}. 
\end{proof}

\appendix
\section{Eigenvalues of the Bochner Laplacian on bundles of trace-free symmetric tensors over $\mathbb H^n$}\label{app:Bochner_eigenspaces}
Let us compute the eigenvalues obtained in Theorem \ref{thm:main2} explicitly for the following class of examples: For $n\in \{1,2,\ldots\}$, choose $G=\SO(n+1,1)_0$. Then $K\cong \SO(n+1)$ and $M\cong\SO(n)$. First, assume that $n\geq 3$. Then a well-known family of irreducible $K$- and $M$-representations is given by the representations of $\SO(n+1)$ and $\SO(n)$ on spaces of spherical harmonics or, equivalently, on spaces of symmetric, trace-free tensors as in $\cite{dfg}$. 
For $m,m'\in \{0,1,2,\ldots\}$, let $\sigma_m$ and $\tau_{m'}$ be the irreducible representations of $\SO(n+1)$ and $\SO(n)$ on the spaces of spherical harmonics in dimensions $n+1$ and $n$ of degrees $m$ and $m'$, respectively. As $M$ is connected, we have $(\tau_{m'})_1^0=\tau_{m'}$.  According to  \cite[p.\ 219-220, 272-276]{broecker-tomdieck} and \cite[p.\ 254,425,430]{goodmanwallach}, there is the following description of the highest weights $\omega_{[\sigma_m]}\in \k^\ast, \omega_{[\tau_{m'}]}\in \m^\ast$ as well as the elements $\delta_\k\in\k^\ast,\delta_\m\in\m^\ast$: For odd $n=2l+1$, there is an orthonormal family of \footnote{Note that \cite{broecker-tomdieck} and \cite{goodmanwallach} use an inner product that differs from our inner product by a factor of $\frac{1}{2n}$.} 
 vectors $e_1,\ldots,e_{l+1}\in \k^\ast$ such that
\bqn
\omega_{[\sigma_m]}=\frac{i}{\sqrt{2n}} me_1,\qquad \omega_{[\tau_{m'}]}=\frac{i}{\sqrt{2n}} m'e_1,\eqn\bqn \delta_\k=\frac{i}{\sqrt{2n}}\sum_{j=1}^{l+1}(l-j+1)e_j,\qquad \delta_\m=\frac{i}{2\sqrt{2n}}\sum_{j=1}^{l}(2l-2j+1)e_j,
\eqn
while for $n=2l$ even, we obtain similarly $e_1,\ldots,e_l\in \k^\ast$ such that
\bq
\omega_{[\sigma_m]}=\frac{i}{\sqrt{2n}} me_1,\qquad \omega_{[\tau_{m'}]}=\frac{i}{\sqrt{2n}} m'e_1,\label{eq:5237503890}\eq\bq \delta_\k=\frac{i}{2\sqrt{2n}}\sum_{j=1}^{l}(2l-2j+1)e_j,\qquad \delta_\m=\frac{i}{\sqrt{2n}}\sum_{j=1}^{l}(l-j)e_j.\label{eq:52373q47474790}
\eq
A quick computation then shows
\begin{align}
\norm{i\omega_{[\sigma_m]}+i\delta_\k}^2-\big\Vert i\omega_{[\tau_{m'}]}+i\delta_\m\big\Vert^2+\norm{i\delta_\m}^2-\norm{i\delta_\k}^2
=\frac{m' + (m-m')(m+m'+n-1)}{2n} \label{eq:firsthalf}
\end{align}
for all $n\geq 3$. Next, let us consider the case $n=2$. Then one has $K\cong\SO(3)$, so we can still choose $\sigma_m$ from above as our irreducible $K$-representation. In contrast, $M\cong \SO(2)$ is now abelian and its complex irreducible representations can be modeled by
\[
\tau_{m'}:  \SO(2)\cong S^1=\{z\in \C: |z|=1\}\owns z\mapsto \big(w\mapsto z^{m'}w\big)\in \mathrm{End}(\C),\qquad m'\in \Z.
\]
We have $\omega_{[\tau_{m'}]}=\frac{i}{2} m'e_1$, $m'\in \Z.$  As $M$ is abelian, there are no roots, in particular we have $\delta_\m=0$. Using  \eqref{eq:5237503890} and \eqref{eq:52373q47474790} for $K\cong\SO(3)$, 
we find for $n=2$ the same formula as \eqref{eq:firsthalf} but now with $m'$ allowed to be negative. The remaining case to be considered is $n=1$. Then, one has $K\cong\SO(2)$, so what we said in the case $n=2$ about $M$ now applies to $K$, in particular we consider for  $m\in \Z$ the one-dimensional  irreducible representation $\sigma_m$ of $K$. On the other hand, $M\cong \SO(1)$ is now the trivial group, for which we shall consider the trivial representation $\tau_0$ in dimension $n=1$. Taking into acccount that now $K$ and $M$ are both abelian, we arrive at
\[
\delta_\k=\delta_\m=0,\qquad \omega_{[\tau_0]}=0,\qquad \omega_{[\sigma_{m}]}=\frac{i}{2} me_1,\quad m\in \Z.
\]
This yields for the case $n=1$
\bq
\norm{i\omega_{[\sigma_m]}+i\delta_\k}^2 -\norm{i\omega_{[\tau_{0}]}+i\delta_\m}^2+\norm{i\delta_\m}^2-\norm{i\delta_\k}^2=\norm{\frac{1}{2} me_1}^2=\frac{m^2}{4}.\label{eq:firsthalf3}
\eq
We are finished with the computation of the summand constituting the second half of the eigenvalue from Theorem \ref{thm:main2}. Identifying in Theorem \ref{thm:main2} the element $\lambda\in \aL^\ast_\C$ with the Pollicott-Ruelle resonance $\lambda(H_0)\in \C$, it remains to compute the number 
\[
\norm{\rho}^2-(\lambda+\rho)(H_0)^2=\norm{\rho}^2-(\lambda+\norm{\rho})^2=-\lambda(\lambda + 2\norm{\rho}),
\]
which amounts to the computation of $\norm{\rho}$. To this end, recall from Section \ref{sec:setupnot} how $\rho$ and the norm on $\aL^\ast\subset \g^\ast$ are defined. The Lie algebra $\g=\mathfrak{so}(n+1,1)$ has a type $B_1$ restricted root system, see \cite[p.\ 365]{knapp}. In particular, the unique reduced positive restricted root $\alpha_0$ is the only positive restricted root, so the system of restricted roots is just $\Sigma=\{\pm \alpha_0\}$. 
One can express 
the restriction of the Killing form $\mathfrak{B}$ to $\aL\times \aL$ as
\[
\mathfrak{B}(H,H')
=\sum_{\alpha \in \Sigma}(\mathrm{dim}\,\g_{\alpha})\alpha(H)\alpha(H')=2(\mathrm{dim}\,\g_{\alpha_0})\alpha_0(H)\alpha_0(H'),\qquad H,H'\in \aL.
\]
A straightforward computation then shows
\[
\norm{\rho}^2=\Big\Vert\frac{1}{2}(\mathrm{dim}\,\g_{\alpha_0})\,\alpha_0\Big\Vert^2=\frac{1}{4}\norm{\alpha_0}^2(\mathrm{dim}\,\g_{\alpha_0})^2=\frac{\mathrm{dim}\,\g_{\alpha_0}}{8}.
\]
From the restricted root space decomposition 
$ 
\g=\aL\oplus \m\oplus \g_{\alpha_0}\oplus \g_{-\alpha_0}
$ 
and the fact that $\g_{\alpha_0}$ is isomorphic to $\g_{-\alpha_0}$ via the Cartan involution, one computes with $\dim \g=\dim \mathfrak{so}(n+1,1)=\frac{(n+2)(n+1)}{2}$, $\dim \m=\dim \mathfrak{so}(n)=\frac{n(n-1)}{2}$, $\dim \aL=1$:
\begin{align*}
\mathrm{dim}\,\g_{\alpha_0}&=\frac{1}{2}(\dim \mathfrak{so}(n+1,1)-\dim \mathfrak{so}(n)-1)=n.
\end{align*}
In total, we conclude $
\norm{\rho}=\frac{1}{2}\sqrt{\frac{n}{2}}
$ and we can summarize our computations in
\begin{cor}
Suppose that  the irreducible representations $\sigma_m$, $\tau_{m'}$ introduced above are compatible in the sense that $[\sigma_m:\tau_{m'}]\neq 0$.  Then the eigenvalue of the Bochner Laplacian obtained in Theorem \ref{thm:main2} for these representations is given by
\bq
-\lambda\Big(\lambda + \sqrt{\frac{n}{2}}\Big)+\frac{m' + (m-m')(m+m'+n-1)}{2n},\qquad n\in \{1,2,3,\ldots\}.\label{eq:eiglaplaceex}
\eq\qed
\end{cor}
Here we took into account that for $n=1$ one has $[\sigma_m:\tau_{0}]\neq 0$ iff $m=0$. 
\begin{rem}\label{rem:compare} In \cite{dfg}, Dyatlov, Faure, and Guillarmou consider essentially the situation of this Example with $n\geq 2$ and  $m=m'$. They obtain in \cite[Theorem 6]{dfg} the eigenvalue $-\lambda(\lambda + n)+m$ that differs from our result when putting $m=m'$ in \eqref{eq:eiglaplaceex}. The explanation is that they use an inner product in $\p\cong T_{e K}(G/K)$ which differs from the restriction $\eklm{\cdot,\cdot}_{\p\times \p}$ of our inner product by a factor of $2n$. Their choice of inner product has the advantage that the resulting sectional curvature of the Riemannian symmetric space $G/K=\SO(n+1,1)_0/\SO(n+1)$ is precisely $-1$. On the other hand, our choice is natural from a Lie theoretic point of view because the inner product is constructed from the Killing form and the Cartan involution without extra normalization factors.
\end{rem}

\section{An example on $\mathbb H^3$}\label{app:joint_eigenspace}
In this section we apply the constructions and results of this paper to some particular vector bundles over the hyperbolic space $\mathbb{H}^3=\SO(3,1)_0/\SO(3)$ and its cosphere bundle. Some results regarding this space are already covered by Appendix~\ref{app:Bochner_eigenspaces}.
\begin{example}\label{ex:EX1} For $G=\SO(3,1)_0\subset \GL(4,\R)$, the maximal compact subgroup $K\subset G$ can be chosen as
\[
K=\left\{\begin{pmatrix}R_1(\varphi_1)R_2(\varphi_2)R_1(\varphi_3) & 0 \\0 & 1\end{pmatrix}: \varphi_1,\varphi_2,\varphi_3\in \R\right\}\cong \SO(3),
\]
\[R_1(\varphi)=\begin{pmatrix}\cos(\varphi) & \sin(\varphi)  & 0 \\
 -\sin(\varphi) &\cos(\varphi) & 0 \\
0 & 0 & 1 \\
\end{pmatrix},\qquad R_2(\varphi)=\begin{pmatrix}1 & 0 & 0\\
0 &\cos(\varphi) & \sin(\varphi) \\
0& -\sin(\varphi) & \cos(\varphi)
\end{pmatrix}.
\]
We then get
\begin{align*}
M&=\left\{\begin{pmatrix}R_1(\varphi_1) & 0 \\0 & 1\end{pmatrix}: \varphi_1\in \R\right\}\cong \SO(2),\\
M'&=\left\{\begin{pmatrix}R_2(s\pi)R_1(\varphi_1) & 0 \\0 & 1\end{pmatrix}: \varphi_1\in \R,\; s\in \{0,1\}\right\}\cong \Z_2.
\end{align*}
As a representant $m'_{w_0}$ of the non-trivial Weyl group element $w_0$ we can take
\bq
m'_{w_0}:=\begin{pmatrix}R_2(\pi) & 0 \\0 & 1\end{pmatrix}
.\label{eq:w0EX}
\eq
\subsection{Choosing compatible irreducible representations}
Let us now choose irreducible complex representations $\sigma$ of $K\cong\SO(3)$ and $\tau$ of $M\cong \SO(2)$ which are compatible in the sense that $[\tau]$ occurs in $\sigma|_M$. Canonical candidates for $\sigma$ are the classical irreducible representations of $\SO(3)$ on spaces of homogeneous harmonic polynomials, or, equivalently, spherical harmonics. By \cite[Thm. 3.1]{helgason84}, the complex vector space $H(k,n)$ of homogeneous harmonic polynomials of degree $k$ on $\R^n$, $n\geq 2$, is given by
\bq
H(k,n)=\mathrm{span}_\C\big\{(c_1x_1+\cdots+c_nx_n)^k: c^2_1+\cdots + c_n^2=0\big\},\label{eq:Hkn}
\eq
where $x_i$ is the $i$-th standard coordinate function. The left regular action of $\SO(3)$ on $H(k,3)$ is irreducible for every $k$. To keep this example simple, we shall choose $k=1$ and put
\bqn
W:=H(1,3)=\mathrm{span}_\C\big\{x_1+ix_2,x_1-ix_2, x_2+ix_3\big\}=\mathrm{span}_\C\big\{x_1,x_2,x_3\big\}.
\eqn
Writing
\[
h_1(x):=x_1,\qquad h_2(x):=x_2,\qquad h_3(x):=x_3,
\]
we identify $W$ with $\C^3$ using the basis $\{h_1,h_2,h_3\}$. Accordingly, we identify $\mathrm{End}(W)$ with the space of complex $3\times 3$-matrices. Then  one has
\begin{align*}
\sigma(R_1(\varphi))&=\begin{pmatrix}\cos(\varphi) & \sin(\varphi)  & 0 \\
 -\sin(\varphi) &\cos(\varphi) & 0 \\
0 & 0 & 1 \\
\end{pmatrix}=R_1(\varphi),\qquad 
\sigma(R_2(\varphi))=\begin{pmatrix}1 & 0 & 0\\
0 &\cos(\varphi) & \sin(\varphi) \\
0& -\sin(\varphi) & \cos(\varphi)
\end{pmatrix}=R_2(\varphi),\end{align*}
that is, $\sigma$ is simply the defining representation of $\SO(3)$ on $\C^3$. The next step is to find out how the  restriction to $M\cong \SO(2)$  of this representation decomposes into irreducibles. This is an easy task: The one-dimensional complex vector spaces spanned by
\[
e_1:=(1, i,0),\qquad e_{-1}:=(1,-i,0),\qquad e_0:=(0,0,1)
\]
are irreducible $\SO(2)$-submodules with characters $e^{i\varphi},e^{-i\varphi},1$ of the $\SO(2)$-representation on $W=\R^3$ realized by the  matrices $\sigma(R_1(\varphi))$:
\[
\sigma(R_1(\varphi))e_s=e^{si\varphi}e_s,\qquad s\in \{-1,0,1\}.
\]
As the subgroup $M$ is generated by the rotations $R_1(\varphi)$, we find
\[
\sigma|_M=\tau_{-1}\oplus \tau_0\oplus \tau_1,
\]
where for $s\in \{-1,0,1\}$ the $M$-representation $\tau_s$ acts on the $1$-dimensional complex vector space
\bq
V_s:=\mathrm{span}_\C(e_s)\subset \C^3=W\label{eq:defVs}
\eq
by
\[
\tau_s(R_1(\varphi)):=\sigma(R_1(\varphi))|_{V_s}: V_s\to V_s,
\]
which simply means
\bq
\tau_s(R_1(\varphi))v=e^{is\varphi}v\qquad \forall\; v\in V_s,\;\varphi\in \R.\label{eq:characters}
\eq
Having chosen the irreducible representations $\sigma$ and $\tau_s$, we obtain the associated vector bundles
\[
\W_\sigma=\SO(3,1)_0\times_{\sigma} W,\qquad \V_{\tau_s}=\SO(3,1)_0\times_{\tau_s}V_s,\quad s\in \{-1,0,1\}.
\]
$\W_\sigma$ is a complex vector bundle of rank $3$ over the $3$-dimensional hyperbolic space\footnote{equipped with a Riemannian metric that differs from the usual one by a factor of $\frac{1}{4}$, compare Remark \ref{rem:compare}.} 
\[
\mathbb{H}^3=\SO(3,1)_0/\SO(3),
\]
and, according to Remark \ref{rem:geodvf}, the $\V_{\tau_s}$ can be regarded as vector bundles over the cosphere bundle
\[
S^\ast \mathbb{H}^3=\SO(3,1)_0/\SO(2).
\]
They are complex vector bundles of rank $1$, that is, complex line bundles.
\subsubsection{Checking Assumptions \ref{ass:1} and \ref{ass:2}}
With \eqref{eq:w0EX}, one easily checks that the Weyl group acts on the equivalence classes of the representations $\tau_s$ according to
\bq
w_0[\tau_{s}]=[\tau_{-s}],\qquad s\in\{ -1,0,1\},\label{eq:Wtau}
\eq
which means that neither $[\tau_1]$ nor $[\tau_{-1}]$ is invariant under the Weyl group. In particular, these equivalence classes violate Assumption \ref{ass:2}, while $[\tau_0]$ fulfills Assumption \ref{ass:2}. However, as $\tau_1$ and $\tau_{-1}$ are non-trivial and non-isomorphic, the pairs $\sigma,\tau_s$ fulfill Assumption \ref{ass:1} for each $s\in \{-1,0,1\}$.

\subsection{The algebra of invariant differential operators}\label{eq:algebraex}
In \cite[Folg.\ 2.5]{olbrichdiss}, Olbrich describes a method to determine generators of the algebra $D(G,\sigma)$ of invariant differential operators. Applying his method, one finds that $D(G,\sigma)$ is in the case of this example generated by the Bochner Laplacian $\Delta$ and an additional differential operator $D_1$ of oder $1$ which can be characterized as follows.
\begin{itemize}[leftmargin=*]
\item If we identify $W$ with $\p_\C$ via the bases $h_1,h_2,h_2$ and $v_1,v_2,H_0$, so that 
\bq
\W_\sigma\cong G\times_{\Ad\, K}\p_\C\cong T_\C(G/K)\cong T_\C\mathbb{H}^3,\label{eq:1ident}
\eq then under this identification we have
\bq
D_1=\mathrm{curl},\label{eq:D1rot}
\eq
where $\mathrm{curl}: \Gamma^\infty(T_\C\mathbb{H}^3)\to \Gamma^\infty(T_\C\mathbb{H}^3)$ is the familiar curl operator defined by the Riemannian metric, extended to sections of the complexified tangent bundle. 
\item If we identify $W$ with $\p_\C^\ast$ via the bases $h_1,h_2,h_2$ and $v^\ast_1, v^\ast_2, H^\ast_0$, so that 
\bq
\W_\sigma\cong G\times_{\mathrm{Ad}^\ast K}\p_\C^\ast\cong T^\ast_\C(G/K)\cong T_\C^\ast \mathbb{H}^3,\label{eq:2ident}
\eq 
then under this identification we have
\bq
D_1=\ast\, d,\label{eq:D1stard}
\eq
where $d: \Omega^1_\C(\mathbb{H}^3)\to \Omega^2_\C(\mathbb{H}^3)$ is the exterior derivative and $\ast: \Omega^2_\C(\mathbb{H}^3)\to \Omega^1_\C(\mathbb{H}^3)$ is the Hodge star operator induced by the Riemannian metric, both operators understood to be extended to sections of the complexified exterior algebra bundles.
\end{itemize}
The two characterizations of $D_1$ are dual to each other because \emph{curl} corresponds to $\ast\,  d$ under the  isomorphism $T_\C\mathbb{H}^3\cong T_\C^\ast \mathbb{H}^3$ provided by the Riemannian metric. 
\subsection{Eigenvalues and main result}
Our next goal is to determine the eigenvalues $\mu_{[\sigma],[\tau_s],\lambda}(D_1)$ that Corollary \ref{cor:es1} associates to $D_1$ for each $\lambda\in \aL_\C^\ast$, $s\in\{-1,0,1\}$. By Lemma \ref{lem:smbaction}, this reduces to the computation of the polynomial $\mathrm{smb}_I(D_1)$. Carrying out that compuation, one finds
\[
\mu_{[\sigma],[\tau_s], \lambda \nu_0}(D_1)=-s\, i\lambda,\qquad \lambda \in \C,\qquad s\in \{-1,0,1\}.
\]
We can also determine the eigenvalue of $\Delta$ using \eqref{eq:eiglaplaceex} and \eqref{eq:characters}, taking into account that the variable $m'$ in \eqref{eq:eiglaplaceex} corresponds to the variable $s$, while $m=1$. In summary, we obtain for the shifted spectral parameter $\lambda\nu_0 +\rho=(\lambda + \frac{1}{2})\nu_0$ the eigenvalues
\begin{align*}
 \mu_{[\sigma],[\tau_s], \lambda\nu_0 +\rho}(D_1)&=-s\,i\Big(\lambda+\frac{1}{2}\Big),\qquad s\in \{-1,0,1\},\\
 \mu_{[\sigma],[\tau_s], \lambda\nu_0 +\rho}(\Delta)&=\begin{cases}-\lambda(\lambda+1) + \frac{1}{4},\qquad &s\in \{-1,1\},\\
-\lambda(\lambda+1) + \frac{1}{2}, & s=0, \end{cases}\qquad \lambda \in \C.\end{align*}
For $s\in \{-1,0,1\}$, let $h^s:V_s\to W$ be a non-zero $M$-equivariant homomorphism, for example\footnote{Any other $M$-equivariant homomorphism $V_s \to W$ is just a scalar multiple of the inclusion.}  the inclusion $\iota^s:V_s\hookrightarrow W$. Passing to the quotient by a cocompact torsion-free discrete subgroup $\Gamma\subset G$, Theorems \ref{thm:main} and \ref{thm:rough} now tell us that for each $\lambda\in \C$ the natural pushforward ${{^\Gamma}h^s_{\Pi}}_\ast:\D'(\gam G/M,\gam\V_{\tau_s})\to\D'(\gam G/K,\gam \W_\sigma)$ 
bi-restricts to a linear map
\bqn
{{^\Gamma}h^s_{\Pi}}_\ast: 
\mathrm{Res}^{0}_{\,\gam\V_{\tau_s}}(\lambda)\to E_{[\tau_s],\lambda+ \frac{1}{2}}(\Gamma\backslash\W_\sigma)
=\mathrm{Eig}\Big(\Delta_\Gamma,-\lambda(\lambda+1) + \frac{1}{2+2|s|}\Big)\cap \mathrm{Eig}\Big((D_1)_\Gamma,-s\,i\Big(\lambda+\frac{1}{2}\Big)\Big)
\eqn
which is bijective for all $\lambda\in \C$ outside a discrete subset of $\R$. \end{example}

\section{Distributional sections and their wave front sets}\label{sec:WF}

The availability of the smooth volume density $d(gM)$ on $G/M$ enables us to describe the spaces of distributional sections $\D'(G/M,\V_{\tau})$ as\footnote{Compare \cite[Section 6.3]{hoermanderI}.}
\bq
\D'(G/M,\V_{\tau})=\CT(G/M,\V^\ast_{\tau})^\ast=\{L:\CT(G/M,\V^\ast_\tau)\to \C\text{ linear and continuous}\}.
\eq
The inclusion $\Gamma^\infty(\V_\tau)\hookrightarrow\D'(G/M,\V_{\tau})$ is given by
\bq
s\longmapsto \Big(t\mapsto \int_{G/M}\left<t(gM),s(gM)\right>\, d(gM)\Big).\label{eq:embedsmooth}
\eq
Here we introduced the notation $\left<t(gM),s(gM)\right>:=t(gM)(s(gM))$. Due to the situation described in Remark \ref{rem:dual} we can give a slightly more explicit  description of $\D'(G/M,\V_\tau)$. Let $\tau^\ast: M\to \mathrm{End}(V^\ast)$ be the dual $M$-representation induced by $\tau$. Then the associated vector bundle $$\V_{\tau^\ast}=(G\times_{\tau^\ast} V^\ast,\pi_{\V_{\tau^\ast}}),\qquad \pi_{\V_{\tau^\ast}}([g,v])=gM$$ is canonically isomorphic to the dual vector bundle $\V_\tau^\ast$, and we will identify $\V_\tau^\ast$ with $\V_{\tau^\ast}$ in the following\footnote{Since $M$ is compact, we can choose an $M$-invariant inner product on $V$ which makes the representations $\tau$ and $\tau^\ast$, and consequently the bundles $\V_\tau$ and $\V_{\tau^\ast}$, isomorphic. However, such an isomorphism is not canonical as it depends on the choice of the inner product.}. This is convenient because it enables us to use all previously defined constructions with $\V_\tau$ replaced by $\V_{\tau^\ast}=\V^\ast_{\tau}$. In particular, we have a connection $\nabla^\ast:\Gamma^\infty(\V^\ast_{\tau})\to \Gamma^\infty(\V^\ast_{\tau}\otimes T^\ast(G/M))$.

Let us now describe the topology on $\CT(G/M,\V^\ast_{\tau})$ used in the definition of $\D'(G/M,\V_{\tau})$. 
For $U\subset G/M$ with compact closure $\mathcal{K}:=\overline U$  consider the Fr\'echet topology on $\Cinft(U,{\V^\ast_\tau}|_U)$ defined by the seminorms
\bq
\sup_{g\in \mathcal{K}}\norm{\nabla^\ast_{\X_{\alpha_1}}\cdots \nabla^\ast_{\X_{\alpha_{|\alpha|}}}t(g)},\qquad t\in \Cinft(\widetilde U,V^\ast),\;\X_{\alpha_j}\in \Gamma^\infty(TU),\;\alpha \text{ multi-index},\label{eq:seminorm}
\eq
where $\Vert\cdot\Vert$ is an arbitrary norm on $V^\ast$, and we identified sections of  ${\V^\ast_\tau}|_{U}$ with right-$M$-equivariant (w.r.t.\ $\tau^\ast$) functions  $\widetilde U\to V^\ast$, where $\widetilde U\subset G$ fulfills $\widetilde U/M=U$. The locally convex topology on $\Gamma_c^\infty(\V^\ast_\tau)$ is then defined as the inductive or direct limit topology with respect to any directed exhaustion $G/M=\bigcup_{i\in I}U_i$, where $U_i\subset G/M$ is open with compact closure and $U_i\subset U_j$ if $i\leq j$ for $i,j\in I$, $I$ being a directed set. This topology is characterized by the property that a sequence $\{t_n\}_{n\in \N}\subset \Gamma^\infty_c(\V^\ast_{\tau})$ converges to an element $t\in \Gamma^\infty_c(\V^\ast_{\tau})$ iff there is a compact set $\mathcal{K}\subset G/M$ such that $\supp t_n\subset \mathcal{K}$ for almost all $n\in \N$ and for each multiindex $\alpha$, the corresponding sequence of seminorms as in (\ref{eq:seminorm}) of $t_n-t$ converges to $0$ as $n\to \infty$.
\subsection{Definition of wave front sets of distributional sections} Let $\{U_\kappa\}_{\kappa\in\mathscr{K}}$, $U_\kappa\subset G/M$ open, be a cover of $G/M$ such that ${\V^\ast_\tau}$ is trivial over each $U_\kappa$. Denote the local trivializations by $\zeta_\kappa: {\V^\ast_\tau}|_{U_\kappa}\stackrel{\cong}{\to}U_\kappa\times V^\ast$. Then a distributional section $L\in \D'(G/M,\V_{\tau})$ defines a family of distributions $L_\kappa:\CT(U_\kappa,V^\ast)\to \C$ according to $$L_\kappa(f)=L( {\zeta_\kappa}^{-1}\circ (\mathrm{id}_{U_\kappa}\oplus f)),\qquad f\in \CT(U_\kappa,V^\ast),$$
where ${\zeta_\kappa}^{-1}\circ (\mathrm{id}_{U_\kappa}\oplus f)\in \CT(U_\kappa,{\V^\ast_\tau}|_{U_\kappa})$ is extended by $0$ to be considered an element of $\Gamma_c^\infty(\V^\ast_\tau)$. The family $\{L_\kappa\}$ has for $\kappa,\kappa'\in\mathscr{K}$ the compatibility property
\[
L_{\kappa'}=(\zeta_\kappa\circ{\zeta_{\kappa'}}^{-1})^\ast L_\kappa\qquad \text{on }\,U_\kappa\cap U_{\kappa'}\times V^\ast.
\]
Conversely, for any family $\{L_\kappa\}$ of distributions $L_\kappa: \CT(U_\kappa,V^\ast)\to \C$ with this compatibility property there is a distributional section $L\in\D'(G/M,{\V_\tau})$ inducing the family $\{L_\kappa\}$, compare \cite[Theorem 6.3.4]{hoermanderI}. Thus, we can identify an element $L\in\D'(G/M,{\V_\tau})$ with a family $L:=\{L_\kappa\}$ of distributions $L_\kappa: \CT(U_\kappa,V^\ast)\to \C$. Now, choose a basis of $V^\ast$. 
Then the isomorphism $V^\ast\cong \C^{\dim V^\ast}$ induces a homeomorphism $\CT(U_\kappa,V^\ast)\cong\CT(U_\kappa,\C)^{\dim V^\ast}$ which in turn induces an isomorphism of vector spaces
\[
\{L_\kappa:\CT(U_\kappa,V^\ast)\to \C\; \text{linear and continuous}\}\cong \D'(U_\kappa)^{\dim V^\ast},
\]
where now $\D'(U_\kappa)$ is the space of distributions on the open set $U_\kappa\subset G/M$. Thus, any member $L_\kappa$ of the family $L$ splits as $L_\kappa=L_\kappa^1\oplus\cdots\oplus L_\kappa^{\dim V^\ast}$. The wave front set of $L_\kappa$ is then given by
\[
\mathrm{WF}(L_\kappa)=\bigcup_{j=1}^{\dim V^\ast}\mathrm{WF}(L^j_\kappa)\subset T^*U_\kappa\setminus\{0\},
\]
where $\mathrm{WF}(L^j_\kappa)\subset T^*U_\kappa\setminus\{0\}$ is the usual wave front set of a distribution, see \cite[Section 8.1]{hoermanderI}. Since a change of basis in $V^\ast$ and the transition maps from one trivializing neighborhood $U_\kappa$ to another neighborhood $U_{\kappa'}$ are both diffeomorphisms, the wave front set of $L_\kappa$ is unchanged when pulling back $L_\kappa$ along them (on a suitable restriction of the domain). Thus, one has
$$\mathrm{WF}(L_\kappa)|_{U_\kappa\cap U_{\kappa'}}=\mathrm{WF}(L_{\kappa'})|_{U_\kappa\cap U_{\kappa'}}\qquad \forall\; \kappa,\kappa'\in \mathscr{K},$$
and all the wave front sets of the members $L_\kappa$ of the family $L$ glue together to the wave front set of the distributional section $L\in \D'(G/M,\V_{\tau})$ given by
\[
\mathrm{WF}(L)=\bigcup_{\kappa\in \mathscr{K}}\mathrm{WF}(L_\kappa)\subset T^*(G/M)\setminus\{0\},
\]
which is independent of the choices of the trivializing cover $\{U_\kappa\}$ and the basis of $V^\ast$, see \cite[p.\ 265]{hoermanderI}. 
\newcommand{\etalchar}[1]{$^{#1}$}
\providecommand{\bysame}{\leavevmode\hbox to3em{\hrulefill}\thinspace}
\providecommand{\MR}{\relax\ifhmode\unskip\space\fi MR }
\providecommand{\MRhref}[2]{%
  \href{http://www.ams.org/mathscinet-getitem?mr=#1}{#2}
}
\providecommand{\href}[2]{#2}

\end{document}